\documentclass[11pt]{amsart}
\usepackage{url}
\usepackage{hyperref}
\usepackage{enumerate}
\usepackage{comment}

\makeatletter

\@namedef{subjclassname@2010}{

  \textup{2010} Mathematics Subject Classification}

\usepackage{amssymb, amsmath,amsthm}
\usepackage{mathrsfs}
\newtheorem{thm}{Theorem}[section]

\newtheorem{prop}[thm]{Proposition}
\newtheorem{conj}[thm]{Conjecture}

\newtheorem{cor}[thm]{Corollary}

\newtheorem{lem}[thm]{Lemma}
\theoremstyle{definition}
\newtheorem{rem}{Remark}

\newcommand{\ra}{\rightarrow}
\newcommand{\bk}{\backslash}
\newcommand{\mc}{\mathcal}
\newcommand{\mf}{\mathfrak}
\newcommand{\mb}{\mathbb}
\newcommand{\sg}{\sigma}

\renewcommand{\ss}{\substack}

\newcommand{\e}{\varepsilon}

\newcommand{\mbf}{\boldsymbol}

\begin{document}
\title{A Hal\'{a}sz-type asymptotic formula for logarithmic means and its consequences}
\author{Oleksiy Klurman}
\address{School of Mathematics, University of Bristol, Woodland Road, Bristol, BS8 1UG, UK}
\email{lklurman@gmail.com}
\author{Alexander P. Mangerel}
\address{Department of Mathematical Sciences, Durham University, Stockton Road, Durham, DH1 3LE, UK}
\email{smangerel@gmail.com}
\begin{abstract}
We establish an asymptotic formula for the logarithmic mean value of a 1-bounded multiplicative function that is sharp in many cases of interest. We derive from it a variety of applications, making progress on several old problems. \\
As a first application, we show that if $f$ is a completely multiplicative function taking values in $[-1,1]$ then there is a constant $c > 0$ such that for every $x \geq 3$,
$$
L_f(x) := \sum_{n \leq x} \frac{f(n)}{n} > -\frac{c}{(\log x)^{1-2/\pi}},
$$
thus significantly improving on a 20-year-old result of Granville and Soundararajan. We also show that the exponent of $\log x$ in this result can be improved to $-1+o(1)$, as long as $f$ does not ``behave like'' the Liouville function $\lambda$ in a precise sense. \\
As a second application, we show that for a Rademacher random completely multiplicative function $\mbf{f}$, the probability that $L_{\mbf{f}}(x)$ is negative is $O(\exp(-x^c))$ for some $c \in (0,1)$, thus establishing a previously conjectured bound. \\
Finally, we obtain a converse theorem for small absolute values $|L_f(x)|$, and construct examples $f$ that show that it is (essentially) best possible. 
\end{abstract}
\maketitle
\section{Introduction}
\subsection{Motivation} 
Let $f:\mathbb{N}\to\mathbb{U}$ be a multiplicative function taking values in the closed unit disc $\mathbb{U}$. It is a central problem in analytic number theory to understand the behaviour of the (Ces\`{a}ro) partial sums
\[M_f(x) := \sum_{n\le x}f(n), \quad x \geq 1.\]
Building on the work of Delange, Wirsing and others, Hal\'asz \cite{Hal71} proved a general result that provides upper bounds for such partial sums in broad generality (see \cite[Sec. III.4.3]{Ten} for a comprehensive discussion). A refinement of Hal\'{a}sz' bound, due to Granville and Soundararajan \cite[Cor. 1]{GSDecay}, states that for any $x \geq 3$ and $T \geq 1$,
\begin{equation}\label{eq:HalGS}
M_f(x) \ll x(1+D_f(x;T)) e^{-D_f(x;T)} + \frac{x}{T},
%+ \frac{x\log\log x}{\log x},
\end{equation}
wherein we have set
$$
D_f(x;T) := \min_{|t| \leq T} \mb{D}(f,n^{it};x)^2 := \min_{|t| \leq T} \sum_{p \leq x} \frac{1-\text{Re}(f(p)p^{-it})}{p}.
$$
Hal\'{a}sz' bound and its refinements have been extremely influential in multiplicative number theory, due to their numerous applications. \\
Rather than upper bounds that are not necessarily sharp in every case, one might hope for an asymptotic formula for $M_f(x)$ under general conditions. While some examples of this exist in the literature (including in Hal\'{a}sz' own paper \cite{Hal68}; see also the comprehensive bibliography given in \cite{GHS}), they tend to require rather rigid distributional information on the behaviour of $f$ along the primes. While various extensions, generalisations and new treatments of Hal\'{a}sz' theorem have been given over the years (see e.g., \cite{GHS} and \cite{TenMV} for two significant recent examples), it remains an interesting problem\footnote{In \cite{GHS}, a reference is made to a forthcoming paper of those authors on precisely this question, though as far as we are aware this has not appeared in the literature since then.} to obtain a general asymptotic version of Hal\'{a}sz' theorem. \\ 
In this paper, we look at the corresponding problem for \emph{logarithmic} partial sums 
\[L_f(x) := \sum_{n\le x}\frac{f(n)}{n}, \quad x \geq 1.\]
These have been studied extensively in their own right over the past century, in part due to their intimate connection to values $L(1,\chi)$ of Dirichlet $L$-functions. Unlike $M_f(x)$, which may have large absolute value if $f(n)$ ``pretends'' to be $n^{it}$ for some $t \in \mb{R}$ (i.e., as a function of $x$, $\mb{D}(f,n^{it};x) = O(1)$), $|L_f(x)|$ is not large in this case \emph{except} when $|t|$ is quite small; this reflects the fact that %$L_{n^{it}}(x)$ is an approximation of 
$$
L_{n^{it}}(x) \approx \zeta(1+1/\log x + it) \asymp \min\left\{\log x, \frac{1}{|t|}\right\}, \text{ for } |t| \leq 1, 
$$
whereas when $|t| \geq 1$ we instead have 
$$
|\zeta(1+1/\log x + it)| \ll \log(2+|t|).
$$ 
Thus, while bounds for $L_f(x)$ may be obtained by partial summation from corresponding bounds for $M_f(x)$ as in Hal\'{a}sz' theorem (as Montgomery and Vaughan obtained in \cite{MVMV}, and which Goldmakher refined later in \cite{Gold}), this treatment does not capture the essence of the problem of bounding $|L_f(x)|$. In later works \cite{LamMan} and \cite{GraMan}, more precise upper bounds were obtained that addressed the phenomenon that $|L_f(x)|$ can only be large when $f$ ``pretends'' to be $n^{it}$ for rather small $|t|$. However, both of these results have the drawback that they are far from sharp as soon as $|L_f(x)| = O(1)$ (which is the case for many interesting examples, including the M\"{o}bius function $\mu$ and variants thereof). \\
To see why this is an unfortunate situation, it suffices to note the completely elementary fact that when $g := 1\ast f$,
\begin{equation}\label{eq:trivial}
\frac{1}{x}M_g(x) = \frac{1}{x}\sum_{n \leq x} f(n) \left\lfloor \frac{x}{n}\right\rfloor 
%= x L_f(x) - \sum_{n \leq x} f(n) \{x/n\} 
= L_f(x) + O(1),
\end{equation}
and so in particular one can cheaply obtain an \emph{asymptotic formula} for $L_f(x)$ up to $O(1)$ error. This, of course, shifts the problem to understanding $M_g(x)$, but when e.g. $f$ is a completely multiplicative function that takes real values in $[-1,1]$, the function $g$ is non-negative. Further tools to analyse $M_g(x)$ then become available in such a case. \\ 
%Hal\'{a}sz' upper bound may be used (via partial summation) to prove corresponding estimates for $|L_f(x)|$ that depend also on $M(x;T)$
Our primary goal in this paper is to refine the estimate \eqref{eq:trivial} and establish a much stronger, often \emph{sharp} asymptotic formula for logarithmic sums $L_f(x)$. We will apply it to make progress on several open questions for both deterministic and random multiplicative functions, in particular in the case when $f$ is real-valued and completely multiplicative.

\subsection{Main results}
\noindent For an arithmetic function $h : \mb{N} \ra \mb{C}$ and $x \geq 1,$ we recall and introduce the following notation:
$$
L_h(x):= \sum_{n \leq x} \frac{h(n)}{n}, \quad M_h(x) := \sum_{n \leq x} h(n), \quad \tilde{M}_h(x) := \frac{1}{x} M_h(x).
$$
Throughout this paper we write $w_0 := e^{1-\gamma} ,$ where $\gamma$ is the Euler-Mascheroni constant. 
Our main result
%, concerning $L_f(x)$ when $f$ is bounded and real-valued, 
is the following.
\begin{thm} \label{thm:HalDeltaRefined}
Let $f: \mb{N} \ra [-1,1]$ be a 
%completely 
multiplicative function and let $g := 1\ast f$. Let $x \geq 3$ and\footnote{This condition on $T$ may be relaxed considerably, though it does not create a significant constraint in any of our applications.} $1 \leq T\leq (\log x)^{100}$.  Then 
\begin{align}\label{eq:LftoMg}
L_f(x) &= \left(1+O\left(\frac{1}{\log x} \right)\right) \tilde{M}_g(w_0x) 
%\nonumber \\
+ O\left(\frac{1}{\log x}\int_1^x\left(H_1(y) + H_2(y;T)\right) \frac{dy}{y\log y} + \frac{\log (2T)}{T}\right),
%\left(1 + \frac{(\log (2T)) \log\log x}{\log x} \right)\right),
 %O\left(\frac{\log\log x}{\log x}\left(e^{-M(x;T)} + \frac{\log (2T)}{T}\left(1 + \frac{(\log (2T)) \log\log x}{\log x} \right) \right),
\end{align}
where for $1 \leq y \leq x$ we have set
\begin{align*}
H_1(y) &:= \max_{|t|\leq 1/2} \exp\left(\sum_{p \leq y} \frac{(f(p)-1) \cos(t\log p)}{p}\right), \\
H_2(y;T) &:= \left(\sum_{1 \leq k \leq T-1/2} \frac{(\log(2k))^4}{k^2} \max_{|t-k| \leq 1/2} \exp\left(2\sum_{p \leq y} \frac{f(p)\cos(t\log p)}{p}\right)\right)^{1/2}.
%M(x;T) := \min_{|t| \leq T} \sum_{p \leq x} \frac{(1-f(p)) \cos(t\log p)}{p}.
\end{align*}
Moreover, if there is $c \in (0,1)$ such that $f(p) = 0$ for all $x^{1-c} < p \leq x$ then the integral in the error term of \eqref{eq:LftoMg} may be restricted to the interval $[x^c/2,x]$. 
\end{thm}
The new feature that the logarithmic partial sums up to $x$ depend on the integers $>x,$ that is, the appearance of the constant $w_0>1$, is crucial in these estimates and may be surprising to readers. For an explanation of why it appears, see Section \ref{sec:w0} below. \\
As a consequence, we obtain the following variant of Theorem \ref{thm:HalDeltaRefined} that, while weaker, is slightly simpler to state.
%(and superficially similar to the bound in \eqref{eq:HalGS}).
%, follows as a consequence.
\begin{thm}\label{thm:HalDelta}
Let $f: \mb{N} \ra [-1,1]$ be a 
%completely 
multiplicative function and let $g := 1 \ast f$. Let $x \geq 3$, and let $1 \leq T\leq (\log x)^{100}$.  Then 
$$
L_f(x) = \left(1+O\left(\frac{1}{\log x} \right)\right) \tilde{M}_g(w_0x) + O\left(\frac{\log (2T)}{T} + \frac{\log\log x}{\log x}e^{-M(x;T)}\right),
%+ \frac{1}{\log x}\right),
$$
where we have set
$$
M(x;T) := \min_{|t| \leq T} \sum_{p \leq x} \frac{(1-f(p)) \cos(t\log p)}{p}.
$$
\end{thm}
It may be difficult to appreciate whether Theorem \ref{thm:HalDelta} is genuinely an improvement over \eqref{eq:trivial}. Indeed, unlike the (minimal) ``pretentious distance'' $D_f(x;T)$ appearing in \eqref{eq:HalGS} above, the sum $M(x;T)$ may be \emph{negative}, and consequently the second error term in the above equation is not necessarily as small as $\tfrac{\log\log x}{\log x}$. The following corollary shows, however, that this error term can never be too large. 
%As a corollary, we will derive the following consequence.
\begin{cor} \label{cor:LipError}
Let $f: \mb{N} \ra [-1,1]$ be a multiplicative function and let $g := 1 \ast f$. Then for $x \geq 3$, we have
$$
L_f(x) = \left(1+O\left(\frac{1}{\log x}\right)\right)\tilde{M}_g(w_0x) + O\left(\frac{1}{(\log x)^{1-2/\pi}}\right).
$$
\end{cor}
We also prove that this estimate is best possible. 
\begin{thm}\label{thm:CorOpt}
If $\psi(x) \ra 0$ monotonically then there exists a multiplicative function $f:\mb{N} \ra [-1,1]$ such that
%$$
%\tilde{M}_g(w_0x) \ll x,
%$$
%and also
$$
\max_{X < x \leq X^2} \left|L_f(x) - \left(1+O\left(\frac{1}{\log x}\right)\right)\tilde{M}_g(w_0x)\right| \gg \frac{\psi(X)}{(\log X)^{1-2/\pi}}.
$$
\end{thm}
As mentioned earlier, one should think of $\tilde{M}_g(w_0x)$ in the results above as the main term, though one may construct many examples in which the error term of Corollary \ref{cor:LipError} will be larger than $\tilde{M}_g(w_0x)$. We remark, however, that the shape of the more general error term in Theorem \ref{thm:HalDeltaRefined} is beneficial in such cases, and will be crucial for our forthcoming applications. 
\subsubsection{Results for general $1$-bounded multiplicative functions}
Our results extend to more general $1$-bounded multiplicative functions, albeit in a slightly weaker form.
%In this case we obtain the following slightly weaker result.
\begin{thm}\label{thm:HalDeltaGen}
Let $f: \mb{N} \ra \mb{U}$ be a multiplicative function. Let $x \geq 3$ and $1 \leq T \leq (\log x)^{100}$. Then
\begin{align*}
L_f(x) &= \left(1 + O\left(\frac{(\log\log x)^2}{\log x}\right)\right) \tilde{M}_g(w_0x) \\
&+ O\left(\frac{1}{\log x} \int_1^x(H_1(y) + H_2(y;T)) \frac{dy}{y\log y} + \frac{\log (2T)}{T} +\frac{(\log\log x)^4}{\log x}\right), 
\end{align*}
wherein, analogously to the above, we have
\begin{align*}
H_1(y) &:= \max_{|t| \leq 1/2} \exp\left(\sum_{p \leq y} \frac{\text{Re}((f(p)-1)p^{-it})}{p}\right), \\
H_2(y;T) &:= \left(\sum_{1 \leq k \leq T-1/2} \frac{(\log(2k))^4}{k^2} \max_{|t-k| \leq 1/2} \exp\left(2 \sum_{p \leq y} \frac{\text{Re}(f(p)p^{-it})}{p}\right)\right)^{1/2}.
\end{align*}
\end{thm}
In contrast to the situation for real-valued $f$, in general $g$ is no longer a non-negative function even when $f$ is completely multiplicative. Therefore, the analysis of $\tilde{M}_g(w_0x)$ as a main term is a more subtle problem. While for the purposes of our forthcoming applications we concentrate here on the case of real-valued functions $f$, we hope to return to the analysis of more general bounded $f$ on another occasion.

\section{Applications}
Theorem \ref{thm:HalDeltaRefined} allows us to prove a variety of new results on the extremal behaviour of logarithmic partial sums.
%\noindent In fact, the secondary term is optimal, up to $\log\log x$ powers. [SHOW THIS?] \\
\subsection{On the spectrum of logarithmic partial sums and the negative truncation problem of Granville and Soundararajan}
Our first application is related to obtaining lower bounds on sums of multiplicative functions. In 1994, Heath-Brown conjectured that for any completely multiplicative function $f: \mb{N} \ra [-1,1]$,
\[\sum_{n\le x}f(n)\geq (\delta_1+o(1))x, \quad x \ra \infty,\]
for some $\delta_1 > -1$. Hall \cite{Hall} proved this claim, and conjectured (as did Montgomery, independently) that the best such constant is
\[\delta_1=1-2\log (1+\sqrt{e})+4\int_1^{\sqrt{e}}\frac{\log t}{t+1}dt=-0.656999\dots,\]
for which the inequality is then sharp (up to the $o(x)$ term).
This conjecture was then famously proved by Granville and Soundararajan \cite{SpecGS} a few years later. \\
In 2007, Granville and Soundararajan \cite{GSNeg} raised a similar question about lower bounds for logarithmic partial sums, and showed that, surprisingly the corresponding situation is less clear. They proved that there is a constant $c>0$ such that for any completely multiplicative $f:\mathbb{N}\to [-1,1],$  
\begin{equation}\label{eq:GSLowBd}
L_f(x)\ge -\frac{c}{(\log\log x)^{3/5}}.
\end{equation}
They also constructed an example of such an $f$, taking values $\pm 1$, with 
\begin{equation}\label{eq:GSExample}
L_f(x)<-\frac{C}{\log x},
\end{equation}
for some $C > 0$. Writing
$$
\mc{F} := \{f : \mb{N} \ra [-1,1] : \ f \text{ completely multiplicative}\},
$$
they raised the question of whether one could obtain stronger bounds on the size of 
$$
\delta(x) := \min_{f \in \mc{F}} L_f(x).
$$
Despite multiple efforts from several researchers (private communication), the progress on this problem has been rather modest. Indeed, only the power of $\log\log x$ in \eqref{eq:GSLowBd} has been slightly improved by Kerr and the first author \cite{KeKl}, and very recently by Ter\"av\"ainen and Xu (private communication). \\
An immediate application of our main result gives the following.
\begin{cor}\label{cor:NegTrunc}
Let $f: \mb{N} \ra [-1,1]$ be completely multiplicative. Then there is a constant $c > 0$ such that
$$
L_f(x) \geq - \frac{c}{(\log x)^{1-2/\pi}}.
$$
In particular, we have
$$
-\frac{c}{(\log x)^{1-2/\pi}} \leq \delta(x) \leq -\frac{C}{\log x}.
$$
\end{cor}
%\begin{proof}
% Since $g = 1\ast f \geq 0$, we have $\tilde{M}_g(w_0x) \geq 0$. The claim now follows from Corollary \ref{cor:LipError}.
% \end{proof}
It remains an interesting open problem to determine the optimal exponent of $\log x$ in this problem. On one hand, in view of our Theorem \ref{thm:CorOpt}, the exponent $1-2/\pi$ might not be far from the truth. On the other hand, given the apparent difficulty of generating examples that are substantially different from that of Granville and Soundararajan, one might speculate that $\delta(x)$ should not be much smaller that $-(\log x)^{-1+o(1)}$. The following result lends credence to this latter possibility.
\begin{cor}\label{cor:improveNeg}
Let $f: \mb{N} \ra \{-1,+1\}$ be a completely multiplicative function. Let $x$ be large, and assume that there is a constant $\e > 0$ and $v \geq v_0(\e)$ such that 
\begin{equation}\label{eq:fplus1intro}
\sum_{\ss{x^{1/v} < p \leq x \\ f(p) = +1}} \frac{1}{p} \geq 1+\e.
\end{equation}
Then there are constants $c_1,c_2 > 0$ such that 
%either $L_f(x) \geq 0$ or else
$$
L_f(x) \geq -\frac{c_1}{\log x} \exp\left(c_2(\log\log x)^{2/3} (v \log(v\log\log x))^{1/3}\right).
$$
In particular, if $v \log v = o(\log\log  x)$ then $L_f(x) > -\frac{1}{(\log x)^{1-o(1)}}$. 
\end{cor}
Corollary \ref{cor:improveNeg} follows from the slightly more general Proposition \ref{prop:improveNeg} below, which treats general bounded, real-valued multiplicative functions. \\
It should be mentioned that for ``typical'' functions $f$, \eqref{eq:fplus1intro} holds for rather \emph{small} (even bounded) values of $v$. A notable exception to this is when $f$ ``pretends'' to be the M\"{o}bius function. See Section \ref{sec:Conv} for a related discussion on this issue, and the appearance of the parameter $v$ here.
\subsection{On the ``random'' Tur\'{a}n problem}
Let $\mbf{f}$ denote a Rademacher random completely multiplicative function, i.e., let $(\mbf{f}(p))_p$ be a sequence of i.i.d. Rademacher $\pm 1$-valued random variables, and if $n$ has prime factorisation $n = p_1^{a_1} \cdots p_k^{a_k}$ then set
$$
\mbf{f}(n) := \prod_{1 \leq j \leq k} \mbf{f}(p_j)^{a_j}.
$$ 
Inspired by the so-called ``Tur\'{a}n conjecture'' on negative values of the partial sums $L_{\lambda}(x)$ of the Liouville function $\lambda$ (which was famously disproved by Haselgrove \cite{Has}), it is natural to consider the problem of determining the probability of the event $P_x$ that
$$
\sum_{n \leq x} \frac{\mbf{f}(n)}{n} < 0.
$$
The example of Granville and Soundararajan giving \eqref{eq:GSExample} shows that there is at least one realisation of $\mbf{f}$ for which this event holds, so that $\mb{P}(P_x) \geq 2^{-\pi(x)}$. If we let $p(x) := \log(1/\mb{P}(P_x))$ then this implies that
$$
p(x) \leq (\log 2 + o(1))\frac{x}{\log x}.
$$
Angelo and Xu \cite{AngXu} showed that, remarkably, $\mb{P}(P_x)$ is very small, and thus $p(x)$ is rather large. In fact, they showed that there is a constant $c > 0$ such that
%they showed that
$$
p(x) \geq \exp\left(c\frac{\log x}{\log\log x}\right),
$$
This bound was then slightly improved by Kerr and the first author \cite[Thm. 1.2]{KeKl}, who showed that there is $c > 0$ such that
$$
p(x) \geq \exp\left(c\frac{(\log x)\log\log\log x}{\log \log x}\right).
$$
Very recently, Kucheriaviy \cite[Thm. 2]{Kuch} has shown the stronger bound
$$
p(x) \geq \exp\left((1+o(1)) (\log x) \frac{\log\log\log\log x}{\log\log\log x}\right).
$$
It has been speculated that there is $\beta \in (0,1)$ such that
\begin{equation}\label{eq:pxbeta}
p(x)\ge x^{\beta}.
\end{equation} 
%for some $\beta>0.$
A bound of this quality was obtained \emph{conditionally} on a conjectural improvement of Hal\'{a}sz' theorem in \cite{Kuch} (see Thm. 3 there). While we have been unable to obtain such an improvement, 
%we may still claim t and instead opted for a different route, relying on the sharpness 
%by obtaining a slight variant of Theorem \ref{thm:HalDeltaRefined}, 
we can still prove \eqref{eq:pxbeta} unconditionally using a slight variant of Theorem \ref{thm:HalDeltaRefined} (see Proposition \ref{prop:passToSmooth} below).
\begin{thm}\label{thm:random}
There is a constant $\beta > 0$ such that if $x$ is sufficiently large then
$$
\mb{P}\left(\sum_{n \leq x} \frac{\mbf{f}(n)}{n} < 0\right) \ll \exp(-x^{\beta}).
$$
\end{thm}
\subsection{Converse theorems for small values of $|L_f(x)|$} \label{subsec:Conv}
While the problem of Granville and Soundararajan discussed earlier addresses small \emph{negative} values of $L_f(x)$, one may naturally also consider the question of classifying those real-valued, bounded functions $f$ having small \emph{absolute} values $|L_f(x)|$. In the case of unweighted sums of multiplicative functions, this was resolved in a well-known work of Koukoulopoulos \cite{Kou}.
Our Theorem \ref{thm:HalDeltaRefined} may also be applied to prove the following classification theorem in this direction, suggesting that any such $f$ must ``pretend'' to be the Liouville function $\lambda(n)$.
\begin{cor} \label{cor:convLf}
Let $x$ be large and fix $\e \in (0,1)$ and $C > 0$. Let $f: \mb{N} \ra \{-1,+1\}$ be a completely multiplicative function such that 
\begin{equation}\label{eq:uppBdLfConv}
|L_f(x)| \leq \frac{\exp(C(\log\log x)^{2/3})}{\log x}.
\end{equation}
Suppose \eqref{eq:fplus1intro} above holds for some $v = v(x) = O(1)$, Then
%one of the following holds: \\
%(a) there is $\delta > 0$ such that if $v\log v %\geq \delta^3 \log\log x$, \\
%(b) we have $\mb{D}(f,\lambda;x)^2 \ll \delta %\log\log x.$
%Furthermore, if $v = O(1)$ then we have
$$
\mb{D}(f,\lambda;x)^2 \ll (\log\log x)^{2/3} (\log\log\log x)^{1/3}.
$$
\end{cor}
In fact, we prove a more general result that applies to all bounded, real-valued, completely multiplicative functions, see Proposition \ref{prop:convLf} below. \\
%The bound on $\mb{D}(f,\lambda;x)^2$ just given seems \emph{a priori} unnatural. 
Perhaps surprisingly, the bound on $\mb{D}(f,\lambda;x)^2$ just given is essentially sharp, assuming \eqref{eq:uppBdLfConv}. Indeed, we will construct an example in Section \ref{sec:Optim} to show the following.
%that it is nearly best possible assuming \eqref{eq:uppBdLfConv}.
%, proving the following.
\begin{cor} \label{cor:const}
%\begin{lem}
There is a completely multiplicative function $f: \mb{N} \ra \{-1,+1\}$ such that \eqref{eq:fplus1intro} holds with $\e > 0$ fixed and $v = v(x) = O(1)$, and such that for any $C > 4$ the following bounds hold: 
%$|L_f(x)| = 
\begin{align*}
\mb{D}(f,\lambda;x)^2 &\asymp (\log\log x)^{2/3}, \\
|L_f(x)| &\ll \frac{\exp\left(C(\log\log x)^{2/3}\right)}{\log x}.
\end{align*}
\end{cor}
\subsection{On a conjecture of Goldmakher}
 Our final application is related to the following conjecture of Goldmakher (see Conj. 2.6 of \cite{Gold}).
\begin{conj}[Goldmakher] \label{conj:Gold}
Let $f: \mb{N} \ra \mb{U}$ be a completely multiplicative function. Then for any $1 \leq y \leq x$,
$$
\sum_{\ss{n \leq x \\ P^+(n) \leq y}} \frac{f(n)}{n} \ll (\log y)e^{-\mb{D}(f,1;y)^2} + 1.
$$
\end{conj}
If we restrict ourselves to real-valued functions then Goldmakher's conjecture in this case follows immediately from \eqref{eq:trivial} and the standard bound
$$
M_g(x) \ll \frac{x}{\log x} \exp\left(\sum_{p \leq x} \frac{g(p)}{p}\right),
$$
for non-negative, divisor-bounded multiplicative functions (see Lemma \ref{lem:UppBdNonNeg} below). \\ 
%The corresponding statement follows from Prop. 3.1 of \cite{GSNeg}. 
We show here that a much stronger result holds.
%for general real-valued multiplicative functions.
\begin{cor}\label{cor:Gold}
Let $f: \mb{N} \ra [-1,1]$ be a multiplicative function. If $x$ is sufficiently large then for each $2 \leq y \leq x$ we have
$$
\sum_{\ss{n \leq x \\ P^+(n) \leq y}} \frac{f(n)}{n} \ll (\log y)e^{-\mb{D}(f,1;y)^2} + \frac{1}{(\log x)^{1-2/\pi}}.
$$
\end{cor}
When $f$ is not real-valued, any corresponding estimate relies first on applying the trivial bound $|\tilde{M}_g(w_0x)| \leq \tilde{M}_{|g|}(w_0x)$, which is too weak to obtain the conjectured bound of Goldmakher.   

\section{Proof Ideas}
\subsection{On the proof of Theorem \ref{thm:HalDeltaRefined}} \label{sec:w0}
\subsubsection{Motivation for $w_0$}
The key novelty of our approach is in relating $L_f(x)$ to the mean value of $M_g(w_0x),$ rather than $M_g(x)$ as in all previous works in the subject.
%evaluated at a different scale.  
It is not \emph{a priori} obvious why the precise choice $w_0 = e^{1-\gamma}$ enters into this problem. The following discussion will elucidate this matter. \\
Let $x \geq 3$ and $\sg_0 := 1+1/\log x$. For $\text{Re}(s) > 1$ we write
$$
L(s,f) := \sum_{n \geq 1} \frac{f(n)}{n^s}, \quad L(s,g) := \sum_{n \geq 1} \frac{g(n)}{n^s} = L(s,f) \zeta(s).
$$
Applying Perron's formula, we have
\begin{align*}
L_f(x) &= \frac{1}{2\pi i} \int_{\sg_0 - i\infty}^{\sg_0 + i\infty} L(s,f) \frac{x^{s-1}}{s-1} ds, \\
\tilde{M}_g(w_0 x) &= \frac{1}{2\pi i} \int_{\sg_0-i\infty}^{\sg_0 + i\infty} L(s,f)\zeta(s) \frac{(w_0 x)^{s-1}}{s} ds.
\end{align*}
%owing to $L(s,g) = L(s,f) \zeta(s)$. 
Subtracting the two, we obtain
\begin{align} \label{eq:diffLf}
L_f(x) - \tilde{M}_g(w_0x) = \frac{1}{2\pi i } \int_{\sg_0 - i\infty}^{\sg_0 + i\infty} L(s,f) \left(\frac{sw_0^{1-s}}{s-1} - \zeta(s)\right) \frac{(w_0x)^{s-1}}{s} ds.
\end{align}
The following lemma is a straightforward consequence of the Laurent expansion of $\zeta(s)$ near $s=1$. (For convenience, we have included the proof of \eqref{eq:zetaprime} here as well, as we will need it later.)
\begin{lem}\label{lem:snear1}
Assume that $|s-1| \leq 1/2$. Then
\begin{align}
&\frac{s}{s-1}w_0^{1-s} - \zeta(s) \asymp |s-1|, \label{eq:w0s}\\
&\zeta'(s) + \frac{s}{(s-1)^2} - \frac{\zeta(s)}{s} = O(1). \label{eq:zetaprime}
\end{align}
\end{lem}
\begin{proof}
It is well-known that when $|s-1| \leq 1/2$,
\begin{align*}
\zeta(s) = \frac{1}{s-1} + \gamma + \gamma_1 (1-s) + O(|s-1|^2) \\
\zeta'(s) = -\frac{1}{(s-1)^2} - \gamma_1 + O(|s-1|),
\end{align*}
where $\gamma_1= - 0.0728...$ is the first Stieltjes constant. For the first claim, recalling that $w_0 = e^{1-\gamma}$, we have
$$
\frac{s}{s-1}w_0^{1-s} = \frac{s}{s-1} - s(1-\gamma) - s\frac{(1-\gamma)^2}{2}(1-s) + O(|s-1|^2),
$$
so that combining these expressions yields
\begin{align*}
\frac{s}{s-1}w_0^{1-s} - \zeta(s) &= \left(\frac{s}{s-1} -s(1-\gamma) - s \frac{(1-\gamma)^2}{2} (1-s)\right) - \left(\frac{1}{s-1} +\gamma + \gamma_1(1-s)\right) + O(|s-1|^2) \\
&= (1-s) \left(1-\gamma - \frac{(1-\gamma)^2}{2} - \gamma_1\right)  + O(|s-1|) \asymp |s-1|, 
\end{align*}
as claimed. Similarly, for the second claim we have
\begin{align*}
\zeta'(s) - \frac{\zeta(s)}{s} + \frac{s}{(s-1)^2} = \frac{1}{s-1} - \frac{1}{s(s-1)} -\gamma_1 - \frac{\gamma}{s} + O(|s-1|) = O(1),
\end{align*}
as required.
\end{proof}
As $L(\sg_0 + it,f)$ is likely to have its maximum in $|t| \leq 1$ when $f$ is real-valued, the factor on the LHS of \eqref{eq:w0s} dampens its contribution in \eqref{eq:diffLf}. This is similar in nature to what transpires in the proof of Lipschitz theorems (for all complex-valued, $1$-bounded functions; see e.g. \cite[Sec. 4]{GHS}), wherein one needs to control
$$
\max_{|t| \leq T} \left|\left(1-w^{\sg_0 - 1 + it}\right) L(\sg_0 + it)\right|,
$$
and one obtains savings when $|t| \log w$ is not too large. \\
\subsubsection{Towards Theorem \ref{thm:HalDeltaRefined}}
Given the previous discussion, in order to prove Theorem \ref{thm:HalDeltaRefined} we will follow the classical strategy of ``averages of averages'' proof of Hal\'{a}sz' theorem. The analysis, however, is rather more delicate. \\
First, we make a standard reduction to the case that $f$ is completely multiplicative (see Lemma \ref{lem:redtoCM}), which then allows us to invoke the aforementioned fact that $g = 1\ast f \geq 0$. Setting
$$
\Delta(y) := L_f(y) - \tilde{M}_g(w_0y), \quad 1 \leq y \leq x,
$$
we then aim to bound $|\Delta(x)|$ from above in terms of the logarithmic average
\begin{equation}\label{eq:avgDelta}
\frac{1}{\log x} \int_1^x |\Delta(y)| \frac{dy}{y}.
\end{equation}
Our result in this direction (see Proposition \ref{prop:toAvg}) delivers the (essentially lossless) estimate
$$
|\Delta(x)| \ll \frac{1}{\log x} \int_1^x |\Delta(y)| \frac{dy}{y} + \frac{1 + |L_f(x)| + \tilde{M}_g(w_0x)}{\log x}.
$$
%error term when $f$ is real-valued, directly in terms of $|L_f(x)|$ and $\tilde{M}_g(w_0x)$. 
This crucially uses the fact that $g$ is non-negative, though a slightly weaker variant is still available for general 1-bounded multiplicative functions $f$. It is worth noting that this integral can be truncated from below at $y = x^c$ when $f(n)$ is supported on primes $p \leq x^{1-c}$, a feature noticed in a related context in \cite{GHS}. We will discuss the benefits of this observation in connection to Theorem \ref{thm:random}, in the next subsection. \\
The integral \eqref{eq:avgDelta} is then estimated (after applying the Cauchy-Schwarz inequality and Plancherel's theorem) in terms of an $L^2$ integral of Dirichlet series, the key contribution of which being of the shape
$$
\int_{1+\alpha-iT}^{1+\alpha+iT} \left|\frac{L'(s,f)w_0^{s-1}}{s}\left(\frac{sw_0^{1-s}}{s-1} - \zeta(s)\right)\right|^2 ds,
$$
for varying $\tfrac{1}{\log x} \leq \alpha \leq 1$. Each of the ranges $|\text{Im}(s)| \leq 1/2$ and $1/2 \leq |\text{Im}(s)|\leq T$ is controlled by the respective error terms $H_1(e^{1/\alpha})$ and $H_2(e^{1/\alpha};T)$, which are qualitatively different according to the influence that the size of $\zeta(1+\alpha+it)$ has on the integral. The shape of the error  term in Theorem \ref{thm:HalDeltaRefined} corresponds to an average over $\alpha$. \\
In particular, in the range $|\text{Im}(s)| \leq 1/2$ the relationship between $w_0$ and the Laurent expansion of $\zeta(s)$ at $s = 1$ becomes essential, and results in significant savings when $f$ is real-valued. Because $\log|\zeta(1+\alpha+it)| \ll \log\log(2+|t|)$ in the range $1/2 \leq |t| \leq T$, it plays only a minor role in estimates for the error terms $H_2(e^{1/\alpha};T)$. 
\subsubsection{Optimal examples} 
The example in Theorem \ref{thm:CorOpt} has been mentioned in previous works on multiplicative functions in relation to the optimality of ``Lipschitz bounds'' that relate $\tilde{M}_f(y)$ to $\tilde{M}_f(y/w)$, when $1 \leq w \leq y^{o(1)}$. In \cite{GHS}, estimates for $\tilde{M}_f(y)$ for such examples were invoked without proof. One motivation for Theorem \ref{thm:CorOpt} is to provide the details of these\footnote{Strictly speaking, our results are tailored to logarithmic mean values but the same methods can be used to analyse Ces\`{a}ro mean values as well.} estimates. This is done by adapting a method from \cite{GraMan} to study multiplicative functions $f$ that are defined at the primes via
$$
f(p) = h(t \log p),
$$
where $t \in \mb{R}$ and $h$ is a 1-periodic function with sufficient decay in its Fourier coefficients. In the situation of Theorem \ref{thm:CorOpt} our function $h$ is of bounded variation but not sufficiently smooth for a direct application of this result to hold. We bypass these issues by replacing $h$ pointwise by a \emph{smoother} approximation. See Section \ref{sec:optimal} for the details.
\subsection{Towards the main applications}
\subsubsection{On negative values of truncations and converse theorems}  The proof of Corollary \ref{cor:NegTrunc} follows immediately from Theorem \ref{thm:HalDeltaRefined},
once again using that $g = 1\ast f$ is a non-negative function whenever $f(n) \in [-1,1]$ and is completely multiplicative. This fact was already crucially exploited in all of the previous works on the negative truncations problem (following the example of \cite{GSNeg}). \\
More novel in this context is Corollary \ref{cor:improveNeg}, which incorporates the \emph{size as well as sign} of $\tilde{M}_g(w_0x)$ into the problem. Starting from the assumption $L_f(x) < 0$, one quickly obtains from Theorem  \ref{thm:HalDelta} an upper bound condition of the shape
\begin{equation}\label{eq:MgUpp}
\tilde{M}_g(w_0x) \ll \frac{\log\log x}{\log x} \max\{H_1(x), H_2(x;T)\},
\end{equation}
taking $T = (\log x)^2$, say. Both of $H_1(x)$ and $H_2(x;T)$ can be explicitly described in terms of the distribution of values of $f(p)$ with $p \leq x$. To obtain a precise characterisation of the values of $f(p)$, we seek to pair this condition with a corresponding lower bound for $\tilde{M}_g(w_0x)$ in terms of these values $f(p)$ as well. \\
For convenience, let us assume in this discussion that $f(p) \in \{-1,+1\}$ (which is the case in the statement of Corollary \ref{cor:improveNeg}, but not of the more general Proposition \ref{prop:improveNeg}). 
It is clear that when $f(p) = -1$ at all large primes $x^{1/v} < p \leq x$ then $g(p) = 0$ in this range and thus $g$ is (essentially) supported on $x^{1/v}$-friable numbers. For large $v$, one therefore expects $\tilde{M}_g(w_0x)$ to be rather small, and that this phenomenon should persist when $f(p) = -1$ ``most of the time'' (in a precise sense) in this range. Building on the seminal works of Granville, Koukoulopoulos and Matom\"{a}ki \cite{GKM} and Matom\"{a}ki and Shao \cite{MS}, Kerr and the first author \cite{KeKl} showed, however, that if $\e > 0$ is fixed, $y$ is sufficiently large and $v \geq v_0(\e)$ satisfies
\begin{equation}\label{eq:KeKlProofStrat}
\sum_{\ss{y^{1/v} < p \leq y \\ f(p) = +1}} \frac{1}{p} \geq 1 + \e,
\end{equation}
%for some $\e > 0$ and $y$ sufficiently large 
then one can get a precise lower bound
$$
\tilde{M}_g(y) \gg v^{-(1+o(1))v/e} \exp\left(\sum_{p \leq y} \frac{f(p)}{p}\right).
$$
When $v^{v} = (\log y)^{o(1)}$ this provides a sufficient lower bound for the purposes of our classification. \\
The same idea underpins the proof of our converse result, Corollary \ref{cor:convLf}, on small values of $|L_f(x)|$.  In that case, 
%because one can show that $H_1(x) \gg 1$, 
Theorem \ref{thm:HalDelta} can be manipulated to obtain
$$
\tilde{M}_g(w_0x) \ll |L_f(x)| + \frac{\log\log x}{\log x} \max\{H_1(x),H_2(x;T)\}.
$$
Under the assumption that 
$$
|L_f(x)| \ll \frac{\exp(C(\log\log x)^{2/3})}{\log x}
$$
for some $C > 0$, this either implies that \eqref{eq:MgUpp} holds, or else that $\tilde{M}_g(w_0x)$ is directly upper bounded by the same upper bound as $|L_f(x)|$. In either case, a sufficient bound may be obtained.
See Section \ref{sec:Conv} for further details.\\
\subsubsection{On the random problem}
Let $\mbf{f}$ be a Rademacher random completely multiplicative function and let $\mbf{g} := 1 \ast \mbf{f}$. An immediate issue that arises in the proof of Theorem \ref{thm:random} is that the error term in Theorem \ref{thm:HalDeltaRefined} is too weak to improve on previous bounds for $\mb{P}(L_{\mbf{f}}(x) < 0)$.  
%progress on the problem. 
Indeed, with sufficiently high probability one can show that $H_1(y)$ and $H_2(y;T)$ are both of bounded size for $1 \leq y \leq x$, but on invoking the lower bound $M_{\mbf{g}}(w_0 x) \geq 0$ (and integrating in $y$) we are left to seek the probability of the event
$$
L_f(x) \geq -C \frac{\log\log x}{\log x}.
$$
The work of Angelo and Xu \cite{AngXu} already shows that this is bounded above by $\exp(\exp(c\tfrac{\log x}{\log\log x}))$, and that this is sharp up to the constant $c> 0$. \\
To do better, one must remove the $\log\log x$ factor arising in this error term. We do this by employing the following trick: if we define a completely multiplicative function at primes via $\mbf{f}_{1/2}(p) = \mbf{f}(p) 1_{p \leq x^{1/2}}$ then
\begin{equation}\label{eq:multMTs}
L_{\mbf{f}}(x) = \sum_{x^{1/2} < p \leq x} \frac{\mbf{f}(p)}{p} L_{\mbf{f}}(x/p) + L_{\mbf{f}_{1/2}}(x).
\end{equation}
There are two key features of this identity that are crucial to our analysis:
\begin{enumerate}[(i)]
\item we can take complete advantage of the independence of $(\mbf{f}(p))_{x^{1/2} < p \leq x}$ (of which there are many) to force the sum over $p$ in \eqref{eq:multMTs} to be small, conditioning on the primes $p \leq x^{1/2}$ (see Lemma \ref{lem:extraMTs});
\item as $\mbf{f}_{1/2}$ vanishes at all $p > x^{1/2}$, the improved version of Theorem \ref{thm:HalDeltaRefined} may now be applied to the term $L_{\mbf{f}_{1/2}}(x)$.
%(noting that $\mbf{g}_{1/2}(p) := 1\ast \mbf{f}_{1/2}(p) > 0$ for all $p > x^{1/2}$).
\end{enumerate}
Item (ii) saves the necessary $\log\log x$ factor in the error term (since $\int_{x^{1/2}}^x \tfrac{dy}{y\log y} \asymp 1$). \\
In view of item (i) above, ascertaining the probability of $L_{\mbf{f}}(x) < 0$ then boils down to determining the likelihood that
$$
\tilde{M}_{1\ast \mbf{f}_{1/2}}(x) \ll \frac{1}{\log x} \max\{H_1(x), H_2'(x;T)\},
$$
reminiscent of the converse theorems discussed above. Here, $H_2'(x;T)$ is a variant\footnote{The benefit of this lower truncation is that the variation in argument of $\cos(t\log p)$ can be better controlled when $|t-k| \leq 1/2$.} of $H_2(x;T)$ wherein the prime sum is truncated to $p > y_k := \exp((\log (2+k))^2)$ for each term $1 \leq k \leq T$ (this loses only a constant factor in the estimates).  With high probability we can ensure that $\mbf{f}_{1/2}(p) = +1$ sufficiently often among the ``large primes'' $x^c < p \leq x^{1/2}$, for $c > 0$ not too small, that $\tilde{M}_{1\ast \mbf{f}_{1/2}}(x)$ may be effectively bounded below via \eqref{eq:KeKlProofStrat} (with $v = O(1)$) and matters are reduced to comparing the sums
$$
\sum_{p \leq x} \frac{\mbf{f}(p)}{p} \text{ and } \max_{|t-k| \leq 1/2} \sum_{y_k < p \leq x} \frac{(\mbf{f}(p) - 1_{k = 0})\cos(t\log p)}{p},
$$
for each $0 \leq k \leq T$ (writing $y_0 := 1$). By modifying the argument of Angelo and Xu, we give the tail estimates 
$$
\max_{|t - k| \leq 1/2} \mb{P}\left(\sum_{y_k < p \leq x} \frac{f(p)(1-\cos(t\log p))}{p} \geq \log(1/\delta)\right) \ll \exp(\exp(-c/\delta))
$$
whenever $\delta \gg \tfrac{1}{\log x}$, for some $c > 0$ absolute (see Lemma \ref{lem:AngXu}). Since the Euler products vary by $O(1)$ on intervals of length $1/\log x$, a simple discretisation argument (see Lemma \ref{lem:maxTail}) implies that the same tail estimates hold (with a slightly smaller $c > 0$) for 
$$
\mb{P}\left(\max_{|t-k| \leq 1/2} \sum_{y_k < p \leq x} \frac{f(p)(1-\cos(t\log p))}{p}\geq \log(1/\delta)\right).
$$
This ultimately leads to the proof of Theorem \ref{thm:random}. See Section \ref{sec:random} for the details.

\section{Proofs of the main theorems}
%In this section, let $f: \mb{N} \ra \mb{U}$ be a 1-bounded multiplicative function. 
\subsection{Auxiliary bounds}
We will use the following two estimates repeatedly in the sequel.
\begin{lem} \label{lem:UppBdNonNeg}
Let $h$ be a multiplicative function such that\footnote{Here, $d(n)$ denotes the divisor function.} $0 \leq h(n) \leq d(n)$ for all $n\geq 1$. The following results hold: 
\begin{enumerate}[(a)]
\item For $x\geq 3$ we have
$$
\frac{1}{x}\sum_{n \leq x} h(n) \ll \frac{1}{\log x} \sum_{n \leq x} \frac{h(n)}{n} \ll \exp\left(\sum_{p \leq x} \frac{h(p)-1}{p}\right).
$$
\item Fix $\e \in (0,1)$ and suppose that $x \geq x_0(\e)$. Then for every $x^{\e} < y \leq x$,
$$
\sum_{x < n \leq x + y} h(n) \ll \frac{y}{\log x} \exp\left(\sum_{p \leq x} \frac{h(p)}{p}\right).
$$
\end{enumerate}
\end{lem}
\begin{proof}
The first result is due to Halberstam and Richert \cite{HalRic}. The second, which generalises the first, is due to P. Shiu \cite{Shiu}.
\end{proof}
We also record the following simple \emph{Lipschitz-type} estimate for $\tilde{M}_g(x)$, which improves significantly on the general such bounds for divisor-bounded functions (as in \cite[Thm. 1.5]{GHS}).
\begin{lem}\label{lem:LipGHS}
Let $f: \mb{N} \ra \mb{U}$ be a $1$-bounded multiplicative function, and let $g := 1\ast f$. Then for any $w \geq 1$,
$$
\tilde{M}_g(x) - \tilde{M}_g(x/w) \ll \log(2w).
$$
\end{lem}
\begin{proof}
By the triangle inequality,
$$
|L_f(x)-L_f(x/w)| \leq \sum_{x/w < n \leq x} \frac{1}{n} = \log w + O(1).
$$
Now, applying the elementary estimate \eqref{eq:trivial} above, i.e.,
$$
\tilde{M}_g(y) = L_f(y) + O(1),
$$
with $y = x$ and $y = x/w$ and subtracting, we get $$
|\tilde{M}_g(x)-\tilde{M}_g(x/w)| \leq \log w + O(1) \ll \log(2w),
$$
as claimed.
\end{proof}
Finally, we will frequently make use of the following special case of a result due to Hall and Tenenbaum (see \cite[Lem. 30.1]{HTDiv}), which follows from the prime number theorem.
\begin{lem}\label{lem:TenEst}
%the following, which is a special case of Lemma %III.4.13 of \cite{Ten}: \\
Let $t \neq 0$ and $z > w \geq 2$. Let $\phi(u)$ be a $2\pi$-periodic function of bounded variation. Then  
\begin{align} \label{eq:TenEst}
\sum_{w < p \leq z} \frac{\phi(t\log p)}{p} = \left(\frac{1}{2\pi} \int_0^{2\pi} \phi(u) du\right) \log\left(\frac{\log z}{\log w}\right) + O_{\phi}\left(\frac{1}{|t| \log w} + \frac{1+|t|}{\exp(\sqrt{\log w})}\right).
\end{align}
\end{lem}
\subsection{Averaging on scales}
% We use the usual trick of multiplying by $\log x = \log n + \log(x/n)$. Here, unlike in the $1$-bounded case, $g(n)\log(x/n)$ does not have bounded mean value in general. Moreover, the difference between $\log x$ and $\log(w_0x) = \log x + 1-\gamma$ is non-trivial in these arguments, and in fact we will multiply both $L_f(x)$ and $\tilde{M}_g(w_0x)$ by $\log(w_0x)$. As we end up dividing by $\log (w_0x)$ in the end, we can tolerate $O(1)$ error terms throughout. \\
In the sequel, for $y \geq 1$ and $f: \mb{N} \ra \mb{U}$ write
$$
\Delta(y) := L_f(y) - \tilde{M}_g(w_0 y).
$$
We will use the approach of Montgomery and Vaughan in \cite{MVMV}, inspired by Hal\'asz' original proof strategy of taking ``averages of averages" at different scales. Let $x_0$ be a large absolute constant to be chosen later, and select $x$ to belong to the set of local maxima
\begin{equation} \label{eq:mfSDef}
\mathfrak{S} := \{x \geq x_0 : x_0 \leq y \leq x \Rightarrow |\Delta(y)|\log(w_0y) \leq |\Delta(x)|\log (w_0x)\}.
\end{equation}
Our main proposition for this section is the following. 
\begin{prop}\label{prop:toAvg}
Let $f: \mb{N} \ra \mb{U}$ be a multiplicative function. Assume that $x \in \mathfrak{S}$. \\
(a) If $f$ is real-valued and completely multiplicative then
\begin{align*}
\left|\Delta(x)\right| \ll \frac{1}{\log x} \int_1^x |\Delta(y)| \frac{dy}{y} + \frac{1 + |L_f(x)| + \tilde{M}_g(w_0x)}{\log x}.
\end{align*}
Moreover, if there is $1 \leq z \leq x/2$ such that $f(p) = 0$ for all $z < p \leq x$ then
\begin{align*}
\left|\Delta(x)\right| \ll \frac{1}{\log x} \int_{x/(2z)}^x |\Delta(y)| \frac{dy}{y} + \frac{1 + |L_f(x)| + \tilde{M}_g(w_0x)}{\log x}.
\end{align*}
(b) In general, the above estimates hold with
$$
\frac{1+|L_f(x)| + \tilde{M}_g(w_0x)}{\log x} \text{ replaced by } \frac{(\log\log x)^4 + |L_f(x)| + |\tilde{M}_g(w_0x)| (\log\log x)^2}{\log x}.
$$
\end{prop}
\noindent In what follows we will give a unified treatment that applies to all $1$-bounded multiplicative functions, since the relevant arguments only differ in one place (see Lemma \ref{lem:S1Sec} below).  \\
% Finally, the logarithmic mean is also slightly delicate.\\
To proceed, we shall consider $\Delta(x) \log(w_0x)$, and expand it using the well-known decomposition
$$
\log y = \log n + \log(y/n), \quad y \geq n \geq 1,
$$
treating $L_f(x) \log(w_0x)$ and $\tilde{M}_g(w_0x)\log(w_0x)$ separately in the following two lemmas.
\begin{lem} \label{lem:LfDev}
We have
\begin{align}\label{eq:LfDev}
L_f(x)\log(w_0 x) = \sum_{d \leq x} \frac{f(d) \Lambda(d)}{d} L_f(x/d) + L_g(x) + O(|L_f(x)| + 1).
\end{align}
\end{lem}
\begin{proof}
Recalling that $w_0 = e^{1-\gamma}$, we have
\begin{align} \label{eq:Lflog}
L_f(x) \log(w_0x)  = \sum_{n \leq x} \frac{f(n) \log n}{n} + \sum_{n \leq x} \frac{f(n)\log(x/n)}{n} + O(|L_f(x)|).
\end{align}
Observe that the second term on the RHS arises from considering $L_g(x)$ instead: 
\begin{align}\label{eq:Lgarising}
L_g(x) &= \sum_{m \leq x} \frac{f(m)}{m} \sum_{d \leq x/m} \frac{1}{d} = \sum_{m \leq x} \frac{f(m)}{m}\left(\log(x/m) + \gamma + O(m/x)\right) \nonumber \\
&= \sum_{m \leq x} \frac{f(m)\log(x/m)}{m} + O(|L_f(x)| + 1).
\end{align}
The first term in \eqref{eq:Lflog} becomes
\begin{align*}
\sum_{md \leq x} \frac{f(md) \Lambda(d)}{md} = \sum_{d \leq x} \frac{f(d) \Lambda(d)}{d} L_f(x/d) + H(x),
\end{align*}
where we have set
$$
H(x) := \sum_{dm \leq x} \frac{\Lambda(d)}{d} \frac{f(md)-f(m)f(d)}{m} = \sum_{\nu \geq 1} \sum_{p^k \leq x} \frac{\log p}{p^k}\sum_{\ss{m \leq x/p^k \\ p^\nu || m}} \frac{f(mp^k) - f(m)f(p^k)}{m}.
$$
If $p^\nu||m$ then we write $m = p^\nu m'$ where $p \nmid m'$, and thus
$$
f(mp^k) - f(m)f(p^k) = f(m')(f(p^{k+\nu})-f(p^k)f(p^\nu)).
$$
Then,
\begin{align*}
|H(x)| &= \left|\sum_{\ss{p^{k+\nu} \leq x \\ k,\nu \geq 1}} \frac{(f(p^{k+\nu})-f(p^k)f(p^\nu))\log p}{p^{k+\nu}} \sum_{\ss{m'\leq x/p^{k+\nu} \\ p \nmid m'}} \frac{f(m')}{m'}\right| \\
&\leq 2\sum_{\ss{p^\ell \leq x \\ \ell \geq 2}} \frac{\log p^\ell}{p^\ell} \left|L_f(x/p^\ell) - \frac{1}{p}L_f(x/p^{\ell+1})\right| \ll \sum_{\ss{p^\ell \leq x \\ \ell \geq 2}} \frac{\log p^{\ell}}{p^\ell} |L_f(x/p^\ell)|.
\end{align*}
Using the trivial bound $|L_f(y/w)| \leq |L_f(y)| + \log (2w)$ for $1\leq w \leq y$, we deduce that
$$
|H(x)| \ll |L_f(x)| \sum_{\ss{p^{\ell} \leq x \\ \ell \geq 2}} \frac{\log p^\ell}{p^{\ell}} + \sum_{\ss{p^\ell \leq x \\ \ell \geq 2}} \frac{(\log p^\ell)^2}{p^{\ell}} \ll |L_f(x)| + 1,
$$
and so
$$
\sum_{n \leq x} \frac{f(n)\log n}{n} = \sum_{d \leq x} \frac{f(d)\Lambda(d)}{d} L_f(x/d) + O(|L_f(x)| + 1).
$$
The claim follows upon combining this with \eqref{eq:Lgarising} and \eqref{eq:Lflog}.
\end{proof}
%Next, we look at $\tilde{M}_g(w_0x)$. 
\begin{lem}\label{lem:Mglog}
We have
\begin{align*}
\tilde{M}_g(w_0 x)\log(w_0 x) = \sum_{\ss{p^r \leq x \\ r \geq 1}} \frac{f(p^r) \log p^r}{p^r} \tilde{M}_g(w_0 x/p^r) + \sum_{\ss{p^r \leq x \\ r \geq 1}} \frac{g(p^{r-1}) \log p}{p^r} \tilde{M}_g(w_0x/p^r) + O(|L_f(x)| + 1).
\end{align*}
\end{lem}
\begin{proof}
As in the previous lemma, 
\begin{align} \label{eq:Mglog}
\tilde{M}_g(w_0x)\log(w_0x) = \frac{1}{w_0x} \sum_{n \leq w_0x} g(n) \log n + \frac{1}{w_0x} \sum_{n \leq w_0x} g(n) \log(w_0x/n).
\end{align}
We note that for any $y \geq 1$,
\begin{align*}
\sum_{n \leq y} g(n) \log(y/n) = \sum_{m \leq y} f(m) \sum_{d \leq y/m} \log(y/(md)) = \sum_{m \leq y} f(m) \sum_{d \leq y/m} \int_d^{y/m} \frac{du}{u} = \sum_{m \leq y} f(m) \int_1^{y/m} \lfloor u\rfloor \frac{du}{u}.
\end{align*}
Writing $\lfloor u\rfloor = u-\{u\}$, 
%then splitting the integral into unit intervals, 
we obtain
$$
\int_1^{y/m} \lfloor u \rfloor \frac{du}{u} = 
%\frac{y}{m} -1 - \sum_{1 \leq k \leq y/m} \int_0^1 \frac{u}{u+k} du = \frac{y}{m} - \sum_{1 \leq k \leq y/m} \frac{1}{k} \int_0^1 u \ du + O(1) = 
\frac{y}{m} + O(\log(y/m)).
$$
Multiplying by $f(m)$, summing over $y$ and using $\sum_{m \leq y} \log (y/m) \ll y$, we get
$$
\sum_{m \leq y} f(m) \sum_{d \leq y/m} \log(y/md) = y\sum_{m \leq y} \frac{f(m)}{m} + O(y).
$$
We therefore deduce, when $y =w_0x$ that
\begin{align} \label{eq:Mglogy}
\frac{1}{w_0 x}\sum_{n \leq y} g(n) \log(w_0 x/n) = L_f(w_0x)+ O\left(1\right) = L_f(x) + O(1).
\end{align}
% In light of the estimate $L_f(w_0 x) = L_f(x) + O(1)$,
% we deduce, therefore, that
% \begin{align} \label{eq:glogx}
% \frac{1}{w_0 x} \sum_{n \leq w_0 x} g(n) \log(w_0x/n) = L_f(x) + O\left(1\right).
% \end{align}
%Summarising the above, we have the following.
%\begin{lem}\label{lem:toSmoothPart}
%Let $y$ be sufficiently large. Then
%$$
%(L_f(y) - \tilde{M}_f(w_0y)) \log(w_0y) = \sum_{d \leq y} %\frac{f(d)\Lambda(d)}{d} L_f(y/d) - \tilde{M}_{g\log}(w_0y) %+ L_g(y) + O(|L_f(y)| + 1).
%$$
%\end{lem}
Next, we reexpress $\tilde{M}_{g\log}(w_0x)$. In a similar way as above,
\begin{align*}
\tilde{M}_{g\log}(y) &= \sum_{\ss{p^k \leq y \\ k \geq 1}} \frac{\log p}{p^k} \cdot \frac{p^k}{y} \sum_{m\leq y/p^k} g(mp^k) = \sum_{\ell \geq 0} \sum_{\ss{p^{k+\ell} \leq y \\ k \geq 1}} \frac{g(p^{k+\ell}) \log p}{p^{k+\ell}} \cdot \frac{p^{k+\ell}}{y} \sum_{\ss{n \leq y/p^{k+\ell} \\ p \nmid n}} g(n)  \\
&= \sum_{\ss{p^r \leq y \\ r \geq 1}} \frac{g(p^r) \log p}{p^r} \cdot \frac{p^r}{y} \left(M_g(y/p^r) - M_g(y/p^{r+1})\right) \left(\sum_{0 \leq \ell \leq r-1} 1\right) \\
&= \sum_{\ss{p^r \leq y \\ r \geq 1}} \frac{g(p^r) \log p^r}{p^r} \cdot \frac{p^r}{y} \left(M_g(y/p^r) - M_g(y/p^{r+1})\right).
\end{align*}
Noting that $M_g(y/p^{r+1}) = 0$ unless $p^{r+1} \leq y$, this becomes
\begin{align*}
\sum_{\ss{p^r \leq y \\ r \geq 1}} \left(\frac{g(p^r) \log p^r}{p^r} - \frac{g(p^{r-1}) \log(p^{r-1})}{p^r}\right) \tilde{M}_g(y/p^r)
&= \sum_{\ss{p^r \leq y \\ r \geq 1}} \frac{\log p}{p^r} \tilde{M}_g(y/p^r) \left(rg(p^r) - (r-1) g(p^{r-1})\right).
\end{align*}
Now, by definition we have $g(p^r) = g(p^{r-1}) + f(p^r)$, so that
$$
rg(p^r) - (r-1)g(p^{r-1}) = rf(p^r) + g(p^{r-1}).
$$
Inserting this into the previous expression, we thus find that
\begin{align*}
\tilde{M}_{g\log}(y) = \sum_{\ss{p^r \leq y \\ r \geq 1}} \frac{\log p}{p^r} \tilde{M}_g(y/p^r) \left(rf(p^r) + g(p^{r-1})\right) = \sum_{\ss{p^r \leq y \\ r \geq 1}} \frac{f(p^r) \log p^r}{p^r} \tilde{M}_g(y/p^r) + \sum_{\ss{p^r \leq y \\ r \geq 1}} \frac{g(p^{r-1}) \log p}{p^r} \tilde{M}_g(y/p^r).
\end{align*}
Taking $y = w_0x$, and noting that if $x < d \leq w_0x$ then $M_g(w_0x/d) = 1$ and $\sum_{x < p^k \leq w_0 x} (\log p^k)/p^k \ll 1$, we have
$$
\tilde{M}_{g\log}(w_0x) = \sum_{\ss{p^r \leq x \\ r \geq 1}} \frac{f(p^r) \log p^r}{p^r} \tilde{M}_g(w_0 x/p^r) + \sum_{\ss{p^r \leq x \\ r \geq 1}} \frac{g(p^{r-1}) \log p}{p^r} \tilde{M}_g(w_0x/p^r) + O\left(1\right).
$$
Combining this with \eqref{eq:Mglogy}, we obtain the claim. 
\end{proof}
Taking the estimates from Lemma \ref{lem:LfDev} and \ref{lem:Mglog} and subtracting them, we obtain
%We now subtract this expression from $L_f(x)$, using %\eqref{eq:LfDev}, so that
\begin{align} \label{eq:LfMinMg}
\Delta(x)\log(w_0x) 
%\nonumber\\
&=\sum_{\ss{p^r \leq x \\ r \geq 1}} \frac{\log p}{p^r} \left(f(p^r) L_f(x/p^r)-(rf(p^r) + g(p^{r-1})) \tilde{M}_g(w_0x/p^r)\right) + L_g(x)\\
&+ O\left(|L_f(x)| + 1\right). \nonumber
\end{align} 
We now split the prime sum according to $r$. Write
$$
S_r(x) := \sum_{p^r \leq x} \frac{\log p}{p^r}\left(f(p^r) L_f(x/p^r) - (rf(p^r) + g(p^{r-1}))\tilde{M}_g(w_0x/p^r)\right).
$$
We dispense with $r \geq 2$ using Lemma \ref{lem:LipGHS} as follows.
\begin{lem} \label{lem:S2}
We have
$$
\sum_{r \geq 2} S_r(x) \ll |L_f(x)| + |\tilde{M}_g(w_0x)| + 1.
$$
\end{lem}
\begin{proof}
First, we have by the triangle inequality that 
$\left|\sum_{r \geq 2} S_r(x)\right| \leq |T_{f}(x)| + |T_{g}(x)|$, where
\begin{align*}
T_{f}(x) &:= \sum_{\ss{p^r \leq x \\ r \geq 2}} \frac{f(p^r) \log p}{p^r} L_f(x/p^r) \\
T_{g}(x) &:= \sum_{\ss{p^r \leq x \\ r \geq 2}} \frac{(rf(p^r) + g(p^{r-1})) \log p}{p^r} \tilde{M}_g(w_0x/p^r).
\end{align*}
We handle each of these terms separately, starting with $T_f(x)$. \\
By the trivial bound, $|L_f(x)-L_f(x/d)| \leq \log d + O(1)$ for $d \geq 1$, we get
\begin{align*}
\left|T_f(x)\right| \ll |L_f(x)|\sum_{\ss{p^r \leq x \\ r \geq 2}} \frac{\log p}{p^r} + \sum_{\ss{p^r \leq x \\ r \geq 2}} \frac{(\log p^r)^2}{p^r} \ll |L_f(x)| + 1.
\end{align*}
Next, we deal with $T_g(x)$. 
%Set $Z = (\log x)^{100}$. 
Note that $|rf(p^r) + g(p^{r-1})| \leq 2r$ for all $r \geq 1$. Therefore, by Lemma \ref{lem:LipGHS},
\begin{align*}
\left|T_g(x)\right|&\ll |\tilde{M}_g(w_0x)| \sum_{\ss{p^r \leq x \\ r \geq 2}} \frac{\log p^r}{p^r} + \sum_{\ss{p^r \leq x \\ r \geq 2}} \frac{\log p^r}{p^r} \left|\tilde{M}_g(w_0x/p^r) - \tilde{M}_g(w_0x)\right| \\
&\ll |\tilde{M}_g(w_0x)| + \sum_{\ss{p^r \leq x \\ r \geq 2}} \frac{(\log p^r)^2}{p^r} \ll |\tilde{M}_g(w_0x)| + 1,
%&+ \sum_{\ss{Z < p^r \leq x \\ r \geq 2}} \frac{\log p^r}{p^r} \tilde{M}_{|g|}(w_0x/p^r)|.
\end{align*}
%Observe that $\log p^r \ll \log\log x$ for all %$p^r \leq Z$. Therefore, applying Lemma %\ref{lem:LipGHS} to the second and bounding the %sum over $p^r$, $r \geq 2$ trivially, we find the %contribution of the first two terms on the RHS to %be
%$$
%\ll |\tilde{M}_g(w_0x)| + 1.
%$$
%For the last term, we use Lemma %\ref{lem:UppBdNonNeg} to obtain
%$$
%\tilde{M}_{|g|}(y) \ll \frac{1}{\log y} %\exp\left(\sum_{p \leq y} \frac{|g(p)|}{p}\right) %\ll \log y,
%$$
%whence we find that
%$$
%\sum_{\ss{Z < p^r \leq x \\ r \geq 2}} \frac{\log %p^r}{p^r} \tilde{M}_{|g|}(w_0x/p^r) \ll (\log x) %\sum_{\ss{p^r > (\log x)^{100} \\ r \geq 2}} %\frac{\log p^r}{p^r} \ll \frac{1}{(\log %x)^{24}}\sum_{p^r, r \geq 2} \frac{(\log p^r)}
%{p^{3r/4}} \ll \frac{1}{(\log x)^{24}}.
%$$
%Combining these error terms together, we get
%$$
%|T_g(x)| \ll |\tilde{M}_g(w_0x)| + 1,
%$$
and the claim follows.
\end{proof}
\noindent When $r = 1$ we have
$$
S_1(x) = \sum_{p \leq x} \frac{f(p) \log p}{p} \left(L_f(x/p) - \tilde{M}_g(w_0x/p)\right) - \sum_{p \leq x} \frac{\log p}{p} M_g(w_0 x/p).
$$
The first of these terms is of the shape needed to apply Hal\'{a}sz' averaging arguments, which we shall do below.
%, see Lemma \ref{lem:toInt}. 
Before doing so, we handle the second term.
%The second can be rewritten as
\begin{lem} \label{lem:S1Sec}
(a) If $f$ is real-valued and completely multiplicative then 
$$
\sum_{p \leq x} \frac{\log p}{p} M_g(w_0x/p) = L_g(x) + O(|L_f(x)| + |\tilde{M}_g(w_0x)| + 1).
$$
(b) More generally, if $f$ is $1$-bounded then we have the weaker estimate
$$
\sum_{p \leq x} \frac{\log p}{p} M_g(w_0x/p) = L_g(x) + O(|\tilde{M}_g(w_0x)|(\log\log x)^2 + (\log\log x)^4).
$$
\end{lem}
\begin{proof}
In the sequel, for $t \geq 1$ write 
$$
\theta(t) := \sum_{p\leq t} \log p, \quad R(t) := t^{-1} \theta(t) - 1.
$$
%and set $R(t) := t^{-1}\theta(t)- 1$. \\
%Before considering parts (a) and (b) separately, we make the following observations. \\
(a)  Upon rearranging and using Chebyshev's bounds,
\begin{align*}
\sum_{p \leq x} \frac{\log p}{p} \tilde{M}_g(w_0 x/p) &= \frac{1}{w_0 x}\sum_{n \leq w_0x} g(n) \sum_{p \leq \min\{x,w_0x/n\}} \log p = \frac{1}{w_0 x} \sum_{n \leq w_0 x} g(n)\theta(w_0 x/n) + O(1) \\
&= L_g(w_0x) + \sum_{n \leq w_0 x} \frac{g(n)}{n} \left(\frac{n}{w_0x}\theta(w_0x/n) - 1\right).
\end{align*}
 Note that
\begin{align*}
L_g(w_0x) &= L_g(x) + \sum_{m \leq x} \frac{f(m)}{m} \sum_{x/m < d \leq w_0 x/m} \frac{1}{d} + \sum_{x < m \leq w_0x} \frac{f(m)}{m}  \\
&= L_g(x) + \sum_{m \leq x} \frac{f(m)}{m} \left(\log w_0 + O\left(\frac{m}{x}\right)\right) + O(1) = L_g(x) + O(|L_f(x)| + 1).
\end{align*}
%In the sequel, write $R(t) := t^{-1}\theta(t)-1$, so that 
Let $1 \leq Z \leq x^{1/10}$ be a parameter to be chosen later, and define
$$
E := \left|\sum_{n \leq w_0 x} \frac{g(n)}{n} R(w_0x/n)\right|.
$$
%We now address parts (a) and (b) separately. \\
We therefore have that
$$
\sum_{p \leq x} \frac{\log p}{p} \tilde{M}_g(w_0x/p) = L_g(x) + O(E + |L_f(x)| + 1),
$$
and it remains to give an upper bound for $E$. \\
As $f$ is real-valued and completely multiplicative, we have $g \geq 0$. Thus, using the prime number theorem in the weak form 
$$
|R(t)| \ll \frac{1}{(\log (2t))^{10}}, \quad t \geq 1, 
%\quad A > 0,
$$
%taking $A = 10$ 
we find
$$
E = O\left(\frac{1}{(\log Z)^{10}}\sum_{n \leq x/Z} \frac{g(n)}{n}\right) + \left|\sum_{x/Z < n \leq w_0x} \frac{g(n)}{n} R(w_0x/n)\right|.
$$
We apply the triangle inequality, then decompose the second term into dyadic segments and apply the bound for $|R(t)|$ again, getting
\begin{align*}
&\leq \sum_{1 \leq 2^k \leq Z} \sum_{w_0x/2^{k+1} < n \leq w_0 x/2^k} \frac{g(n)}{n} \left|R(w_0x/n)\right| \ll \sum_{1 \leq 2^k \leq Z} \frac{1}{(k+1)^{10}} \frac{2^k}{w_0x} \sum_{w_0x/2^{k+1} < n \leq w_0x/2^k} g(n) \\
&= \sum_{1 \leq 2^k \leq Z} \frac{1}{(k+1)^{10}} \left(\tilde{M}_g(w_0x/2^k) - \frac{1}{2}\tilde{M}_g(w_0x/2^{k+1})\right).
\end{align*}
By Lemma \ref{lem:LipGHS} we have that for each $2^k \leq Z$,
\begin{align*}
\tilde{M}_g(w_0x/2^k) = \tilde{M}_g(w_0x) + O(k).
\end{align*}
Applying this with $k$ and $k+1$ in each summand above, we find that
$$
\tilde{M}_g(w_0x/2^k) - \frac{1}{2}\tilde{M}_g(w_0x/2^{k+1}) = \frac{1}{2}\tilde{M}_g(w_0x) + O(k+1).
$$
Inserting this into the sum over $k$, we find that
\begin{align*}
\sum_{1 \leq 2^k \leq Z} \frac{1}{(k+1)^{10}} \left(\tilde{M}_g(w_0x/2^k) - \frac{1}{2}\tilde{M}_g(w_0x/2^{k+1})\right) \ll \tilde{M}_g(w_0x) + 1.
\end{align*}
Thus, we have
$$
E \ll \frac{L_g(x)}{(\log Z)^{10}} + \tilde{M}_g(w_0x) + 1.
$$
Taking $Z = \exp(\sqrt{\log x})$ and using
$$
L_g(x) \ll \exp\left(\sum_{p \leq x} \frac{g(p)}{p}\right) \ll (\log x)^2,
$$
we obtain $E \ll \tilde{M}_g(w_0x) + 1$, and the first claim follows. \\
(b) 
%Some of the details of the proof here are the same. 
This time, let $Z = \exp(C (\log\log x)^2)$ for some large absolute constant $C > 0$ to be chosen later. We split the sum on the LHS in the statement as
$$
\sum_{p \leq Z} \frac{\log p}{p} \tilde{M}_g(w_0x/p) + \sum_{Z < p \leq x} \frac{\log p}{p} \tilde{M}_g(w_0x/p) =: T_1 + T_2.
$$
%The contribution from $p \leq Z$. 
By Lemma \ref{lem:LipGHS},  
\begin{align*}
T_1 
%= \sum_{p \leq Z} \frac{\log p}{p} \tilde{M}_g(w_0x/p) 
= \tilde{M}_g(w_0x) (\log Z + O(1)) + O\left(\sum_{p \leq Z} \frac{\log^2 p}{p}\right) &= \tilde{M}_g(w_0 x) \log Z + O(|\tilde{M}_g(w_0x)| + \log^2 Z).
%&\ll |\tilde{M}_g(w_0x)| (\log\log x)^2 + (\log\log x)^4.
\end{align*}
Arguing as above, we may now estimate $T_2$ as
\begin{align*}
T_2 &= \frac{1}{w_0x} \sum_{Z < p \leq x} (\log p) \sum_{n \leq w_0x/p} g(n) = \frac{1}{w_0x} \sum_{n \leq w_0 x/Z} g(n) \left(\theta(w_0x/n) - \theta(Z)\right) \\
&= \sum_{n \leq w_0x/Z} \frac{g(n)}{n} - \frac{Z}{w_0x} \sum_{n \leq w_0x/Z} g(n) + \sum_{n \leq w_0x/Z} g(n) \left(R(w_0x/n) - \frac{Z}{w_0x} R(Z)\right).
\end{align*}
%Adding and subtracting $(w_0x/n - Z)$ inside the %sum and using the prime number theorem in the form
%$$
%\theta(y) = y + O\left(\exp(-c\sqrt{ \log %y})\right),
%$$
%the above becomes
Applying the prime number theorem with $|R(t)| \ll e^{-c\sqrt{\log t}}$ this time, we find
\begin{align*}
%&\sum_{n \leq w_0x/Z} \frac{g(n)}{n} - \frac{Z}%%{w_0 x} \sum_{n \leq w_0x/Z} g(n) + O\left(\tilde{M}_{|g|}(w_0x/Z) \exp\left(-c\sqrt{\log Z}\right)\right) \\
T_2 = L_g(w_0x/Z) - \tilde{M}_g(w_0x/Z) + O\left(\frac{\tilde{M}_{|g|}(w_0 x/Z)}{(\log x)^{c\sqrt{C}}}\right).
\end{align*} 
%Similarly, again by Lemma \ref{lem:LipGHS}, we %have
%$$
%\tilde{M}_g(w_0x/Z) \ll |\tilde{M}_g(w_0x)| + %\log Z.
%$$
By partial summation,
\begin{align*}
L_g(x) - L_g(w_0x/Z) + \tilde{M}_g(w_0x/Z) &= \tilde{M}_g(x) + \int_{w_0x/Z}^x \tilde{M}_g(u) \frac{du}{u}.
\end{align*}
To treat the integral, we use Lemma \ref{lem:LipGHS} as
\begin{align*}
\int_{w_0x/Z}^x \tilde{M}_g(u) \frac{du}{u} = \int_1^{Z/w_0} \tilde{M}_g(x/v) \frac{dv}{v} 
&= \tilde{M}_g(x) \log(Z/w_0) + O\left(\int_1^Z \frac{\log v}{v} dv\right) \\
&\ll |\tilde{M}_g(w_0x)| \log Z + (\log Z)^2.
\end{align*}
%&\ll |\tilde{M}_g(w_0x)| + 1 + \int_1^{Z/w_0} \frac{\log(x/v)}{v} du \ll |\tilde{M}_g(w_0x)| \log Z + (\log Z)^2.
%\end{align*}
Putting all of these bounds together, we obtain 
$$
T_1 + T_2 = L_g(x) + O\left( |\tilde{M}_g(w_0x)| \log Z + (\log Z)^2 + \frac{\tilde{M}_{|g|}(w_0 x/Z)}{(\log x)^{c\sqrt{C}}}\right).
$$
Choosing $C$ sufficiently large in our definition of $Z$ and applying Lemma \ref{lem:UppBdNonNeg}(a), the error term here is 
$$
\ll |\tilde{M}_g(w_0x)| (\log\log x)^2 + (\log\log x)^4,
$$
as claimed.
%$\ll 1$ provided $C$ is large enough, and the claim follows.
% \begin{align*}
% \sum_{p \leq x} \frac{\log p}{p} \tilde{M}_g(w_0 x/p) = \sum_{p \leq Z} \frac{\log p}{p} \tilde{M}_g(w_0x/p) + \sum_{Z < p \leq x} \tilde{m}_
\end{proof}
%    Returning to \eqref{eq:LfMinMg}, we may now see in light of Lemmas \ref{lem:S2} and \ref{lem:S1Sec} that
%    \begin{align} \label{eq:allbutS1}
 %   &\left|L_f(x)- \tilde{M}_g(w_0x)\right| \nonumber\\
  %  &\ll \frac{1}{\log x} \left|\sum_{p\leq x} \frac{f(p)\log p}{p} \left(L_f(x/p) - \tilde{M}_g(w_0x/p)\right)\right| + \frac{1}{\log x}\left(|L_f(x)| + \tilde{M}_g(w_0x) + 1\right).
  %  \end{align}
%A priori, the second term should be easy to handle, but once again the fact that $g(n)$ can grow is a bit of an issue. 
%We deal with the prime sum as follows. 
% \begin{lem}\label{lem:toInt}
% Recall that for $y \geq 1$ we write
% $$
% \Delta(y) := L_f(y) - \tilde{M}_g(w_0 y).
% $$
% Let $x \in \mf{S}$ from \eqref{eq:mfSDef}. Then we have
% $$
% \frac{1}{\log x} \left|\sum_{p \leq x} \frac{f(p)\log p}{p} \Delta(x/p) \right| \ll \frac{1}{\log x} \int_1^x |\Delta(y)| \frac{dy}{y} + O\left(\frac{|\tilde{1}{\log x}\right).
% $$
% \end{lem}
\begin{proof}[Proof of Proposition \ref{prop:toAvg}]
We will prove part (a), and highlight the necessary changes needed to prove part (b) afterwards. As the property $g \geq 0$ is only important in one place, for the most part the details below may be applied to $1\ast f$ for any $1$-bounded function. \\
Assume that $x \in \mf{S}$, as defined in \eqref{eq:mfSDef}.
For each $1 \leq y \leq x$ define
$$
T(y) := L_{f\log}(y) - \tilde{M}_{g\log}(w_0y) + L_g(y).
$$
We note that by combining \eqref{eq:Lflog}, \eqref{eq:Lgarising}, \eqref{eq:Mglog} and \eqref{eq:Mglogy},
%$$
%$$
%By Lemma \ref{lem:toSmoothPart}, for each $y \geq x_0$,
\begin{equation}\label{eq:DeltatoT}
\Delta(x) \log(w_0x) = T(x) + O(|L_f(x)|+ 1).
%\Delta(x) \log(w_0x) = L_{f\log}(y) - \tilde{M}_{g\log}(y) + L_g(y) + O(|L_f(y)| + 1).
\end{equation}
%Call the main term $T(y)$, and 
Let $h := x/(\log x)^2$. We will first show that
\begin{equation}\label{eq:TAvg}
T(x) = \frac{1}{h} \int_{x-h}^x T(y) dy + O(1).
\end{equation}
Observe that for each $x-h < y \leq x$,
\begin{align*}
T(x)-T(y) &= \sum_{y < n \leq x} \frac{f(n)\log n}{n} + \sum_{y < n \leq x} \frac{g(n)}{n} + \frac{1}{w_0x}\left(\frac{x}{y} \sum_{n \leq w_0y} g(n)\log n - \sum_{n \leq w_0 x} g(n)\log n\right) \\
&
%+ \sum_{y < n \leq x} \frac{g(n)}{n} 
=: A_1 + A_2 + A_3.
\end{align*}
Since $|f(n)| \leq 1$ and $x-h \geq x/2$, we have the trivial bound
$$
|A_1| \leq \frac{2(\log x)}{x} \sum_{x-h < n \leq x} 1 \ll \frac{h \log x}{x} \ll \frac{1}{\log x}.
$$
To bound $A_2$ we use 
%$(\log n)/n \ll (\log x)/x$ together with 
Lemma \ref{lem:UppBdNonNeg}(b) to get
$$
|A_2| \ll \frac{1}{x} \sum_{x-h < n\leq x} |g(n)| \ll \frac{h}{x} \exp\left(\sum_{p \leq x} \frac{|g(p)|-1}{p}\right) \ll \frac{h \log x}{x} \ll \frac{1}{\log x}.
$$
% As $y \geq x/2$, $A_4$ can similarly be bounded as
% $$
% |A_4| \ll \frac{1}{x} \sum_{x-h < n \leq x} |g(n)| \ll \frac{h\log x}{x} \ll \frac{1}{\log x}.
% $$
Finally, note that the bracketed expression in the definition of $A_3$ may be bounded above by
$$
\left(\frac{x}{y}-1\right) \sum_{n \leq w_0y} |g(n)|\log n + \sum_{w_0y <n \leq w_0x} |g(n)|\log n \leq (\log x)\left(\left(\frac{x}{y}-1\right) \sum_{n \leq w_0y} |g(n)| + \sum_{w_0y <n \leq w_0x} |g(n)|\right).
$$
The first of these expressions is, by Lemma \ref{lem:UppBdNonNeg}(a)
$$
\ll h(\log x)^2 \ll x.
$$
For the second we again use Lemma \ref{lem:UppBdNonNeg}(b) to obtain
$$
\ll (\log x) \sum_{w_0(x-h) <n \leq w_0x} |g(n)| \ll h(\log x)^2 \ll x.
$$
It follows that $|A_3| \ll 1$, and combined with the earlier bounds, we obtain 
$$
T(x) = \frac{1}{h}\int_{x-h}^x(T(y) + (T(x)-T(y))) dy = \frac{1}{h} \int_{x-h}^x T(y) dy + O(1),
$$
as claimed.\\
%\eqref{eq:TAvg}. \\
Next, on combining \eqref{eq:LfMinMg} with Lemmas \ref{lem:S2} and \ref{lem:S1Sec}(a) (using here that $f$ is real and completely multiplicative), we get for each $x-h \leq y \leq x$ that
\begin{align}
T(y) &= \sum_{\ss{p^r \leq y \\ r \geq 1}} \frac{\log p}{p^r} \left(f(p^r) L_f(y/p^r) - (rf(p^r) + g(p^{r-1})) \tilde{M}_g(w_0y/p^r)\right) + L_g(y) \nonumber\\
&= \sum_{p \leq y} \frac{f(p)\log p}{p} \Delta(y/p) + O(|L_f(y)| + |\tilde{M}_g(w_0y)| +1). \label{eq:changenonReal}
\end{align}
By the triangle inequality and Lemma \ref{lem:LipGHS} respectively, we have $|L_f(y)| \leq |L_f(x)| + O(1)$ and $|\tilde{M}_g(w_0y)| \leq |\tilde{M}_g(w_0x)| + O(1)$ for all $x-h < y \leq x$, so that on integrating, we obtain
$$
T(x) = \frac{1}{h} \int_{x-h}^x \sum_{p \leq y} \frac{f(p)\log p}{p} \Delta(y/p) dy + O(|L_f(x)| + |\tilde{M}_g(w_0x)| + 1).
$$
Using \eqref{eq:DeltatoT} and the triangle inequality, we thus obtain
%we have 
%$$
%\Delta(x)\log = T(x) + O(|L_f(x)| + 1).
%$$
%It thus follows by the triangle inequality, that
$$
|\Delta(x)|\log(w_0x) \leq \frac{1}{h} \int_{x-h}^x \sum_{p \leq y} \frac{|f(p)|\log p}{p} |\Delta(y/p)| dy + O(|L_f(x)| + |\tilde{M}_g(w_0x)| +1). 
$$
We now swap orders of summation and integration, make the change of variables $u = y/p$, then swap back, to upper bound the integral as
\begin{align} \label{eq:intBd}
\frac{1}{h} \int_{x-h}^x \sum_{p \leq y} \frac{|f(p)|\log p}{p} |\Delta(y/p)| dy &\leq \frac{1}{h}\sum_{p \leq x} \frac{\log p}{p} \int_{x-h}^x |\Delta(y/p)|dy = \frac{1}{h} \sum_{p \leq x} (\log p) \int_{(x-h)/p}^{x/p} |\Delta(u)| du \nonumber \\
&= \int_1^x |\Delta(u)| \left(\frac{1}{h}\sum_{(x-h)/u < p \leq x/u} \log p\right) du.
\end{align}
In the range $u \leq h/(\log x)$ we apply the Brun-Titchmarsh inequality (e.g. \cite[Thm. I.4.16]{Ten} with $q = 1$) to obtain
$$
\sum_{(x-h)/u < p \leq x/u} \log p \leq (\log x/u) \left(\pi(x/u) - \pi((x-h)/u)\right) \ll \frac{h \log(x/u)}{u \log(h/u)} \ll \frac{h}{u}.
$$
When $h/(\log x) < u \leq x$ we instead use the trivial bound
$$
\sum_{(x-h)/u < p \leq x/u} \log p \leq \sum_{(x-h)/u < n\leq x/u} \log n \leq \frac{h}{u}\log (x/u).
$$
By hypothesis, $x \in \mf{S}$ and thus for all $h/\log x < u \leq x$
$$
|\Delta(u)| \leq |\Delta(x)| \frac{\log(w_0x)}{\log (w_0u)} \ll |\Delta(x)|.
$$
Hence, the RHS of \eqref{eq:intBd} is
$$
\ll \frac{1}{h}\int_1^{h/\log x} |\Delta(u)| \frac{h}{u} du + \frac{1}{h} \int_{h/\log x}^x |\Delta(x)| (\log (x/u)) \frac{h}{u}du \ll \int_1^x |\Delta(u)| \frac{du}{u} + |\Delta(x)| (\log \log x)^2.
$$
We therefore deduce that
$$
|\Delta(x)| \log(w_0x) \ll \int_1^x |\Delta(u)| \frac{du}{u} + |\Delta(x)| (\log\log x)^2 + |L_f(x)| + |\tilde{M}_g(w_0x)| + 1.
$$
Rearranging and dividing through by $\log x$, we obtain the first claim when $x_0$ is sufficiently large. \\
The proof of the second claim is identical, except that in treating \eqref{eq:intBd} the sum over $p$ can be restricted to $p \leq z$ (as otherwise $f(p) = 0$). As we get that $(x-h)/p \geq x/(2z)$ for $x$ sufficiently large, making the change of variables $u = y/p$ then swapping orders of summation and integration,  we find that the integral over $u$ in \eqref{eq:intBd} may be restricted to $[x/(2z),x]$. The remainder of the argument is unchanged. This finishes the proof of part (a). \\
To prove part (b), all of the above arguments continue to hold, but we apply part (b) of Lemma \ref{lem:S1Sec} in \eqref{eq:changenonReal}, rather than part (a).
%the error term must be replaced by $O(|L_f(y)| + |\tilde{M}_g(w_0y)|(\log\log y)^2 + (\log\log y)^4)$. 
The remainder of the proof is unchanged.
\end{proof}
\subsection{Bounding the integral average}
%In this subsection we \emph{do not} assume that $f$ is real-valued. \\ 
Next, for $y \geq e$ we consider 
\begin{equation}\label{eq:Jdef}
J(y) := \int_e^y |\Delta(u)| \log (w_0u) \frac{du}{u}.
\end{equation}
As a Stieltjes measure, $dJ(u) = |\Delta(u)| \log(w_0u) \tfrac{du}{u}$, and on integrating by parts, we see that
\begin{equation}\label{eq:toJ}
\int_e^x |\Delta(u)| \frac{du}{u} = \frac{J(x)}{\log x} - J(e) + \int_e^x \frac{J(u)}{u(\log u)^2} du \leq \frac{J(x)}{\log x} + \int_e^x \frac{J(u)}{u(\log u)^2} du.
\end{equation}
%We thus next consider estimating $J(y)$, where $y$ is a sufficiently large integer that is fixed for the time being.
Let $\alpha := 1/\log y$. We apply the Cauchy-Schwarz inequality, multiply the integrand by $u^{-2\alpha}$ (which is $\asymp 1$ for $u \leq y$) and extend the integral to infinity, then finally make the change of variables $u = e^v$, to obtain
\begin{align}\label{eq:CStoJ}
J(y) &\ll (\log y)^{1/2} \left(\int_e^y|\Delta(u)|^2 (\log(w_0u))^2 \frac{du}{u}\right)^{1/2} \ll (\log y)^{1/2}\left(\int_1^{\infty} |\Delta(u)| (\log(w_0u))^2 \frac{du}{u^{1+2\alpha}}\right)^{1/2} \nonumber\\
&= (\log y)^{1/2} \left(\int_0^{\infty} |\Delta(e^v)|^2 (\log(w_0e^v))^2 e^{-2\alpha v} dv\right)^{1/2} =: (\log y)^{1/2} \mc{I}(y)^{1/2}.
\end{align}
We will prove the following bound for $\mc{I}(y)$.
\begin{prop}\label{prop:IyBd}
Let $T \geq 1$, $x_0 \leq y \leq x$ and write $\alpha := 1/\log y$. Then
%In fact, we have the more precise bound
\begin{align*}
\mc{I}(y) &\ll \alpha^{-1} \left(\max_{|t| \leq 1/2} \left|\frac{L(1+\alpha+it,f)}{\zeta(1+\alpha+it)}\right|^2 + \sum_{1 \leq k \leq T-1/2} \frac{(\log(2k))^4}{k^2} \max_{|t-k| \leq 1/2} |L(1+\alpha+it,f)|^2 + 1\right) \\
&+ \frac{(\log (2T))^4}{T}\left(1+\frac{1}{\alpha T}\right) + \frac{(\log (2T))^2}{\alpha^3T^2}.
\end{align*}
Consequently, we obtain
$$
\mc{I}(y) \ll \alpha^{-1} \left(\max_{|t| \leq T} \left|\frac{L(1+\alpha+it,f)}{\zeta(1+\alpha+it)}\right|^2 + 1\right) + \frac{(\log (2T))^4}{T}\left(1+\frac{1}{\alpha T}\right) + \frac{(\log (2T))^3}{\alpha^3 T^2}.
$$
\end{prop}
Let us now define
\begin{align*}
\phi_{\alpha}(v) &:= e^{-\alpha v} \sum_{n \leq e^v} \frac{f(n)\log n}{n} - \frac{1}{w_0 e^{(1+\alpha) v}} \sum_{n \leq w_0 e^v} g(n) \log n, \\
\psi_{\alpha}(v) &:= e^{-\alpha v} \sum_{n \leq w_0e^v} \frac{f(n) \log(w_0 e^v/n)}{n} - \frac{1}{w_0 e^{(1+\alpha)v}} \sum_{n \leq w_0e^v} g(n) \log(w_0e^v/n)
\end{align*}
(both of which vanish for $v \leq -1$). It is clear that
\begin{align*}
\Delta(e^v)\log(w_0e^v) e^{-\alpha v} &= \phi_{\alpha}(v) +\psi_{\alpha}(v) - e^{-\alpha v}\sum_{e^v < n\leq w_0e^v} \frac{f(n) \log(w_0e^v/n)}{n} \\
&= \phi_{\alpha}(v) + \psi_{\alpha}(v) + O(e^{-\alpha v}). 
\end{align*}
For an absolutely integrable function $\phi(t)$ write 
$$
\hat{\phi}(\xi) := \int_{-\infty}^{\infty} \phi(v) e^{-i\xi v} dv. 
$$
It is easily shown that for each $\xi \in \mb{R}$,
\begin{align*}
\hat{\phi}_{\alpha}(\xi) = -\frac{L'(1+\alpha + i \xi,f)}{\alpha + i\xi} + w_0^{\alpha + i \xi} \frac{L'(1+\alpha+i \xi,g)}{1+\alpha + i\xi} \\
\hat{\psi}_{\alpha}(\xi) = w_0^{\alpha + i \xi}\left(\frac{L(1+\alpha + i \xi,f)}{(\alpha + i \xi)^2} - \frac{L(1+\alpha+i \xi,g)}{(1+\alpha + i \xi)^2}\right).
\end{align*}
Using the factorisation $L(s,g) = L(s,f)\zeta(s)$ for $\text{Re}(s) > 1$, we get
$$
L'(s,g) = L'(s,f)\zeta(s) + L(s,f)\zeta'(s),
$$
so that on grouping terms together and applying Plancherel's theorem, 
\begin{align} \label{eq:Idecomp}
\mc{I}(y) &\ll \int_{-\infty}^{\infty} \left|\hat{\phi}_{\alpha}(t) + \hat{\psi}_{\alpha}(t)\right|^2 dt + \int_0^{\infty} e^{-2\alpha v} dv \nonumber\\
&\ll \int_{-\infty}^{\infty} \left|\frac{L'(1+\alpha + it,f)w_0^{\alpha + it}}{1+\alpha + it}\left(\frac{(1+\alpha + it) w_0^{-\alpha-it}}{\alpha + it} - \zeta(1+\alpha + it)\right)\right|^2 dt \nonumber\\
&+ \int_{-\infty}^{\infty} \left|\frac{L(1+\alpha + it,f)w_0^{\alpha + it}}{1+\alpha + it}\left(\zeta'(1+\alpha + it) + \frac{1+\alpha + it}{(\alpha + it)^2} - \frac{\zeta(1+\alpha + it)}{1+\alpha + it}\right)\right|^2 dt + \alpha^{-1} \nonumber\\
&=: \mc{I}_1(y) + \mc{I}_2(y) + O(\alpha^{-1}).
\end{align}
%To estimate $\mc{I}(y)$ we shall first show the following simple bounds.
Consider the integral $\mc{I}_1(y)$. For $0 \leq T_1 \leq T_2 \leq \infty$, set
$$
\mc{I}_1(y;T_1,T_2) := \int_{T_1 \leq |t| \leq T_2} \left|\frac{L'(1+\alpha + it,f)w_0^{\alpha + it}}{1+\alpha + it}\left(\frac{(1+\alpha + it) w_0^{-\alpha-it}}{\alpha + it} - \zeta(1+\alpha + it)\right)\right|^2 dt,
$$
and decompose $\mc{I}_1(y) = \mc{I}_1(y;0,T) + \mc{I}_1(y; T,\infty)$. The following simple bound holds for the tail integral.
\begin{lem}\label{lem:I1yTail}
We have 
$$
\mc{I}_1(y; T,\infty) \ll \frac{(\log (2T))^2}{T}\left(1 + \frac{1}{\alpha^3 T}\right).
$$
\end{lem}
\begin{proof}
Using the standard bound $|\zeta(1+\alpha + it)| \ll \log(2+|t|)$ for $|t| \geq 1$, we have
\begin{align*}
\mc{I}_1(y;T,\infty) &\ll \sum_{k \geq 1} \int_{-T}^T \left(\left|\frac{L'(1+\alpha + i(2kT + t),f)}{\alpha + i(2kT+t)}\right|^2 +  \left| \frac{L'(1+\alpha + i(2kT + t),f)\zeta(1+\alpha + i(2kT+t))}{1+\alpha + i(2kT+t)}\right|^2\right) dt \\
&\ll \sum_{k \geq 1} \frac{1+\log(2kT)^2}{k^2T^2} \int_{-T}^{T} |L'(1+\alpha+ i(2kT + t))|^2 dt.
\end{align*}
On each segment we apply Gallagher's lemma (see e.g. \cite[Lem. III.4.9]{Ten}) in the form
$$
\int_{-T}^T \left|\sum_{n \geq 1} \frac{a_n}{n^{\sg+it}}\right|^2dt \ll \sum_{n \geq 1} \frac{|a_n|^2(T+n)}{n^{2\sg}}
$$
with $a_n = f(n)n^{-i2kT} \log n$, getting
\begin{align*}
\mc{I}_1(y;T,\infty) &\ll \frac{1}{T^2} \sum_{k \geq 1} \frac{(\log(2k))^2 + (\log (2T))^2}{k^2} \sum_{n \geq 1} \frac{(T+n) (\log 2n)^2}{n^{2+2\alpha}} \\
&\ll \frac{(\log (2T))^2}{T^2} \left(T + \frac{1}{\alpha^{3}}\right) \ll \frac{(\log (2T))^2}{T}\left(1 + \frac{1}{\alpha^3 T}\right),
\end{align*}
as claimed.
\end{proof}
We now handle the segment $[-T,T]$.
\begin{lem} \label{lem:I1y}
We have
\begin{align*}
\mc{I}_1(y;0,T) \ll \alpha^{-1} \max_{|t| \leq 1/2} \left|\frac{L(1+\alpha + it,f)}{\zeta(1+\alpha + it)}\right|^2 +  \alpha^{-1}\sum_{1 \leq k \leq T-1/2} \frac{(\log(2k))^2}{k^2} \max_{|t-k| \leq 1/2}  \left|L(1+\alpha + it,f)\right|^2.
\end{align*}
Consequently, we obtain
$$
\mc{I}_1(y; 0,T) \ll \alpha^{-1} \max_{|t| \leq T} \left|\frac{L(1+\alpha+it,f)}{\zeta(1+\alpha+it)}\right|^2.
$$
\end{lem}
\begin{proof}
Consider first the range $|t| \leq 1/2$. Employing the factorisation $L' = (L'/L) \cdot L$ and applying the first estimate in Lemma \ref{lem:snear1}, we get an upper bound
\begin{align*}
&\ll \int_{-1/2}^{1/2} |L'(1+\alpha+it,f)|^2 \left|\frac{(1+\alpha+it)w_0^{-\alpha-it}}{\alpha + it} - \zeta(1+\alpha + it)\right|^2 dt \\
&\ll \int_{-1/2}^{1/2}|L(1+\alpha+it,f)|^2 |\alpha + it|^2 \left|\frac{L'}{L}(1+\alpha+it,f)\right|^2 dt.
\end{align*}
Noting that $|\zeta(1+\alpha+it)|^{-1} \asymp |\alpha+it|$ in this range, we apply an $L^{\infty}$ bound on $|L/\zeta|^2$ and the Montgomery-Wirsing majorant principle (see Lem. III.4.10 of \cite{Ten}) to the resulting $L^2$ integral, getting
$$
\ll \max_{|t| \leq 1/2} \left|\frac{L(1+\alpha+it,f)}{\zeta(1+\alpha+it)}\right|^2 \cdot \int_{-1/2}^{1/2} \left|\frac{\zeta'}{\zeta}(1+\alpha+it)\right|^2 dt \ll \alpha^{-1}\max_{|t| \leq 1/2} \left|\frac{L(1+\alpha+it,f)}{\zeta(1+\alpha+it)}\right|^2.
$$
Similarly, for each $1 \leq k \leq T-1/2$ we use the bounds
$$
\max_{|t-k| \leq 1/2} \left|\frac{(1+\alpha + it)w_0^{-\alpha-it}}{\alpha+it} - \zeta(1+\alpha+it)\right| \ll \log(2k)
$$
together with a further application of the majorant principle to bound the range $1/2 \leq |t| \leq T$ as
\begin{align*}
&\ll \sum_{1 \leq k \leq T-1/2} \frac{\log(2k)^2}{k^2} \left(\max_{|t-k| \leq 1/2} |L(1+\alpha+it,f)|^2\right) \cdot \int_{k-1/2}^{k+1/2} \left|\frac{L'}{L}(1+\alpha+it,f)\right|^2 dt \\ 
&\ll \alpha^{-1} \sum_{1 \leq k \leq T-1/2} \frac{(\log(2k))^2}{k^2} \max_{|t-k| \leq 1/2} |L(1+\alpha+it,f)|^2.
\end{align*}
Together with the contribution from $|t| \leq 1/2$, this implies the first claim. \\
For the second claim, we note that by the standard bound 
\begin{equation}\label{eq:logzeta}
|\log(\zeta(1+\alpha + it))| \leq \log\log(2+|t|) + O(1), \quad |t| \geq 2
\end{equation}
(see e.g., Sec. II.3.10 of \cite{Ten}) we have that
$$
\max_{|t-k| \leq 1/2} |L(1+\alpha+it,f)|^2 \ll \log(2k)^2 \max_{|t-k| \leq 1/2} \left|\frac{L(1+\alpha+it,f)}{\zeta(1+\alpha + it)}\right|^2.
%\ll \alpha^{-1} .
$$ 
Inserting this into the sum over $k$, the contribution from $k \geq 1$ is
$$
\leq \max_{1 \leq |t| \leq T} \left|\frac{L(1+\alpha + it,f)}{\zeta(1+\alpha + it)}\right|^2 \sum_{ k\geq 1} \frac{(\log(2k))^4}{k^2} \ll \max_{1 \leq |t| \leq T} \left|\frac{L(1+\alpha + it,f)}{\zeta(1+\alpha + it)}\right|^2.
$$
Combining this with the contribution from $|t| \leq 1/2$ implies the claim.
%To get the second, we split the sum at $k = \lfloor T \rfloor$. 
\end{proof}
We analogously split $\mc{I}_2(y) = \mc{I}_2(y;0,T) + \mc{I}_2(y;T,\infty)$, and handle the tail in the same way. 
\begin{lem} \label{lem:I2yTail}
We have
$$
\mc{I}_2(y;T,\infty) \ll \frac{(\log (2T))^4}{T}\left(1+\frac{1}{\alpha T}\right).
$$
\end{lem}
\begin{proof}
The proof is the same as in Lemma \ref{lem:I1yTail}, save that we also use the standard upper bound 
$$
|\zeta'(1+\alpha + i(2kT + t))|^2 \ll \log^4(2k) + \log^4 (2T)
$$
for all $k \geq 1$. The application of Gallagher's lemma is with $a_n = f(n)n^{-i2kT}$ on each segment $[(2k-1)T, (2k+1)T]$, to show that
$$
\int_{-T}^T |L(1+\alpha+i(2kT + t))|^2 dt \ll \sum_{n \geq 1}\frac{T+n}{n^{2+2\alpha}} \ll T + \alpha^{-1}.
$$
We leave the details to the reader.
\end{proof}
The segment $[-T,T]$ then gives rise to a similar bound as in Lemma \ref{lem:I1y}.
\begin{lem} \label{lem:I2y}
%The following estimates hold:
We have 
$$
\mc{I}_2(y;0,T) \ll \alpha^{-1} \max_{|t| \leq 1/2} \left|\frac{L(1+\alpha+it,f)}{\zeta(1+\alpha+it)}\right|^2 + \alpha^{-1}\sum_{1 \leq k \leq T-1/2} \frac{(\log(2k))^4}{k^2}\max_{|t-k| \leq 1/2} \left|L(1+\alpha+it,f)\right|^2.
$$
Consequently, we have
$$
\mc{I}_2(y;0,T) \ll \alpha^{-1}\max_{|t| \leq T} \left|\frac{L(1+\alpha+it,f)}{\zeta(1+\alpha+it)}\right|^2.
$$
\end{lem}
\begin{proof}
In the range $|t| \leq 1/2$ we apply the second part of Lemma \ref{lem:snear1}, while in the range $1/2 \leq |t| \leq T$, we use the bounds
$$
|\zeta(1+\alpha+it)| \ll \log(2+|t|), \quad |\zeta'(1+\alpha + it)| \ll \log^2(2+|t|).
$$
This yields
\begin{align*}
\mc{I}_2(y;0,T) &\ll \int_{-1/2}^{1/2} |L(1+\alpha+it,f)|^2 dt \\
&+ \sum_{1 \leq k \leq T-1/2} \int_{k-1/2}^{k+1/2} |L(1+\alpha+it,f)|^2 \left(\frac{1+|\zeta(1+\alpha+it)|^2}{1+t^2} + |\zeta'(1+\alpha+it)|^2\right) \frac{dt}{1+t^2} \\
&\ll \int_{-1/2}^{1/2} |L(1+\alpha+it,f)|^2 dt + \sum_{k \geq 1} \frac{(\log(2k))^4}{k^2} \left(\int_{-1/2}^{1/2} |L(1+\alpha+i(t+k),f)|^2 dt\right).
\end{align*}
For $k \geq 1$ we use the trivial inequality
$$
\int_{-1/2}^{1/2} |L(1+\alpha+i(t+k),f)|^2 dt \leq \max_{|t-k| \leq 1/2} |L(1+\alpha+it,f)|^2.
$$
For $k = 0$, we write $L = \zeta \cdot (L/\zeta)$ and apply an $L^{\infty}$ bound to the ratio $L/\zeta$, obtaining that
\begin{align}\label{lem:LtoZetaLinf}
\int_{-1/2}^{1/2} |L(1+\alpha+it,f)|^2 dt \leq \left(\max_{|t| \leq 1/2} \left|\frac{L(1+\alpha+it,f)}{\zeta(1+\alpha + it)}\right|^2 \right)\int_{-1/2}^{1/2} |\zeta(1+\alpha + it)|^2 dt.
\end{align}
% We then apply the majorant principle to get
% \begin{equation}\label{eq:majtoZeta}
% \int_{-1/2}^{1/2} |\zeta(1+\alpha+i(t+k))|^2 dt \ll \int_{-1/2}^{1/2} |\zeta(1+\alpha + it)|^2 dt
% \end{equation}
% (which, though wasteful for $k \geq 1$, matches the bound for $k = 0$ anyway). 
Now, applying  $|\zeta(1+\alpha+it)| \asymp |\alpha + it|^{-1} \leq \max\{\alpha, |t|\}^{-1}$ for $|t| \leq 1/2$, the latter integral is bounded above by
$$
\ll \frac{1}{\alpha^2} \int_0^{\alpha} dt + \int_{\alpha}^1 \frac{dt}{t^2} \ll \alpha^{-1}.
$$
Since $\alpha \ll 1$, this completes the proof of the first claim. \\
For the second claim, note that, as before, if $k \geq 1$ and $|t-k| \leq 1/2$ then
$$
|L(1+\alpha+it,f)|^2 \ll (\log(2k))^2 \left|\frac{L(1+\alpha+it,f)}{\zeta(1+\alpha+it)}\right|^2.
$$
Inserting this into the integral before taking the maximum, we obtain instead that
$$
\mc{I}_2(y;0,T) \ll \alpha^{-1} \max_{|t| \leq 1/2} \left|\frac{L(1+\alpha+it,f)}{\zeta(1+\alpha+it)}\right|^2 + \sum_{1 \leq k \leq T-1/2} \frac{(\log(2k))^6}{k^2} \max_{|t-k| \leq 1/2} \left|\frac{L(1+\alpha+it,f)}{\zeta(1+\alpha+it)}\right|^2.
$$
This implies the second claim, similarly to Lemma \ref{lem:I1y}.
% Applying this together with \eqref{eq:majtoZeta} and \eqref{lem:LtoZetaLinf} for each $0 \leq k \leq T-1/2$, we complete the proof of the claim.
\end{proof}
\begin{proof}[Proof of Proposition \ref{prop:IyBd}]
This follows on combining Lemmas \ref{lem:I1yTail}, \ref{lem:I1y}, \ref{lem:I2yTail} and \ref{lem:I2y} in \eqref{eq:Idecomp}.
\end{proof}
\subsection{Completing the proof}
To complete the proof of Theorem \ref{thm:HalDeltaRefined} we need the following simple claim, the proof of which follows immediately from the method of proof of \cite[(III.4.50)]{Ten}.
\begin{lem}\label{lem:monoBdd}
Let $A > 0$ and let $(a_p)_p \subseteq \mb{C}$ be a sequence with $|a_p| \leq A$ for all $p$. Define
$$
\mc{A}(x;T) := \max_{|t| \leq T} \sum_{p \leq x} \frac{\text{Re}(a_p p^{-it})}{p}.
$$
Then there is a constant $C = C(A) > 0$ such that for any $x_1 \leq x_2$ we have
$$
\mc{A}(x_1;T) \leq \mc{A}(x_2;T) + C.
$$
\end{lem}
We will also need the following, which will help clean up some of our estimates. 
\begin{lem} \label{lem:Limit}
Let $f: \mb{N} \ra \mb{U}$ be multiplicative. Then $H_1(x) \gg 1$ in \eqref{eq:LftoMg}.
% There exists a constant $C > 0$ such that for any $T \geq 1$ we have
% $$
% M(x;T) \leq C.
% $$
\end{lem}
\begin{proof}
We have the obvious lower bound
\begin{align*}
\log H_1(x) &= -\min_{|t| \leq 1/2} \sum_{p \leq x} \frac{\text{Re}((1-f(p))p^{-it})}{p} \\
&\geq -\text{Re}\int_{-1/2}^{1/2} \sum_{p \leq x} \frac{(1-f(p))p^{-it}}{p} dt = -\text{Re}\sum_{p \leq x} \frac{1-f(p)}{p}  \int_{-1/2}^{1/2} p^{-it} dt \\
&= -\text{Re}\sum_{p \leq x}\frac{1-f(p)}{p} \frac{\sin(\tfrac{1}{2}\log p)}{\tfrac{1}{2}\log p} \geq -\sum_{p \leq x} \frac{|1-f(p)|}{p} \min\left\{1,\frac{2}{\log p}\right\}.
\end{align*}
Since $\log p \geq 2$ for all $p \geq 7$, say, we get the lower bound
$$
\log H_1(x) \geq - 2\sum_{p \geq 7} \frac{|1-f(p)|}{p\log p} + O(1) \geq -C
$$
for some $C > 0$. The claim now follows on exponentiating.
\end{proof}

Next, in order to obtain the stronger Theorem \ref{thm:HalDeltaRefined} for general real-valued, bounded multiplicative functions\footnote{Note that when $f$ is not necessarily completely multiplicative, it is not true that $1 \ast f$ is non-negative, e.g., if $f(p^k) = -1$ for some prime $p$ and all $k \geq 1$ then $g(p^2) = -1$.}  we make a standard reduction to the case that $f$ is completely multiplicative. 
\begin{lem} \label{lem:redtoCM}
Assume that Theorem \ref{thm:HalDeltaRefined} holds for all completely multiplicative functions $f: \mb{N} \ra [-1,1]$. Then it holds for all multiplicative functions $f:\mb{N} \ra [-1,1]$. 
\end{lem}
\begin{proof}
Let $f: \mb{N} \ra [-1,1]$ be multiplicative, and let $\tilde{f}$ be the completely multiplicative function defined at primes via $\tilde{f}(p) = f(p)$. Set also $h := (\mu \tilde{f}) \ast f$, so that $h$ is a multiplicative function with $|h(n)| \leq d(n)$ for all $n$. We observe that as $\tilde{f}$ is completely multiplicative
$$
(\mu \tilde{f}) \ast \tilde{f}(n) = \tilde{f}(n) (\mu \ast 1)(n) = (\mu \ast 1)(n),
$$
and as such,
$$
f(n) = f \ast (\mu \tilde{f}) \ast \tilde{f}(n) = h \ast \tilde{f}(n).
$$
In particular, $h(p) = 0$ for all primes $p$, and thus $h(n) = 0$ unless $n$ is square-full (i.e., $p|n\Rightarrow p^2|n$). Since any such integer may be written as $n = a^2b^3$ for $a,b \in \mb{N}$, we see that
\begin{equation}\label{eq:absConvh}
\sum_{n = 1}^{\infty} \frac{|h(n)|}{n} \leq \sum_{a,b \geq 1} \frac{d(a^2b^3)}{a^2b^3} \leq \left(\sum_{a \geq 1} \frac{d(a^2)}{a^2}\right)\left(\sum_{b \geq 1} \frac{d(b^3)}{b^3}\right) < \infty,
\end{equation}
so the series converges absolutely. Furthermore, by the divisor bound,
\begin{equation}\label{eq:tailforh}
\sum_{n > y} \frac{|h(n)|}{n} \leq \sum_{\max\{a^2,b^3\} > \sqrt{y}} \frac{d(a^2)d(b^3)}{a^2b^3} \ll_{\e} y^{-1/4 + \e}
\end{equation}
for any $\e > 0$. \\
With this in hand, let us now observe that if $D := (\log x)^{10}$ then by the trivial bound $|L_{\tilde{f}}(y)| \leq \log y$ and \eqref{eq:tailforh}, we get
\begin{align} \label{eq:truncftotildef}
L_f(x) = \sum_{d \leq D} \frac{h(d)}{d} L_{\tilde{f}}(x/d) + O\left(\frac{1}{\log x}\right).
\end{align}
Let $\mc{R}_{\tilde{f}}(y;T)$ denote the error term in Theorem \ref{thm:HalDeltaRefined}. Since $ \log x \asymp \log(x/D)$ and $\tilde{f}(p)=f(p)$ for all primes $p$, we have, uniformly over $d \leq D$,
$$
\mc{R}_{\tilde{f}}(x/d;T) \ll \mc{R}_{f}(x;T).
$$
In view of \eqref{eq:absConvh}, Theorem \ref{thm:HalDeltaRefined} applied to each $L_{\tilde{f}}(x/d)$ gives, uniformly over $d \leq D$,
%using that $\log(x/d) \gg \log x$ uniformly over all $d \leq D$ and extending the integral over $[1,x/d]$ to $[1,x]$ in the error term in each application. This gives
\begin{align}\label{eq:Ltildefd}
L_{\tilde{f}}(x/d) &= \left(1+O\left(\frac{1}{\log (x/d)}\right)\right) \tilde{M}_{1 \ast \tilde{f}}(w_0x/d) + O(\mc{R}_{\tilde{f}}(x/d;T))
%O\left(\frac{1}{\log(x/d)} \int_1^{x/d} (H_1(y;\tilde{f}) + H_2(y;T;\tilde{f})) \frac{dy}{y\log y}\right) 
\nonumber\\
&= \left(1+O\left(\frac{1}{\log x}\right)\right) \tilde{M}_{1 \ast \tilde{f}}(w_0x/d) + O\left(\mc{R}_{f}(x;T)\right).
\end{align}
%where here we used the fact that $\tilde{f}(p) = f(p)$ to obtain the same expressions $H_1$ and $H_2$ for $\tilde{f}$ as for $f$. 
Note moreover that by Lemma \ref{lem:Limit}, we have
\begin{equation}\label{eq:1overlog}
\frac{1}{\log x} \ll \frac{1}{\log x} \int_1^x H_1(y;f) \frac{dy}{y\log y} \ll \mc{R}_f(x;T).
\end{equation}
Inserting this and \eqref{eq:Ltildefd} into \eqref{eq:truncftotildef} and using \eqref{eq:absConvh}, we get
$$
L_f(x) = \left(1+O\left(\frac{1}{\log x}\right)\right)\sum_{d \leq D} \frac{h(d)}{d} \tilde{M}_{1\ast \tilde{f}}(w_0x/d) + O\left(\mc{R}_f(x;T)\right).
$$
Finally, we extend the sum over $d$ to $d \leq x$, introducing an acceptable error of size $O((\log x)/D^{1/5}) = O(1/\log x)$. Using \eqref{eq:1overlog} again, this yields
\begin{align*}
\sum_{d \leq D} \frac{h(d)}{d} \tilde{M}_{1\ast \tilde{f}}(w_0x/d) &= \sum_{d \leq w_0x} \frac{h(d)}{d} \cdot \frac{d}{w_0x} \sum_{ab \leq w_0x/d} \tilde{f}(b) + O\left(\frac{1}{\log x}\right) \\
&= \frac{1}{w_0x} \sum_{abd \leq w_0x} h(d) \tilde{f}(b) + O\left(\frac{1}{\log x}\right) = \tilde{M}_{1 \ast h \ast \tilde{f}}(w_0x) + O\left(\mc{R}_f(x;T)\right),
\end{align*}
so that as $g = 1 \ast f = 1 \ast (h \ast \tilde{f})$, Theorem \ref{thm:HalDeltaRefined} now follows for $f$.
\end{proof}

\begin{proof}[Proof of Theorem \ref{thm:HalDeltaRefined}]
%and \ref{thm:HalDelta}]
%As the proofs of both theorems follow the same lines we will focus on Theorem \ref{thm:HalDeltaRefined}, and identify what changes must be made in order to deduce Theorem \ref{thm:HalDelta}.\\
Assume first that $f$ is real-valued. By Lemma \ref{lem:redtoCM}, it suffices to assume that $f$ is completely multiplicative, which we will now assume. \\
Upon rearranging, it suffices to show that
\begin{equation}\label{eq:claimMainThm}
|\Delta(x)|\log x \ll \int_1^x (H_1(y) + H_2(y;T)) \frac{dy}{y\log y} + 1 + |L_f(x)| + \tilde{M}_g(w_0x) + \frac{(\log(2T)) \log x}{T}.
\end{equation}
(Note that the term $1$ on the RHS is bounded by the integral by Lemma \ref{lem:Limit}, but we include it here to make the RHS transparently non-decreasing.)
We may assume that $x \geq x_0$ for any absolute constant $x_0$, otherwise the claim is trivial. Note that the RHS is a non-decreasing function of $x$, so it suffices to bound the LHS when $x \in \mf{S}$, as defined earlier. By Proposition \ref{prop:toAvg} we have in this case that
$$
|\Delta(x)| \ll \frac{1}{\log x} \int_{x_0}^x |\Delta(y)| \frac{dy}{y} + \frac{1+|L_f(x)| + \tilde{M}_g(w_0x)}{\log x}.
$$
Combining this with \eqref{eq:toJ}, we obtain
\begin{equation}\label{eq:beforeJ}
|\Delta(x)| \ll \frac{J(x)}{(\log x)^2} + \frac{1}{\log x} \int_{x_0}^x \frac{J(y)}{y(\log y)^2} du + \frac{1 + |L_f(x)| + \tilde{M}_g(w_0x)}{\log x},
\end{equation}
where $J(y)$ is defined as in \eqref{eq:Jdef}.
Now, let $x_0\leq y \leq x$ and put $\alpha = 1/\log y$ as before. Using \eqref{eq:CStoJ} and Proposition \ref{prop:IyBd},
%\eqref{eq:Idecomp} together with Lemmas \ref{lem:I2y}, \ref{lem:I1yTail} and \ref{lem:I1y}, 
we have
\begin{align}
J(y) &\ll (\log y)^{1/2} \mc{I}(y)^{1/2} \\
%\left( \mc{I}_1(y; 0, T) + \mc{I}_1(y; T,\infty) + \mc{I}_2(y)\right)^{1/2} \nonumber\\
&\ll (\log y)\left(\max_{|t| \leq 1/2} \left|\frac{L(1+\alpha+it,f)}{\zeta(1+\alpha + it)}\right| + \left(\sum_{1 \leq k \leq T-1/2} \frac{(\log(2k))^4}{k^2} \max_{|t-k| \leq 1/2} |L(1+\alpha+it,f)|^2\right)^{1/2} + 1\right) \nonumber\\
&+ (\log y)^{1/2}\frac{(\log (2T))^2}{\sqrt{T}} + (\log y)\frac{(\log (2T))^2}{T} + (\log y)^2 \frac{\log (2T)}{T}.  \label{eq:Tterms}
% \\
% &\ll (\log y) \left(1 + e^{-M(y;T)}\right) + (\log T) \sqrt{\frac{\log y}{T}} + \frac{(\log y)^2}{T}.
\end{align}
By Mertens' theorem, it is easy to show that for every $t \in \mb{R}$,
$$
\max_{|t| \leq 1/2} \left|\frac{L(1+\alpha+it,f)}{\zeta(1+\alpha+it)}\right| \asymp \max_{|t| \leq 1/2}\exp\left( \sum_{p \leq y} \frac{\text{Re}((f(p)-1)p^{-it})}{p}\right) = H_1(y),
$$
and similarly
\begin{align*}
&\sum_{1 \leq k \leq T-1/2} \frac{(\log(2k))^4}{k^2} \max_{|t-k| \leq 1/2} |L(1+\alpha+it,f)|^2 \\
&\ll \sum_{1 \leq k \leq T-1/2} \frac{(\log(2k))^4}{k^2} \max_{|t-k| \leq 1/2}\exp\left(2 \sum_{p \leq y} \frac{\text{Re}(f(p)p^{-it})}{p}\right) = H_2(y;T)^2.
\end{align*}
As $1 \leq T \leq (\log x)^{100}$, the terms in \eqref{eq:Tterms} contribute
\begin{align*}
&\ll \frac{1}{\log x} \int_{x_0}^x\left(\frac{(\log (2T))^2}{\sqrt{T}} \frac{1}{y(\log y)^{3/2}} + \frac{(\log (2T))^2}{T} \frac{1}{y(\log y)} + \frac{\log (2T)}{T} \frac{1}{y}\right) dy \\
&\ll \frac{(\log (2T))^2}{\sqrt{T} \log x} + \frac{(\log (2T))^2 \log\log x}{T \log x} + \frac{\log (2T)}{T} \ll \frac{\log (2T)}{T} + \frac{1}{\log x}
\end{align*}
%for $1 \leq T \leq (\log x)^{100}$. 
%Inserting these bound into the above and using $1 \leq T \leq (\log x)^{100}$, 
to \eqref{eq:beforeJ}. Finally, using the fact that $H_1(y) \asymp H_1(x) \gg 1$ by Lemma \ref{lem:Limit}, and $H_2(x;T) \asymp H_2(y;T)$ uniformly in $\sqrt{x} < y \leq x$, so that
$$
1+ H_1(x) + H_2(x;T) \asymp \int_{\sqrt{x}}^x (H_1(y) + H_2(y;T) \frac{dy}{y \log y} \leq \int_{x_0}^x (H_1(y) + H_2(y;T)) \frac{dy}{y\log y},
$$
we obtain
\begin{align*}
|\Delta(x)| &\ll \frac{1}{\log x} \left(\int_{x_0}^x (H_1(y) + H_2(y;T) + 1) \frac{dy}{y\log y}\right) + \frac{|L_f(x)| + \tilde{M}_g(w_0x)}{\log x} + \frac{\log (2T)}{T} \\
&\ll \frac{1}{\log x} \int_{x_0}^x (H_1(y) + H_2(y;T)) \frac{dy}{y\log y} + \frac{1 + |L_f(x)| + \tilde{M}_g(w_0x)}{\log x} + \frac{\log (2T)}{T},
\end{align*}
This completes the proof of \eqref{eq:claimMainThm}. \\
%Therefore,
%$$
%|\Delta(x)| \ll \frac{1}{\log x} (H_1(x) + H_2(x;T)) %\int_{x_0}^x \frac{dy}{y\log y} + \frac{1 + |L_f(x)| + %\tilde{M}_g(w_0x)}{\log x} + \frac{\log T}{T},
%$$
%and since the integral on the RHS is $\ll \log\log x$, the %first claim follows. \\
To deduce the second claim, note that if $f(p)=0$ for $x^{1-c} < p \leq x$ then the application of Proposition \ref{prop:toAvg} improves to
$$
|\Delta(x)| \ll \frac{1}{\log x} \int_{x^c/2}^x |\Delta(y)| \frac{dy}{y} + \frac{1 + |L_f(x)| + |M_g(x)|}{\log x}.
$$
The remainder of the proof stays the same.
%and in the last step, we obtain
%\begin{align*}
%|\Delta(x)| &\ll \frac{1}{\log x} (H_1(x) + H_2(x;T)) %\int_{x^c}^x \frac{dy}{y\log y} + \frac{1 + |L_f(x)| + %\tilde{M}_g(w_0x)}{\log x} + \frac{\log T}{T} \\
%&\ll \frac{\log(1/c)}{\log x} (H_1(x) + H_2(x;T)) + + %\frac{1 + |L_f(x)| + \tilde{M}_g(w_0x)}{\log x} + \frac{\log %T}{T}
%\end{align*}
%instead. 
\end{proof}
\begin{proof}[Proof of Theorem \ref{thm:HalDeltaGen}]
The proof for general $1$-bounded multiplicative functions is entirely the same as the proof of Theorem \ref{thm:HalDeltaRefined}, except that in \eqref{eq:beforeJ} the rightmost term must be replaced by
$$
\frac{(\log\log x)^4 + |L_f(x)| + |\tilde{M}_g(w_0x)|(\log\log x)^2}{\log x},
$$
according to Proposition \ref{prop:toAvg}.
\end{proof}
\begin{proof}[Proof of Theorem \ref{thm:HalDelta}]
By Lemma \ref{lem:monoBdd}, we have
$$
\int_e^x (H_1(y) + H_2(y;T)) \frac{dy}{y\log y} \ll (H_1(x) + H_2(x;T)) \int_e^x \frac{dy}{y\log y}  \ll (\log\log x) (H_1(x) + H_2(x;T)).
$$
Since $H_1(x) \ll e^{-M(x;T)}$ trivially, it suffices to show that $H_2(x;T) \ll e^{-M(x;T)}$ as well. Indeed, let $1 \leq k \leq T-1/2$ and let $t \in [k-1/2,k+1/2]$. Then, using \eqref{eq:logzeta}, 
$$
\exp\left(\sum_{p \leq x} \frac{\cos(t\log p)}{p}\right) \asymp |\zeta(1+1/\log x + it)| \gg (\log(2k))^{-1}.
$$
It follows that
\begin{align*}
H_2(x;T)^2 &\ll \sum_{1 \leq k \leq T-1/2} \frac{(\log(2k))^6}{k^2} \max_{|t-k| \leq 1/2} \exp\left(2\sum_{p \leq x} \frac{(f(p)-1) \cos(t\log p)}{p}\right) \\
&\ll \max_{1/2 \leq |t| \leq T} \exp\left(2 \sum_{p \leq x} \frac{(f(p)-1) \cos(t\log p)}{p}\right) \leq e^{-2M(x;T)},
\end{align*}
and Theorem \ref{thm:HalDelta} follows.
\end{proof}
\section{On the random Tur\'{a}n problem} \label{sec:random}
We begin by obtaining a slight variant of Theorem \ref{thm:HalDeltaRefined}, which improves on the error term at the cost of additional main terms. The restriction to the smooth support is crucial for our applications.
\begin{prop} \label{prop:passToSmooth}
Let $f: \mb{N} \ra [-1,1]$ be multiplicative. Let $\theta \in [0,1/2]$ and let $f_\theta$ be the multiplicative function defined as $f_\theta(p^k) := f(p^k)$ if $p \leq x^{1-\theta}$, $f_\theta(p^k) = 0$ for $p > x^{1-\theta}$. Define $g_\theta := 1 \ast f_\theta$. If $x$ is sufficiently large and $1 \leq T \leq (\log x)^{100}$ then
% \begin{align} \label{eq:first}
% L_f(x) = \sum_{x^{1-\theta} < p \leq x} \frac{f(p)}{p} L_f(x/p) + \tilde{M}_{g_\theta}(w_0x) + O\left(\frac{e^{-M(x^{1-\theta};T)}}{\log x}\right).
% \end{align}
% In fact, we have the more precise bound
\begin{align}\label{eq:second}
L_f(x) &= \left(1+O\left(\frac{1}{\log x} \right)\right)\left(\sum_{x^{1-c} < p \leq x} \frac{f(p)}{p} L_f(x/p) + \tilde{M}_{g_\theta}(w_0x)\right) \\
&+ O\left(\frac{1}{\log x}\int_{x^\theta/2}^x (H_1(y) + H'_2(y;T))\frac{dy}{y\log y} + \frac{\log(2T)}{T}\right),
\end{align}
where for $1 \leq y \leq x$, $H_1(y)$ and $H_2'(y;T)$ are defined as
\begin{align*}
H_1(y) &:= \exp\left(\max_{|t| \leq 1/2} \sum_{p \leq y} \frac{(f(p)-1)\cos(t\log p)}{p}\right), \\
H'_2(y;T)^2 &:= \sum_{1 \leq k \leq T} \frac{(\log(2k))^6}{k^2} \exp\left(2\max_{|t-k| \leq 1/2} \sum_{\exp((\log(2k))^2) < p \leq y} \frac{f(p)\cos(t\log p)}{p}\right).
\end{align*}
\end{prop}
\begin{proof}
If $0 \leq \theta < \tfrac{1}{10\log x}$ then as $x^{1-\theta} \geq x/2$ we have $L_f(x/p) = 1$ for all $p > x^{1-\theta}$ and thus 
$$
\sum_{x^{1-\theta} < p \leq x} \frac{f(p)}{p} L_f(x/p) \ll \frac{1}{\log x}.
$$
The claim in this case follows from the first statement of Theorem \ref{thm:HalDeltaRefined} and Lemma \ref{lem:Limit}. \\
In the sequel, we assume that $\theta \geq \frac{1}{10\log x}$. 
%[Proof sketch]
%The proof of both of the stated estimates begin by
Splitting $L_f(x)$ according to whether or not $P^+(n) > x^{1-\theta}$, we obtain
$$
L_f(x) = \sum_{\ss{mp \leq x \\ x^{1-\theta} < p \leq x}} \frac{f(m)f(p)}{mp} + \sum_{\ss{n \leq x \\ P^+(n) \leq x^{1-\theta}}} \frac{f(n)}{n} = \sum_{x^{1-\theta} < p \leq x} \frac{f(p)}{p} L_f(x/p) + L_{f_\theta}(x).
$$
By the second statement of Theorem \ref{thm:HalDeltaRefined}, 
%when $T \geq (\log x) \log\log x$ we have
\begin{align*}
L_{f_\theta}(x) = \left(1+O\left(\frac{1}{\log x}\right)\right)M_{g_\theta}(w_0x) + O\left(\frac{1}{\log x}\int_{x^\theta}^x \frac{H_1(y) + H_2(y;T)}{y\log y} dy + \frac{\log T}{T}\right).
\end{align*}
Furthermore, if we set $y_k := \exp((\log(2k))^2)$ then since 
$$
\sum_{p \leq y_k} \frac{f(p) \cos(t\log p)}{p} \leq \log\log y_k + O(1) = 2\log\log(2k) + O(1),
$$
it is immediately clear that $H_2(x;T) \ll H_2'(x;T)$. The claim now follows.
\end{proof}
In the sequel we take $\theta = 1/2$ and let $T := (\log x)^2$. Using Lemma \ref{lem:monoBdd}, the error term in the previous proposition becomes
\begin{equation}\label{eq:remloglog}
\ll \frac{H_1(x) + H_2'(x;T)}{\log x} \int_{x^{1/2}/2}^x \frac{dy}{y\log y} \ll \frac{H_1(x) + H_2'(x;T)}{\log x}.
\end{equation}
Now, let $\mbf{f}$ denote a Rademacher random completely multiplicative function. We have the following simple result, which shows that the additional main terms arising in Proposition \ref{prop:passToSmooth} are small when applied to $\mathbf{f}$.
%Let
%$$
%S(x) := \sum_{x^{1/2} < p \leq x} \frac{\mbf{f}(p)}{p} L_{\mbf{f}}(x/p),
%$$
\begin{lem}\label{lem:extraMTs}
There exists a constant $c > 0$ such that
$$
\mb{P}\left(\left|\sum_{x^{1/2} < p \leq x} \frac{\mbf{f}(p)}{p} L_{\mbf{f}}(x/p)\right| \geq \frac{1}{\log x}\right) \ll \exp\left(-c\frac{\sqrt{x}}{(\log x)^3}\right).
$$
\end{lem}
\begin{proof}
%We may condition on all of the primes less than $x^{1/2}$, so that $L_{\mbf{f}}(x/p)$ is completely determined for all $x^{1/2} < p \leq x$. If we 
For convenience, set
$$
\sum_{x^{1/2} < p \leq x} \frac{\mbf{f}(p)}{p} L_{\mbf{f}}(x/p) =: \sum_{x^{1/2} < p \leq x} X_p.
$$
%then since each $\mbf{f}(p)$ has mean $0$ and variance $1$ %and are mutually independent, $\mb{E}(S(x)) = 0$ and
%$$
%\text{Var}(S(x)) = \sum_{x^{1/2} < p \leq x} \frac{\text{Var}%(L_f(x/p))^2}{p^2}.
%$$
Then as $|L_{\mbf{f}}(x/p)| \leq \log(x/p)$, we have 
$$
-\frac{\log(x/p)}{p} \leq X_p \leq \frac{\log(x/p)}{p}
$$
for all $x^{1/2} < p \leq x$. 
%a simple calculation using the prime number theorem yields
%$$
%\text{Var}(S(x)) \ll \frac{\log x}{\sqrt{x}}.
%$$
% then by the tower property of conditional expectation, and the mutual independence of $f(p)$ a, we have
% $$
% \mb{E}(S(x)) = \mb{E} \mb{E}(S(x) | (f(p))_{p \leq x^{1/2}}) = \mb{E} 
% $$
By Hoeffding's inequality \cite[Thm. 2]{Hoeff}, 
\begin{align*}
\mb{P}\left(\left|\sum_{x^{1/2} < p \leq x} \frac{\mbf{f}(p)}{p} L_{\mbf{f}}(x/p)\right| \geq \frac{1}{\log x}\right) &\ll \exp\left(-\frac{1}{2(\log x)^2} \cdot \left(\sum_{x^{1/2} < p \leq x} \frac{(\log(x/p))^2}{p^2}\right)^{-1}\right)\\
&\ll \exp\left(-c\frac{\sqrt{x}}{(\log x)^3}\right),
\end{align*}
as claimed.
\end{proof}
In particular, in light of \eqref{eq:remloglog} we obtain from Proposition \ref{prop:passToSmooth} and the previous lemma that with probability\footnote{If we don't fix $\theta = 1/2$ then this calculation gives an exceptional probability of size $\exp\left(-O(x^{1-\theta-o(1)})\right)$, which is consistent with the possibility that any $\beta < 1$ is admissible in Theorem \ref{thm:random}.} $\geq 1-O(e^{-x^{1/2-o(1)}})$, we have
\begin{equation} \label{eq:onlyMghalf}
L_{\mbf{f}}(x) = \left(1+O\left(\frac{1}{\log x}\right)\right)\tilde{M}_{\mbf{g}_{1/2}}(w_0x) + O\left(\frac{H_1(x) + H_2'(x;T)}{\log x}\right),
\end{equation}
where $\mbf{g}_{1/2} = 1 \ast \mbf{f}_{1/2}$. \\
In what follows, we will need the following lemma of Kerr and the first author \cite{KeKl}.
\begin{lem} \label{lem:KeKl}
Let $\e > 0$ be small, $f: \mb{N} \ra [-1,1]$ completely multiplicative and $ g = 1 \ast f$ as before. Given $\delta \in (0,1)$ define
$$
\mc{P}_{\delta} := \{p \in \mb{P} : f(p) \geq -\delta\},
$$
and assume that for some
$$
\frac{40000}{\e^2} \leq v \leq \frac{\log x}{1000\log\log x}
$$
we have
$$
\sum_{\ss{p \in \mc{P}_{\delta}\cap [x^{1/v},x]}} \frac{1}{p} \geq 1+\e.
$$
Then we have
$$
\tilde{M}_g(x) \gg \e^2 \left(\frac{1-\delta}{v}\right)^{v(1+o(1))/e} \exp\left(\sum_{p \leq x} \frac{f(p)}{p}\right).
$$
\end{lem}
Our next result is a minor variant of Proposition 2 in \cite{Kuch}. Since we need the version stated here, we give a full proof.
\begin{lem}
There are constants $c,\beta > 0$ such that
$$
\mb{P}\left(\tilde{M}_{\mbf{g}_{1/2}}(w_0x) < c \exp\left(\sum_{p \leq x} \frac{\mbf{f}(p)}{p}\right)\right) \ll \exp(-x^{\beta}).
$$
\end{lem}
\begin{proof}
Fix $\e = 0.01$, say, and define
$$
v_0 := \frac{1}{2}e^{2+4\e}, \quad v := \max\{40000\e^{-2},v_0\}.
$$
We aim to apply Lemma \ref{lem:KeKl} with $\delta = 1/2$.  Since $\mbf{f}_{1/2}$ takes values in $\{-1,0,+1\}$, 
%$\{-1,0,+1\}$: if $\e > 0$,
% $$
% \frac{40000}{\e^2} \leq v \leq \frac{\log x}{1000\log \log x}
% $$
% %and $\mc{P}_{\delta} := \{p \in \mb{P} : f(p) \geq -\delta\}$ satisfies
% and we have
% $$
% \sum_{\ss{x^{1/v} < p \leq x \\ f(p) \neq -1}} \frac{1}{p} \geq 1+\e
% $$
% then we have
% $$
% \tilde{M}_{1\ast f}(x) \gg \e^2 \left(\frac{1}{2v}\right)^{v(1+o(1))/e} \exp\left(\sum_{p \leq x} \frac{f(p)}{p}\right).
% $$
we see that the lower bound 
$$
\tilde{M}_{1\ast \mbf{f}_{1/2}}(x) \gg \e^2 \left(\frac{1}{2v}\right)^{v(1+o(1))/e} \exp\left(\sum_{p \leq x} \frac{\mbf{f}(p)}{p}\right)
$$
holds outside of the event
$$
\sum_{\ss{x^{1/v} < p \leq x^{1/2} \\ \mbf{f}(p) = +1}} \frac{1}{p} \leq 1-\log 2 +\e.
$$
%Let $v_0 := \tfrac{1}{2}e^{2+4\e}$ and let $v:= \max\{40000\e^{-2},v_0\}$. 
Applying Hoeffding's inequality again, we have
$$
\mb{P}\left(\left|\sum_{x^{1/v_0} < p \leq x^{1/2}} \frac{\mbf{f}(p)}{p}\right| \geq \e\right) \ll \exp\left(-x^{1/v_0}\right).
$$
Moreover, since $2(1-\log 2 + \e) - \log (v_0/2) \leq -\e$, when $x$ is sufficiently large we have
\begin{align*}
\mb{P}\left(\sum_{\ss{x^{1/v} < p \leq x^{1/2} \\ f(p) = +1}}\frac{1}{p} \leq 1-\log 2 + \e \right) &\leq \mb{P}\left( \sum_{x^{1/v_0} < p \leq x^{1/2}} \frac{1+\mbf{f}(p)}{p} \leq 2(1-\log 2 + \e)\right) \\
&\leq \mb{P}\left(\left|\sum_{x^{1/v_0} < p \leq x^{1/2}} \frac{\mbf{f}(p)}{p}\right| \geq \e\right).
\end{align*}
It follows that, outside of a probability $O(\exp(-x^{1/v_0}))$ event, we have
$$
\sum_{\ss{x^{1/v} < p \leq x \\ \mbf{f}_{1/2}(p) \neq  -1}} \frac{1}{p} \geq 1+\e,
$$
and thus, as $v \asymp_{\e} 1$ and $\e$ is fixed, by Lemma \ref{lem:KeKl},
$$
\tilde{M}_{\mbf{g}_{1/2}}(w_0 x) \gg \tilde{M}_{\mbf{g}_{1/2}}(x) \gg \exp\left(\sum_{p \leq x} \frac{\mbf{f}_{1/2}(p)}{p}\right) \gg \exp\left(\sum_{p \leq x} \frac{\mbf{f}(p)}{p}\right).
$$
The claim follows with $\beta = 1/v_0$.
%on fixing $\e = 0.01$ say.
\end{proof}

% By Prop. 3.1 of \cite{KeKl}, the following holds: if $\e > 0$, $\delta \in (0,1)$,
% $$
% \frac{40000}{\e^2} \leq v \leq \frac{\log x}{1000\log \log x}
% $$
% and $\mc{P}_{\delta} := \{p \in \mb{P} : f(p) \geq -\delta\}$ satisfies
% $$
% \sum_{p \in \mc{P}_{\delta} \cap [x^{1/v},x]} \frac{1}{p} \geq 1+\e
% $$
% then we have
% $$
% \tilde{M}_g(x) \gg \e^2 \left(\frac{1-\delta}{v}\right)^{v(1+o(1))/e} \exp\left(\sum_{p \leq x} 
% $$
% %Further, as Kucheriaviy proved, one has
% $$
% \tilde{M}_{\mbf{g}_{1/2}}(w_0x) \gg \tilde{M}_{\mbf{g}_{1/2}}(x) \gg \exp\left(\sum_{p \leq x} \frac{f(p)}{p}\right)
% $$
% with probability $1 - O(e^{-x^{\beta}})$ for some\footnote{In fact, the situation is better here, since $g(p) = 1$ for all $x^{1/2} < p \leq x$. A back of the envelope calculation shows that $\beta$ can be chosen as roughly $2/e^2 > 1/4$.} $\beta > 0$. 
%\end{proof}
Combining the previous lemma with \eqref{eq:onlyMghalf}, we see that there are $c_1, c_2, \beta > 0$ such that with probability $1-O(e^{-x^\beta})$, we have
\begin{equation} \label{eq:highprobLowBd}
L_{\mbf{f}}(x) \geq c_1\exp\left(\sum_{p \leq x} \frac{\mbf{f}(p)}{p}\right) - c_2 \frac{H_1(x) + H_2'(x;T)}{\log x}.
\end{equation}
Now, fix $C := 2c_2/c_1$. Clearly, if $L_{\mbf{f}}(x) < 0$ then one of the events
%Thus, in order to obtain a bound on the probability that $L_{\mbf{f}}(x) < 0$, it suffices to show that each of the events
\begin{align}
\exp\left(\sum_{p \leq x} \frac{\mbf{f}(p)}{p}\right) &< \frac{C}{\log x} \exp\left(\max_{|t| \leq 1/2} -\sum_{p \leq x} \frac{(1-\mbf{f}(p)) \cos(t\log p)}{p}\right) \label{eq:tsmall}\\
\exp\left(2\sum_{p \leq x} \frac{\mbf{f}(p)}{p}\right) &< \frac{C^2}{(\log x)^2}\sum_{1 \leq k \leq T} \frac{(\log(2k))^6}{k^2} \exp\left(2 \max_{|t-k| \leq 1/2} \sum_{\exp((\log(2k))^2) < p \leq x} \frac{\mbf{f}(p)\cos(t\log p)}{p}\right). \label{eq:tbig}
\end{align}
must hold. The next proposition shows that each of these events has acceptable probability.
\begin{prop} \label{prop:withEuler}
There is a constant $\beta > 0$, depending only on $C$, such that if $x$ is sufficiently large then
$$
\mb{P}(\eqref{eq:tsmall} \text{ or } \eqref{eq:tbig} \text{ holds}) \ll \exp(-x^\beta).
$$
\end{prop}
\begin{proof}[Proof of Theorem \ref{thm:random} assuming Proposition \ref{prop:withEuler}]
This is immediate from \eqref{eq:highprobLowBd} and the definitions of $H_1(x)$ and $H_2(x;T)$.
\end{proof}
To prove Proposition \ref{prop:withEuler} we will need a few lemmas. The first, which is a simple consequence of Lemma \ref{lem:TenEst}, ensures that our prime sums do not lose a constant factor over the trivial bound from Mertens' theorem.
% Assuming that $|t_0| \leq 1$ (or in fact, is bounded by some absolute constant) then this can be simplified further. 
\begin{lem}\label{lem:uppBdcosSum}
For any $x\geq 2$ we have
\begin{equation}\label{eq:maxt1}
\max_{|t| \leq 1/2} \sum_{p \leq x} \frac{1-\cos(t\log p)}{p} \leq \log\log x + O(1),
\end{equation}
and more generally if $k \in \mb{Z} \bk \{0\}$ and $y_k := \exp((\log(2|k|))^2)$ then
\begin{equation}\label{eq:maxtk}
\max_{|t-k| \leq 1/2} \sum_{y_k < p \leq x} \frac{1-\cos(t\log p)}{p} \leq \log \log x + O(1).
\end{equation}
\end{lem}
\begin{proof}
The bound trivially holds if $t = 0$, so assume henceforth that $t \neq 0$. 
%Both parts of the lemma rely on Lemma \ref{lem:TenEst}. \\
Let us begin with \eqref{eq:maxt1}, and assume that $0 < |t| \leq 1/2$. 
As $\cos(t\log p) \geq 0$ for all $p \leq e^{1/|t|}$, we see that if $x \leq e^{1/|t|}$ then
$$
\sum_{p \leq x} \frac{1-\cos(t\log p)}{p} \leq \sum_{p\leq x} \frac{1}{p} = \log \log x + O(1).
$$
%and the above argument we have 
On the other hand, if $x \geq e^{1/|t|}$ then setting $w := e^{1/|t|}$ and applying the previous case in the range $p \leq w$ and Lemma \ref{lem:TenEst} with $\phi(u) = \cos(u)$ in the range $w < p \leq x$, we obtain
\begin{align*}
\sum_{p \leq x} \frac{1-\cos(t\log p)}{p} &= \sum_{p \leq w} \frac{1-\cos(t\log p)}{p} + \sum_{w < p \leq x} \frac{1-\cos(t\log p)}{p} \\
&\leq \log(1/|t|) + \log(|t|\log x) + O(1) = \log\log x + O(1). 
\end{align*}
This implies \eqref{eq:maxt1}. \\
Next, we consider \eqref{eq:maxtk}, and assume that $|t-k| \leq 1/2$. In particular, $|t| \leq k + 1$, and therefore 
$$
(1+|t|)/\exp(\sqrt{\log y_k}) \ll 1.
$$  
Since $y_k > 3/2$, say, Lemma \ref{lem:TenEst} therefore directly yields 
$$
\sum_{y_k < p \leq x} \frac{1-\cos(t\log p)}{p} = \log\left(\frac{\log x}{\log y_k}\right) + O(1) \leq \log\log x + O(1),
$$
as claimed.
\end{proof}
Following the argument of Angelo and Xu (see Lemma 2.4 of \cite{AngXu}), we next prove the following result.
\begin{lem}\label{lem:AngXu}
Let $\delta \gg \frac{1}{\log x}$ and $\eta \in \{-1,+1\}$. Then there is a constant $c > 0$ such that the following statements hold: \\
(a) We have
$$
\max_{|t| \leq 1/2} \mb{P}\left(\eta \sum_{p \leq x} \frac{\mbf{f}(p)(1-\cos(t\log p))}{p} \geq \log(1/\delta)\right) \ll \exp(\exp(-c/\delta)).
$$
%holds for both $\eta \in \{-1,+1\}$. \\
(b) Let $k \in \mb{Z}\bk\{0\}$ and $y_k := \exp(\log(2k)^2)$. Then, 
%uniformly over $|t-k| \leq 1/2$,
$$
\max_{|t-k| \leq 1/2}\mb{P}\left(\eta \sum_{y_k < p \leq x} \frac{\mbf{f}(p)(1-\cos(t\log p))}{p} \geq \log(1/\delta)\right) \ll \exp(\exp(-c/\delta)).
$$
\end{lem}
\begin{proof}
%[Proof of Lemma \ref{lem:AngXu}]
Note that since $-\mbf{f}(p)$ has the same distribution as $\mbf{f}(p)$ for all $p$, each of the above claims with $\eta = 1$ follow immediately from the claim with $\eta = -1$. Thus, in the sequel we will assume that $\eta = -1$. \\ (a) The claim is vacuously true with $t = 0$, so in the sequel we shall assume that $0 < |t| \leq 1/2$. \\
Given $t \in \mb{R}$, define
\begin{equation}\label{eq:ZtoExp}
Z(t) := \prod_{p \leq x} \frac{p-\mbf{f}(p)}{p-\mbf{f}(p)\cos(t\log p)} = \exp\left(-\sum_{p \leq x} \frac{\mbf{f}(p)(1-\cos(t\log p))}{p} + C\right),
\end{equation}
for some constant $C > 0$.
% it follows immediately that, also,
% $$
% \mb{P}\left(\sum_{p \leq x} \frac{f(p)(1-\cos(t\log p))}{p} \leq -\log(1/\delta)\right) \ll \exp(\exp(-c/\delta)).
% $$
% The argument is essentially the same as in Angelo-Xu. \\
Let $2 \leq z \leq x$ be a parameter to be chosen later, and split $Z(t) = Z_s(t)Z_l(t)$, where 
$$
Z_s(t) := \prod_{p \leq z} \frac{p-\mbf{f}(p)}{p-\mbf{f}(p)\cos(t\log p)}, \quad Z_l(t) := \prod_{z < p \leq x} \frac{p-\mbf{f}(p)}{p-\mbf{f}(p)\cos(t\log p)}.
$$
We will bound $Z_s(t)$ in $L^\infty$, and compute moments of $Z_l(t)$. For the former, we use \eqref{eq:maxt1} to get
% we observe that as $\cos(t\log p) \geq 0$ for all $p \leq e^{\pi/2|t|}$, we see that if $z \leq e^{1/|t|}$ then
% $$
% \sum_{p \leq z} \frac{1-\cos(t\log p)}{p} \leq \sum_{p\leq z} \frac{1}{p} = \log \log z + O(1).
% $$
% Next, appealing to Lemma III.4.13 of Tenenbaum's book and the above argument we have that for any $z \geq e^{1/|t|}$, 
% $$
% \sum_{e^{1/|t|} < p \leq z} \frac{1-\cos(t\log p)}{p} = \log(|t|\log z) + O(1). 
% $$
% (If we want to use this with $|t| \leq T$ and $T \geq 1$ then in order to get $O(1)$ here, we need to begin the sum at $\exp((\log T)^2)$.) \\
% It thus follows that whenever $0 < |t| \leq 1$, we always have
\begin{equation}\label{eq:uppBdBeginning}
\left|-\sum_{p \leq z} \frac{\mbf{f}(p)(1-\cos(t\log p))}{p}\right| \leq \sum_{p \leq z} \frac{1-\cos(t\log p)}{p} \leq \log \log z + O(1),
\end{equation}
whence $Z_s(t) \leq C_0 \log z$ for some $C_0 > 0$. \\ 
Next, let $\lambda > 0$. Then we have
\begin{align*}
\mb{E}Z_l(t)^\lambda = \prod_{z < p \leq x} \frac{1}{2}\left(\left(1+\frac{1-\cos(t\log p)}{p+\cos(t\log p)}\right)^{\lambda} + \left(1-\frac{1-\cos(t\log p)}{p-\cos(t\log p)}\right)^{\lambda}\right).
\end{align*}
Fix $z < p \leq x$ for the moment. Using the inequality $(1+u)^\lambda \leq e^{u\lambda}$, valid for all $u \in \mb{R}$, the $p$th factor in the above product is
\begin{align*}
\leq \frac{1}{2}\exp\left(\frac{\lambda(1-\cos(t\log p))}{p+\cos(t\log p)}\right) +  \frac{1}{2}\exp\left(-\frac{\lambda(1-\cos(t\log p))}{p-\cos(t\log p)}\right).
\end{align*}
Taylor expanding the exponentials, and splitting the series according to even and odd exponents, the latter is
\begin{align*}
&\frac{1}{2}\sum_{m \geq 0} \frac{\lambda^{2m}(1-\cos(t\log p))^{2m}}{(2m)!}\left(\frac{1}{(p+\cos(t\log p))^{2m}} + \frac{1}{(p-\cos(t\log p))^{2m}}\right) \\
&+ \frac{1}{2}\sum_{m \geq 0} \frac{\lambda^{2m+1}(1-\cos(t\log p))^{2m+1}}{(2m+1)!}\left(\frac{1}{(p+\cos(t\log p))^{2m+1}} - \frac{1}{(p-\cos(t\log p))^{2m+1}}\right) \\
&=: S_1(p) + S_2(p).
\end{align*}
Since $p\pm \cos(t\log p) \geq p-1$, we always have
$$
S_1(p) \leq \sum_{m \geq 0} \frac{\lambda^{2m}(1-\cos(t\log p))^{2m}(p-1)^{-2m}}{(2m)!} \leq \exp\left(\frac{\lambda^2(1-\cos(t\log p))^2}{(p-1)^2}\right).
$$
Next, if $\cos(t\log p) \geq 0$ then $S_2(p) \leq 0$. Otherwise, using the simple inequality
$$
1-(1-u)^\alpha \leq u\alpha,
$$
valid when $\alpha \geq 1$ and $u \geq 0$, we have
%(in all cases)
\begin{align*}
S_2(p) &= \frac{1}{2}\sum_{m \geq 0} \frac{\lambda^{2m+1}(1-\cos(t\log p))^{2m+1}(p+\cos(t\log p))^{-2m-1}}{(2m+1)!}\left(1-\left(1-\frac{2|\cos(t\log p)|}{p-\cos(t\log p)}\right)^{2m+1}\right) \\
&\leq \frac{\lambda|\cos(t\log p)|(1-\cos(t\log p))}{(p-1)^2}\sum_{m \geq 0} \frac{\lambda^{2m}(1-\cos(t\log p))^{2m}(p-1)^{-2m}}{(2m)!} \\
&\leq \frac{2\lambda}{(p-1)^2} \exp\left(\frac{\lambda^2(1-\cos(t\log p))^2}{(p-1)^2}\right).
\end{align*}
As the RHS in this last estimate is non-negative, the above bound actually holds regardless of the sign of $\cos(t\log p)$. \\
Combining these two bounds and using \eqref{eq:uppBdBeginning}, we therefore find that if we choose $z := 10\lambda$ then
\begin{align*}
\mb{E}Z(t)^\lambda &\leq (C_0 \log(10 \lambda))^{\lambda} \exp\left(\sum_{10 \lambda < p \leq x} \frac{2\lambda + \lambda^2(1-\cos(t\log p))^2}{(p-1)^2}\right) \\
&\leq \exp\left(\lambda \log(c_0 \log\lambda)\right),
\end{align*}
for a sufficiently large constant $c_0 > 0$. Thus, writing $C = \log c$ in \eqref{eq:ZtoExp}, the event that
$$
-\sum_{p \leq x} \frac{f(p)(1-\cos(t\log p)}{p} \geq \log (1/\delta)
$$
has probability
$$
\mb{P}(Z(t) \geq c\delta^{-1}) \leq (c/\delta)^{-\lambda} \mb{E}Z(t)^{\lambda} \leq \exp\left(\lambda\left(\log(c_0\log \lambda) - \log(c/\delta)\right)\right),
%= \exp\left(-\exp\left(\frac{c'}{\delta}\right)\right),
$$
for some $c > 0$. Taking $\log \lambda = \frac{c}{2c_0\delta}$ gives the claim, uniformly over $|t| \leq 1/2$ as claimed. \\
(b) The proof of (b) is essentially the same as (a), but where 
$$
Z_s(t) = \prod_{y_k < p \leq z} \frac{p-\mbf{f}(p)}{p-\mbf{f}(p)\cos(t\log p)},
$$
and to bound it in $L^\infty$ we apply \eqref{eq:maxtk} in place of \eqref{eq:maxt1}. We leave the proof to the reader.
%(this argument holds for $t \neq 0$, but when $t = 0$ the claim is vacuously true). 
\end{proof}
Finally, we give the following simple technical device to recover uniform tail bounds from pointwise ones.
\begin{lem}\label{lem:maxTail}
Let $\delta \in (0,1/2)$, let $I$ be a closed and bounded interval of positive length, and let $(Y_t)_{t \in I}$ be a real-valued random process for which there is a constant $C > 0$ such that if $|t-s| \leq \delta$ then
$$
|Y_t - Y_s| \leq C.
$$
%Assume that there is a constant $c_0 > 0$ such that for any $|t| \leq T$ and $M > 1$,
Then for any $M > 0$,
% $$
% \mb{P}(Y_t \geq \log M) \leq \exp(\exp(-c_0 M)).
% $$
% Then there is a constant $c_1 > 0$ such that
$$
\mb{P}\left(\max_{t \in I} Y_t \geq \log M\right) \ll |I|\delta^{-1} \max_{t \in I} \mb{P}\left(Y_t \geq \log M - C\right). 
$$
\end{lem}
\begin{proof}
Set $K := \lceil |I| \delta^{-1}\rceil$ and divide $I = [a,b]$ into $K$ subintervals 
%$(I_k)_{1 \leq k \leq K}$ 
$$
I_k := [a + (k-1)\delta, \min\{a+k\delta,b\}],
$$ 
with all but one having length $\delta$ (and the remaining one with length $\leq \delta$).
By the union bound, we have
$$
\mb{P}\left(\max_{t \in I} Y_t \geq \log M\right) \leq K \max_{1 \leq k \leq K} \mb{P}\left(\max_{t \in I_k} Y_t \geq \log M\right).
$$
If $t \in I_k$ for some $k \geq 1$ then
$|Y_t-Y_{a+(k-1)\delta}| \leq C$, and thus 
%writing $C = e^c$, we get
$$
\mb{P}\left(\max_{t \in I_k} Y_t \geq \log M\right) \leq \mb{P}(Y_{a+(k-1)\delta} \geq \log M - C) \leq \max_{t \in I_k} \mb{P}(Y_t \geq \log M - C).
$$
The claim now follows on taking the maximum over $k$. 
\end{proof}
\begin{proof}[Proof of Proposition \ref{prop:withEuler}]
Let us begin by bounding the probability that \eqref{eq:tsmall} holds. Fix $t_0 \in [-1/2,1/2]$ to maximise the RHS of \eqref{eq:tsmall}. Taking logarithms and rearranging, note that we may rewrite this event as
\begin{equation}\label{eq:lowProbevent}
\sum_{p \leq x} \frac{(1-\mbf{f}(p))(1-\cos(t_0\log p))}{p} > \sum_{p \leq x} \frac{1-\mbf{f}(p)}{p} + \sum_{p \leq x} \frac{\mbf{f}(p)}{p} + \log\log x + \log C = 2\log\log x + C'.
\end{equation}
By Lemma \ref{lem:uppBdcosSum}, the LHS of \eqref{eq:lowProbevent} is
$$
\leq -\sum_{p \leq x} \frac{\mbf{f}(p)(1-\cos(t_0 \log p))}{p} + \log\log x + O(1).
$$
Thus, the probability that \eqref{eq:lowProbevent} holds is
\begin{equation}\label{eq:probmax}
\leq \mb{P}\left(\max_{|t| \leq 1/2} -\sum_{p \leq x} \frac{\mbf{f}(p)(1-\cos(t\log p))}{p} > \log((\log x)/c')\right),
\end{equation}
for some $c' > 0$. \\
Set $I = [-1/2,1/2]$, and define the process
%Define the random process
$$
Y_t := -\sum_{p \leq x} \frac{\mbf{f}(p)(1-\cos(t\log p))}{p}, \quad t \in I.
$$
Taking $\delta := 1/\log x$ and $C := 2$, Mertens' theorem implies that if $|s-t| \leq \delta$ then
$$
|Y_t-Y_s| \leq \sum_{p \leq x} \frac{|\cos(t\log p) - \cos(s\log p)|}{p} \leq \sum_{p \leq x} \frac{|e^{i(s-t)\log p}-1|}{p} \leq |s-t| \sum_{p \leq x} \frac{\log p}{p} \leq 1 + O(\delta) < C,
$$
for large enough $x$. Thus, by Lemma \ref{lem:maxTail}, the RHS of \eqref{eq:probmax} is
$$
\ll (\log x) \max_{|t| \leq 1/2} \mb{P}\left(-\sum_{p \leq x} \frac{\mbf{f}(p)(1-\cos(t\log p))}{p} \geq \log((\log x)/(e^2c'))\right),
$$
and by Lemma \ref{lem:AngXu} this is bounded above by
$$
\ll (\log x) \exp(-\exp(e^2c'c \log x)) \leq \exp(-x^{\beta_1}),
$$
for any $\beta_1 \in (0,e^2cc')$ and $x$ sufficiently large in terms of $c,c'$. \\
Next, we consider the event \eqref{eq:tbig}, the treatment of which is similar. For $1 \leq k \leq T$ write $y_k = \exp((\log (2k))^2)$ as before. Since we always have
$$
\sum_{p \leq y_k} \frac{f(p)}{p} \leq 2\log\log(2k),
$$
on rearranging \eqref{eq:tbig} it suffices to bound the probability that
$$
\log x < C \sum_{1 \leq k \leq T} \frac{(\log(2k))^{8}}{k^2} \exp\left(\max_{|t-k| \leq 1/2} -\sum_{y_k < p \leq x} \frac{f(p)(1-\cos(t\log p))}{p}\right).  
$$
Taking the maximum over $k$, we see that there is a constant $\tilde{C} >0$ and a $1 \leq k_0 \leq T$ such that
$$
\log x < \tilde{C} \exp\left(\max_{|t-k_0| \leq 1/2} -\sum_{y_{k_0} < p \leq x} \frac{f(p)(1-\cos(t_0 \log p))}{p} \right), 
$$
or equivalently, that
$$
\max_{t \in [k_0-1/2,k_0+1/2]} -\sum_{y_{k_0} < p \leq x} \frac{\mbf{f}(p)(1-\cos(t_0\log p))}{p} > \log((\log x)/\tilde{C}),
$$
This time, we take $I = [k_0-1/2,k_0+1/2]$ and the same process $(Y_t)_t$ as above but restricted to our new interval $I$.
Now, combining Lemma \ref{lem:maxTail} with $\delta = 1/\log x$ and $C = 2$, together with Lemma \ref{lem:AngXu} as before, the probability that this occurs is
\begin{align*}
&\ll (\log x) \max_{|t-k_0| \leq 1/2} \mb{P}\left(-\sum_{y_{k_0} < p \leq x} \frac{\mbf{f}(p)(1-\cos(t\log p))}{p} \geq \log(\log(x/e^2\tilde{C}))\right) \\
&\ll (\log x) \exp(-x^{e^2\tilde{C}c}) \leq \exp(-x^{\beta_2})
\end{align*}
for any $\beta_2 \in (0,e^2\tilde{C}c)$ and $x$ sufficiently large. Taking $\beta := \min\{\beta_1,\beta_2\} > 0$ proves the claim.
\end{proof}
\section{Proof of Corollaries \ref{cor:LipError}, \ref{cor:NegTrunc} and \ref{cor:Gold}}
Incorporating some of the ideas of the previous section, we will now prove several of our main corollaries.
%Corollary \ref{cor:LipError}. 
\begin{proof}[Proof of Corollary \ref{cor:LipError}]
Applying Proposition \ref{prop:passToSmooth} with $\theta = 0$ and $T = (\log x)^2$, say, we have that
$$
L_f(x) = \left(1+O\left(\frac{1}{\log x}\right)\right) \tilde{M}_g(w_0x) + O\left(\frac{1}{\log x}\int_1^x (H_1(y) + H_2'(y;T))\frac{dy}{y\log y}\right),
$$
the advantage being that we have the error term $H_2'(y;T)$ here in place of $H_2(y;T)$ (which enables us to save some $\log\log x$ powers). \\
Observe now that it suffices to show that for all $2 \leq y \leq x$,
$$
H_1(y) + H_2'(y;T) \ll (\log y)^{2/\pi},
$$
as then the error term is of size
$$
\frac{1}{\log x}\int_1^x (H_1(y) + H_2'(y;T)) \frac{dy}{y \log y} \ll \frac{1}{\log x}\int_1^x \frac{dy}{y (\log y)^{1-2/\pi}} \ll \frac{1}{(\log x)^{1-2/\pi}},
$$
as claimed.\\
First, we have that as $\cos(t\log p)\geq 0$ whenever $p \leq e^{1/|t|}$, and more generally as $f(p)-1 \leq 0$,
%we see that
\begin{align*}
\log H_1(y) \leq \max_{|t| \leq 1/2} \left(2\sum_{\ss{e^{1/|t|} < p \leq y \\ \cos(t\log p) < 0}} \frac{|\cos(t\log p)|}{p}\right).
\end{align*}
(Note that equality holds here when 
\begin{equation}\label{eq:equalLip}
f(p) = \begin{cases} -1 &\text{ if }  \cos(t\log p) < 0 \\ 1 &\text{ if } \cos(t\log p) \geq 0,
\end{cases}
\end{equation}
an example to be discussed in the next section.)\\
Next, note that the function
\begin{equation}\label{eq:defPhi}
\phi(u) := \begin{cases} |\cos(u)|  &:\text{ if } \cos(u) < 0 \\ 0  &:\text{ otherwise}, \end{cases}
\end{equation}
is continuous and $2\pi$-periodic, with
$$
\frac{1}{2\pi} \int_0^{2\pi} \phi(u) du = \frac{1}{\pi}.
$$
Thus, we apply Lemma \ref{lem:TenEst} to obtain
$$
\sum_{e^{1/|t|} < p \leq y} \frac{\phi(t\log p)}{p} = \frac{1}{\pi} \log(|t| \log y) + O(1) \leq \frac{1}{\pi} \log\log y + O(1),
$$
uniformly over $|t| \leq 1/2$. We thus deduce that
$$
H_1(y) \ll (\log y)^{2/\pi}.
$$
Next, we deal with $H_2'(y;T)$ in a similar way. For each $k$ we have by Lemma \ref{lem:TenEst} (taking $\phi(u) = \cos(u)$) that
\begin{align*}
\max_{|t-k| \leq 1/2} \sum_{y_k < p \leq y} \frac{f(p) \cos(t\log p)}{p} &= -\max_{|t-k| \leq 1/2} \sum_{\ss{y_k < p \leq y}} \frac{(1-f(p)) \cos(t\log p)}{p} + O(1) \\
&\leq 2\max_{|t-k| \leq 1/2} \sum_{\ss{y_k < p \leq y \\ \cos(t\log p) < 0}} \frac{|\cos(t\log p)|}{p} + O(1).
\end{align*}
Applying Lemma \ref{lem:TenEst} with $\phi(u)$ as defined in \eqref{eq:defPhi}, we get
$$
\max_{|t-k| \leq 1/2} \sum_{y_k < p \leq y} \frac{f(p)\cos(t\log p)}{p} \leq \frac{2}{\pi} \log\left(\frac{\log y}{\log y_k}\right) + O(1) \leq \frac{2}{\pi} \log\log y + O(1),
$$
uniformly over $k \geq 1$. It follows, then, that
$$
H_2'(x;T) \ll \left((\log y)^{4/\pi} \sum_{k \geq 1} \frac{(\log(2k))^6}{k^2}\right)^{1/2} \ll (\log y)^{2/\pi}.
$$
As discussed above, this implies the claim.
%Inserting these two bounds into the error term, we obtain the claim.
\end{proof}
\begin{proof}[Proof of Corollary \ref{cor:NegTrunc}]
Since $g = 1\ast f \geq 0$, we have $\tilde{M}_g(w_0x) \geq 0$. The claim now follows from Corollary \ref{cor:LipError}.
\end{proof}
\begin{proof}[Proof of Corollary \ref{cor:Gold}]
For $1 \leq z \leq x$ define\footnote{As usual, write $P^+(n)$ to denote the largest prime factor of $n$, $P^+(1) := 1$.} 
$$
f_{\leq z}(n) := f(n)1_{P^+(n) \leq z}, \quad g_z(n) := 1 \ast f_{\leq z}(n).
$$ 
By Corollary \ref{cor:LipError}, the LHS in the statement is
$$
\sum_{\ss{n \leq x \\ P^+(n)\leq z}} \frac{f(n)}{n} = \tilde{M}_{g_z}(w_0x) + O\left(\frac{1}{(\log x)^{1-2/\pi}}\right).
$$
Observe that for each prime $p$, $|g_z(p)| = |1+f_z(p)| = 1+f_z(p)$, so Lemma \ref{lem:UppBdNonNeg} yields
%the bound of Halberstam and Richert yields
\begin{align*}
|\tilde{M}_{g_z}(w_0x)| \leq \tilde{M}_{|g_z|}(w_0x) \ll \frac{1}{\log x} \exp\left(\sum_{p \leq x} \frac{|g_z(p)|}{p}\right) &\ll \exp\left(\sum_{p \leq z} \frac{f(p)}{p}\right) \\
&\asymp (\log z) \exp\left(-\mb{D}(f,1;z)^2\right).
\end{align*}
Inserting this into the previous estimate yields the claim.
\end{proof}

\section{Optimality of Corollary \ref{cor:NegTrunc}} \label{sec:optimal}
In this section, we prove Theorem \ref{thm:CorOpt}. To do this, we must modify a result from \cite{GraMan} that yields asymptotic formulae for logarithmic averages of certain types of multiplicative functions. \\
Before stating the result in question, we introduce the following setup. As above, let $X$ be large and let $\tfrac{1}{(\log \log X)^2} \leq \e <1/4$ and $t_0 = t_0(X)$ be parameters to be chosen later, such that $|t_0| \sqrt{\log X} \geq 1$. Let $\phi$ be a $1$-periodic, smooth and even function defined on $[-1/2,1/2]$ by 
$$
\phi(u) := \begin{cases} 1 &\text{ if } u \in (-\pi/2 + \e, \pi/2-\e) \\ 0 &\text{ if } u \notin (-\pi/2-\e,\pi/2+\e) \\ \in [0,1] &\text{ otherwise.} \end{cases}
$$
Let $h(u) := 2\phi(u) - 1$ and note that $h(u)$ is a smooth approximation of the function given in \eqref{eq:equalLip} in the previous section. We clearly have $\hat{h}(0) = 2\hat{\phi}(0) - 1$ and by the usual integration by parts argument, for each $j\geq 1$ and $n \neq 0$,
$$
\hat{h}(n) = 2 \hat{\phi}(n) \ll_j \e^{-j} n^{-(j+1)}.
$$
In addition, we have, trivially, that
$$
\hat{h}(n) = 2(\hat{1}_{(-\pi/2,\pi/2)}(n) + O(\e)) = \frac{2\chi_4(n)}{\pi n} + O(\e)\ll \max\{\e,\frac{1}{|n|}\}, \quad n \neq 0,
$$
where $\chi_4$ denotes the non-principal character modulo $4$. This implies the general upper bound
\begin{equation}\label{eq:FourCoeffs}
\hat{h}(n) \ll \min\left\{\frac{1}{|n|}, \frac{1}{\e n^2}\right\}, \quad n\neq 0.
\end{equation}
It follows in particular that for any $N \geq 1$, 
\begin{equation}\label{eq:FCconv}
\sum_{1 \leq |n| \leq N} |\hat{h}(n)| \ll \sum_{1 \leq |n|\leq \e^{-1}} \frac{1}{|n|} + \sum_{\e^{-1} < |n| \leq N} \frac{1}{\e n^2} \ll \log(1/\e), 
\end{equation}
so the above series converges absolutely. We will use this shortly. \\
%Given $t_0 = t_0(x)$ 
Let us now define a multiplicative function $f = f_{t_0}$ as follows: at primes,
$$
f(p) = h\left(\frac{t_0 \log p}{2\pi}\right),
$$
% This choice is made so that $(1-f(p))\cos(t_0\log p) = -2|\cos(t_0\log p)|$, and therefore the sum over primes is of size
% $$
% \exp\left(4\sum_{p \leq y} \frac{|\cos(t_0\log p)|}{p}\right) \gg \left(|t_0|\log y\right)^{4/\pi}.
% $$
and, inducting on $m\geq 1$, we obtain $f(p^m)$ at prime powers via the convolution formula
$$
f(p^m) = \frac{1}{m} \sum_{1 \leq j \leq m} f(p^{m-j}) h\left(\frac{t_0 \log p^j}{2\pi}\right).
$$
With this definition, for $\text{Re}(s) > 1$ we have
\begin{equation} \label{eq:LprimeL}
-\frac{L'}{L}(s,f) = \sum_{n \geq 1} h\left(\frac{t_0\log n}{2\pi}\right) \Lambda(n)n^{-s}.
\end{equation}
Integrating \eqref{eq:LprimeL} termwise in the half-plane $\text{Re}(s) > 1$ allows us to write
$$
\log L(s,f) = \sum_{p^m} \frac{h\left(\frac{t_0 \log p^m}{2\pi}\right)}{mp^{ms}} = \sum_{p^m} \frac{1}{mp^{ms}} \sum_{n \in \mb{Z}} \hat{h}(n) p^{imnt_0}. 
$$
Owing to the absolute convergence of the series over $|\hat{h}(n)|$, we may swap orders of summation to rewrite this as
$$
\log L(s,f) = \sum_{n \in \mb{Z}} \hat{h}(n) \sum_{p^m} \frac{1}{mp^{m(s-int_0)}} = \sum_{n \in \mb{Z}} \hat{h}(n) \log \zeta(s-int_0),
$$
and therefore we obtain that for $\text{Re}(s) > 1$, 
$$
L(s,f) = \prod_{n \in \mb{Z}} \zeta(s-int_0)^{\hat{h}(n)} = \zeta(s)^{\hat{h}(0)} \prod_{n \neq 0} \zeta(s-int_0)^{\hat{h}(n)} =: \zeta(s)^{\hat{h}(0)} \tilde{L}(s,f).
$$
In the sequel, let $A > 2$ and let $N := \lceil \tfrac{(\log X)^A}{|t_0|}\rceil$. Set also $T := (N+1/2)|t_0|$. We also write $\sg_0 := 1 + \tfrac{1}{\log X}$, and define 
$$
\sg(\tau) := \frac{c}{\log(2+|\tau|)}, \quad r_0 := \tfrac{1}{4}\min\{\sg(3T), |t_0|\},
$$ 
where $c > 0$ is chosen suitably such that $\zeta(\sg + i\tau) \neq 0$ whenever $\sg \geq 1-\sg(\tau)$. \\
%Finally, we define $$.\\
Let $X < x \leq X^2$. By Perron's formula, we obtain
$$
L_f(x) - \tilde{M}_g(w_0x) = \frac{1}{2\pi i} \int_{\sg_0-iT}^{\sg_0 + iT} \tilde{L}(s,f) \frac{(w_0x)^{s-1}}{s} \zeta(s)^{\hat{h}(0)}\left(\frac{s}{s-1}w_0^{1-s} - \zeta(s)\right) ds + O\left(\frac{1}{\log X}\right).
$$
By Lemma \ref{lem:snear1} above, the bracketed expression is holomorphic in a neighbourhood of $s = 1$, and vanishes at $s = 1$.  In the sequel, we write 
$$
H(s) = \zeta(s)^{\hat{h}(0)} \left(\frac{s}{s-1}w_0^{1-s} - \zeta(s)\right). 
$$
By definition,
\begin{equation}\label{eq:0FC}
|\hat{h}(0)| = |2\hat{\phi}(0) - 1| \leq 2\e < 1,
\end{equation}
so that by Lemma \ref{lem:snear1}, 
\begin{equation}\label{eq:magHnear1}
|H(s)| \asymp |s-1|^{1-\hat{h}(0)}, \quad \frac{1}{\log X} \leq |s-1| \leq 1
\end{equation} Furthermore, by standard bounds for the Riemann zeta function 
\begin{equation}\label{eq:HBd}
|H(\sg + it)| \ll (\log(2+|t|)^2, \text{ for } \sg \geq 1-\sg(t), \ |t| \geq 1,
\end{equation}
a bound we will use shortly.
\\
Following \cite{GraMan}, we define the truncated products
$$
L_N(s,f) := \prod_{|m| \leq 2N} \zeta(s-int_0)^{\hat{h}(n)}, \quad \tilde{L}_N(s,f) := L_N(s,f) \zeta(s)^{-\hat{h}(0)}.
$$
The following lemma, which essentially follows from the proof of Lemma 7.1 of \cite{GraMan}, is tailored to the present circumstances (wherein the rate of decay of the Fourier coefficients depends on a parameter that is varying with $X$).
\begin{lem} \label{lem:GraManSpec}
Assume the above notation, and that $|t_0|^{-1} \ll (\log\log X)^{O(1)}$. Define
$$
H := \sum_{n \in \mb{Z}} |\hat{h}(n)| < \infty.
$$
Then the following holds: \\
(a) For $|t| \leq T$ and $\text{Re}(s) = \sg_0$, we have
$$
L(s,f) = L_N(s,f) + O\left(\frac{(\log\log X)^2}{(\log X)^{A-1}}\right).
$$
(b) We have
$$
\max_{|t| \leq T} |L_N(1-r_0 + it,f)| \ll (|t_0|^{-1} \log T)^H \ll (\log\log X)^{O(\log(1/\e))}.
$$
%In particular, if $|t_0|^{-1} \ll \log T$ then as $H \leq \log(1/\e)$, this is 
(c) For $\eta \in \{-1,+1\}$ we have
$$
\max_{1-r_0 \leq \sg \leq \sg_0} |L_N(\sg + i\eta T,f)| \ll |t_0|^{-\tfrac{c}{\e N}} (\log T)^H \ll (\log\log X)^{O(\log(1/\e))}.
$$
(d) Uniformly over $|n| \leq N$ and $|s-1| \leq 2r_0$, we have
$$
\prod_{\ss{|k| \leq 2N \\ k \neq n}} |\zeta(s - i(k-n)t_0)^{\hat{h}(k)}|\ll(|t_0|^{-1} \log T)^{H} \ll (\log\log X)^{O(\log(1/\e))}.
$$
% Moreover, if $\ell$ is a fixed non-zero integer then
% $$
% \prod_{\ss{|k| \leq 2N \\ k \neq \ell}} \zeta(1-i(k-\ell)t_0)^{\hat{h}(k)}  = (1+O(|t_0|)))(it_0)^{\hat{h}(1)-h(0)}. 
% $$
\end{lem}
\begin{proof}
The proofs of all of these statements are the same as in \cite{GraMan}, but implementing the bounds \eqref{eq:FourCoeffs} and \eqref{eq:FCconv},
the latter of which implies that $H \ll \log(1/\e)$. 
%For the final statement, we similarly
\end{proof}
We next follow the proof of Theorem 7.1 of \cite{GraMan}, substituting Lemma 7.1 there with Lemma \ref{lem:GraManSpec} here. Taking any $A > 2$, precisely the same arguments (\emph{mutatis mutandis} on changing notation) allow one to show that
\begin{align*}
L_f(x) - \tilde{M}_g(w_0x) &= \frac{1}{2\pi i} \int_{\sg_0-iT}^{\sg_0 + iT} \tilde{L}_N(s,f) H(s) \frac{(w_0x)^{s-1}}{s} ds + O\left(\frac{1}{\log X}\right) \\
&= \sum_{|n| \leq N} \frac{(w_0x)^{int_0}}{2\pi i} \int_{\mc{H}} \tilde{L}_N(s+int_0,f) H(s+int_0) \frac{(w_0x)^{s-1}}{s + int_0} ds + O\left(\frac{1}{\log X}\right) \\
&= \sum_{1 \leq |n| \leq N} \frac{(w_0x)^{int_0}}{2\pi i} \int_{\mc{H}} G_n(s) \frac{(w_0x)^{s-1}}{(s-1)^{\hat{h}(n)}} ds \\
&+ \frac{1}{2\pi i} \int_{\mc{H}} \tilde{L}_N(s,f) H(s)\frac{(w_0x)^{s-1}}{s} ds + O\left(\frac{1}{\log X}\right),
\end{align*}
where $\mc{H}$ denotes the Hankel contour of radius $\tfrac{1}{\log x}$ about $s = 1$, truncated at $\text{Re}(s) \geq 1-r_0$, and
$$
G_n(s) := \frac{H(s+int_0)}{s+int_0}((s-1)\zeta(s))^{\hat{h}(n)} \prod_{\ss{|m| \leq 2N \\ m \neq 0,n}} \zeta(s-i(m-n)t_0)^{\hat{h}(m)}, \quad 1\leq |n| \leq N.
$$
Consider the sum over $1 \leq |n| \leq N$. Since it is holomorphic near $s = 1$, we expand $G_n(s)$ as a power series as 
$$
G_n(s) = \sum_{k \geq 0} \mu_{n,k}(t_0) (s-1)^k.
$$
% which leads to the expansion
% $$
% \frac{1}{2\pi i} \int_{\mc{H}} G_n(s) \frac{x^{s-1}}{(s-1)^{\hat{h}(n)}} ds = \sum_{j \geq 0} \mu_{n,j}(t_0) \frac{1}{2\pi i} \int_{\mc{H}} (s-1)^{j-\hat{h}(n)} x^{s-1} ds.
% $$
By Cauchy's integral formula (taking $r = 2r_0$), \eqref{eq:HBd} and Lemma \ref{lem:GraManSpec}(d) above, we get that for each $j \geq 1$,
\begin{align*}
|\mu_{n,j}(t_0)| &\ll r^{-j} \max_{|s-1| = r}\frac{|H(s+int_0)|}{|s+int_0|}\prod_{\ss{|k| \leq 2N \\ k \neq n}} |\zeta(s - i(k-n)t_0)^{\hat{h}(k)}| \\
&\ll r^{-j} \frac{(\log\log X)^{O(\log(1/\e))}}{1+|nt_0|}.
\end{align*}
Since $|s-1| \leq r_0 = r/2$ for all $s \in \mc{H}$, we obtain
\begin{align*}
&\frac{1}{2\pi i} \int_{\mc{H}} \left(G_n(s) - \mu_{n,0}(t_0)\right) \frac{x^{s-1}}{(s-1)^{\hat{h}(n)}} ds \\
&\ll \frac{(\log\log X)^{O(\log(1/\e))}}{1+|nt_0|} \int_{\mc{H}} x^{\text{Re}(s)-1}|s-1|^{1-\text{Re}(\hat{h}(n))} \sum_{j \geq 1}\left(\frac{|s-1|}{r}\right)^j |ds| \\
&\ll \frac{(\log\log X)^{O(\log(1/\e))}}{1+|nt_0|}(\log x)^{\text{Re}(\hat{h}(n))-2}.
\end{align*}
Using Cor. 0.18 of \cite{Ten}, we get
$$
\frac{1}{2\pi i} \int_{\mc{H}} (s-1)^{-\hat{h}(n)} (w_0x)^{s-1} ds = \frac{(\log (w_0x))^{\hat{h}(n) - 1}}{\Gamma(\hat{h}(n) - 1)} + O\left(X^{-r_0/2}\right),
$$
where we used the fact that $|\hat{h}(n)| \leq \frac{2}{\pi |n|} < 1$ for all $n \neq 0$. Applying this for each $1 \leq |n| \leq N$ gives
\begin{align*}
&\sum_{1 \leq |n| \leq N} \frac{(w_0x)^{int_0}}{2\pi i} \int_{\mc{H}} G_n(s) \frac{(w_0x)^{s-1}}{(s-1)^{\hat{h}(n)}} ds \\
&= \sum_{1 \leq |n| \leq N} \frac{\mu_{n,0}(t_0)(w_0x)^{int_0}}{\Gamma(\hat{h}(n) - 1)} (\log (w_0x))^{\hat{h}(n) - 1} + O\left(N X^{-r_0/2} + \frac{(\log\log X)^{O(\log(1/\e))}}{(\log X)^2} \sum_{1 \leq |n| \leq N} \frac{(\log X)^{\text{Re}(\hat{h}(n))}}{1+|nt_0|}\right) \\
&= \sum_{1 \leq |n| \leq N} \frac{\mu_{n,0}(t_0)(w_0x)^{int_0}}{\Gamma(\hat{h}(n) - 1)} (\log (w_0x))^{\hat{h}(n) - 1} + O\left(\frac{1}{\log X}\right).
\end{align*}
Finally, for the term $n = 0$ we use the bound \eqref{eq:magHnear1} and Lemma \ref{lem:GraManSpec}(d) to derive that
\begin{align*}
\frac{1}{2\pi i} \int_{\mc{H}} \tilde{L}_N(s,f) H(s) \frac{(w_0x)^{s-1}}{s}ds &\ll (\log\log X)^{O(\log(1/\e))} \left(\int_{-1/\log X}^{-r_0} |\sg|^{1-\hat{h}(0)} x^{\sg } d\sg + (\log X)^{\hat{h}(0) - 2}\right) \\
&\ll (\log X)^{\hat{h}(0)- 2+ o(1)}.
\end{align*}
Using $|\hat{h}(0)| < 1$ by \eqref{eq:0FC}, this leads to the asymptotic formula
\begin{align*}
L_f(x) - \tilde{M}_g(w_0x) = \sum_{1 \leq |n| \leq N} \frac{\mu_{n,0}(t_0)(w_0x)^{int_0}}{\Gamma(\hat{h}(n) - 1)} (\log x)^{\hat{h}(n)-1} + O\left(\frac{1}{\log X}\right),
\end{align*}
wherein we have
$$
\mu_{n,0}(t_0) = \lim_{s \ra 1} G_n(s) = \frac{H(1+int_0)}{1+int_0} \prod_{\ss{|m| \leq 2N \\ m \neq 0,n}} \zeta(1-i(m-n))^{\hat{h}(m)}.
$$
Note that $|\mu_{n,0}(t_0)|/|\Gamma(\hat{h}(n)-1)| \ll (\log x)^{o(1)}$ by Lemma \ref{lem:GraManSpec}(d), that $|\hat{h}(n)| \leq \hat{h}(1) - \frac{1}{\pi}$ for all $|n| \neq 1$, and $\hat{h}(-1) = \hat{h}(1) \in \mb{R}$. Thus, we get
$$
L_f(x) - \tilde{M}_g(w_0x) = \left(2 \text{Re}((w_0x)^{it_0} \mu_{1,0}(t_0))+ o(1)\right)\frac{(\log x)^{\hat{h}(1)-1}}{\Gamma(\hat{h}(1) - 1)}. 
$$
We now observe that there must be some $X < x \leq Xe^{2\pi/|t_0|}$ such that 
$$
\text{Re}((w_0x)^{it_0} \mu_{1,0}(t_0)) \geq \frac{1}{2}|\mu_{1,0}(t_0)|,
$$ 
say. Since $|t_0| \geq 1/\sqrt{\log X}$, $e^{2\pi/|t_0|} \leq X$, so that $x \in (X,X^2]$. It follows that
$$
\max_{X < x \leq X^2} |L_f(x) -\tilde{M}_g(w_0x)| \gg |\mu_{1,0}(t_0)| \frac{(\log X)^{\hat{h}(1)-1}}{\Gamma(\hat{h}(1) - 1)}.
$$
The choice $\e = (\log\log X)^{-1}$ is admissible, and since $\hat{h}(1) = \frac{2}{\pi} + O(\e)$, the above bound becomes
$$
\max_{X < x \leq X^2} |L_f(x) - \tilde{M}_g(w_0x)| \gg |\mu_{1,0}(t_0)| (\log X)^{\frac{2}{\pi}-1},
$$
%from which the bound
%$$
%|L_f(x) - %\tilde{M}_g(w_0x)| %\asymp |\mu_{1,0}(t_0)| %(\log x)^{\frac{2}{\pi} -1}.
%$$
Next, we bound $|\mu_{1,0}(t_0)|$ from below. From our earlier estimate $\hat{h}(m) = \tfrac{2\chi_4(m)}{\pi m} + O(\e)$, and that
\begin{align*}
\prod_{\ss{1 \leq |m| \leq |t_0|^{-1} \\ m \neq 1}} \zeta(1-i(m-1)t_0)^{\hat{h}(m)} &\asymp \exp\left(\sum_{\ss{1 \leq |m| \leq |t_0|^{-1} \\ m \neq 1}}  \left(\frac{2}{\pi}\frac{\chi_4(m)}{m} + O(\e)\right) \left(\log(1/|t_0|) + \log(m-1)\right)\right) \\
&\asymp |t_0|^{2/\pi-1 + O(\e|t_0|^{-1})}. %\exp(O(\e |t_0|^{-1} \log(1/|t_0|)))
\end{align*}
Estimating the remaining product over $|t_0|^{-1} < m \leq N$, we get an expression
$$
\exp\left(O\left(\e^{-1} \sum_{|m| > |t_0|^{-1}} \frac{\log\log(2m)}{m^2}\right)\right) = \exp\left(O(\e^{-1} |t_0| \log\log(1/|t_0|))\right) = |t_0|^{O(\e^{-1} |t_0| \tfrac{\log\log(1/|t_0|)}{\log(1/|t_0|)})}.
$$
If we select $\e := |t_0| \left(\tfrac{\log\log(1/|t_0|)}{\log(1/|t_0|)}\right)^{1/2}$ with $|t_0| = o(1)$ as $X \ra \infty$ then, using \eqref{eq:magHnear1},
$$
\mu_{1,0}(t_0) = \frac{H(1+it_0)}{1+it_0} \prod_{\ss{1 \leq |m| \leq |t_0|^{-1} \\ m\neq 1}} \zeta(1-i(m-1)t_0)^{\hat{h}(m)} \asymp |t_0|^{1-\hat{h}(0)} \cdot |t_0|^{2/\pi-1 + o(1)} = |t_0|^{2/\pi + o(1)}.
$$
We thus deduce that
$$
\max_{X < x \leq X^2} |L_f(x) - \tilde{M}_g(w_0x)| \gg \frac{|t_0|^{2/\pi + o(1)}}{(\log X)^{1-2/\pi}},
$$
and since we may choose $t_0$ so that $|t_0| \ra 0$ arbitrarily slowly with $X$, we may ensure that $|t_0|^{2/\pi + o(1)} \geq \psi(X)$ at scale $X$, for any given monotonically decreasing function $\psi(X)$. \\
%The claim now follows.
%We may now complete the proof by taking $|t_0|$ 
%$\e = (\log\log x)^{-1.1}$ and $|t_0| = \e^{1.1}$, we get 
%$$
%|\mu_{1,0}(t_0)| \asymp (\log\log x)^{c},
%$$
%where $c := (1.1)^2 (1-2/\pi) \in (0,1)$. This completes the proof of the lower bound
%$$
%|L_f(x) - \tilde{M}_g(w_0x)| \gg \frac{(\log\log x)^c}{(\log x)^{1-2/\pi}}.
%$$
Finally, to obtain the claim it remains to show that
$$
\max_{X < x \leq X^2} \frac{\tilde{M}_g(w_0x)}{\log x} = o\left(\frac{1}{(\log X)^{1-2/\pi}}\right),
$$
and for this we will prove the stronger estimate $\tilde{M}_g(w_0y) \ll (\log y)^{o(1)}$,
for all $y \geq X$.
%about the size of $\tilde{M}_g(w_0x)$ we note 
Indeed, using Lemma \ref{lem:TenEst} and \eqref{eq:0FC}, we have
$$
\exp\left(\sum_{p \leq y} \frac{f(p)}{p} \right) \asymp \exp\left(\sum_{e^{1/|t_0|} < p \leq y} \frac{h(\tfrac{t_0}{2\pi}\log p)}{p}\right) \asymp (|t_0| \log y)^{\hat{h}(0)} \ll (\log y)^{2\e},
$$
recalling that $\e = o(|t_0|) = o(1)$. Since $\tilde{M}_g(w_0y) \geq 0$, the bound $\tilde{M}_g(w_0y) \ll (\log y)^{o(1)}$ now follows from Lemma \ref{lem:UppBdNonNeg} for all $y \geq X$.

\section{Converse theorems for small and negative partial sums} \label{sec:Conv}
% We recall the following lemma of Kerr and Klurman \cite{KeKl}.
% \begin{lem} \label{lem:KeKl}
% Let $\e > 0$ be small, $f: \mb{N} \ra [-1,1]$ completely multiplicative and $ g = 1 \ast f$ as before. Given $\delta \in (0,1)$ define
% $$
% \mc{P}_{\delta} := \{p \in \mb{P} : f(p) \geq -\delta\},
% $$
% and assume that for some
% $$
% \frac{40000}{\e^2} \leq v \leq \frac{\log x}{1000\log\log x}
% $$
% we have
% $$
% \sum_{\ss{p \in \mc{P}_{\delta}\cap [x^{1/v},x]}} \frac{1}{p} \geq 1+\e.
% $$
% Then we have
% $$
% \tilde{M}_g(x) \gg \e^2 \left(\frac{1-\delta}{v}\right)^{v(1+o(1))/e} \exp\left(\sum_{p \leq x} \frac{f(p)}{p}\right).
% $$
% \end{lem}
%As a consequence, we may obtain the following.
In this final section we shall prove our remaining Corollaries \ref{cor:improveNeg}, \ref{cor:convLf} and \ref{cor:const}. 
%Each of these results depend on a parameter $v = %v(x)$, chosen such that
%$$
%The first of these addresses possible %improvements to the bounds in the negative %truncation problem, for $\pm 1$-valued %multiplicative functions.
In fact, we prove the following more general propositions along the same lines. For notational simplicity, in the sequel we write
$$
\mc{F} := \{f : \mb{N} \ra [-1,1] : \ f \text{ completely multiplicative}\}, \quad \mc{F}_{\pm} := \{f \in \mc{F} : f(n) \in \{-1,+1\} \text{ for all } n\}.
$$
\begin{prop} \label{prop:improveNeg}
Let $f \in \mc{F}$.
%: \mb{N} \ra [-1,1]$ be a completely multiplicative function. 
Let $x$ be large and $\e > 0$, and assume that $v \geq v_0(\e)$ is chosen such that 
\begin{equation}\label{eq:fplus1}
\sum_{\ss{x^{1/v} < p \leq x \\ f(p) \geq -\delta}} \frac{1}{p} \geq 1+\e
\end{equation}
holds for some $\delta \in (0,1)$.
%Set $w = v/(1-\delta)$ and
%Define
% \begin{equation}\label{eq:Wdef}
% W := \exp\left(w\log w + \frac{1}{1-\delta}\log\log\log x\right).
% \end{equation}
Then there are constants $c_1,c_2 > 0$ such that 
%either $L_f(x) \geq 0$ or else
$$
L_f(x) \geq -\frac{c_1 \log\log x}{\log x} \exp\left(c_2(\log\log x)^{1-\tau} (v\log(v\log\log x))^{\tau}\right),
$$
where we have set
\begin{equation}\label{eq:tauDef}
\tau := \begin{cases} 1/3 &\text{ if } f \in \mc{F}_{\pm}, \\ 1/4 &\text{ otherwise. }
%f \in \mc{F}.
\end{cases}
\end{equation}
In particular, if $v \log v = o(\log\log  x)$ then 
$$
L_f(x) > -\frac{1}{(\log x)^{1-o(1)}}.
$$ 
\end{prop}
The second main result of this section deals with converse theorems, wherein $|L_f(x)|$ is assumed to be small. Once again, the size of $v$ as above is a relevant consideration.
\begin{prop} \label{prop:convLf}
Let $x$ be large and fix $C > 0$ and $\e > 0$. Suppose $f \in \mc{F}$ satisfies
%: \mb{N} \ra [-1,1]$ be a completely multiplicative function 
%such that 
$$
|L_f(x)| \ll \frac{\exp\left(C (\log\log x)^{2/3}\right)}{\log x},
$$ 
and suppose $v \geq v_0(\e)$ is chosen such that \eqref{eq:fplus1} holds for some $\delta \in (0,1)$. Then 
% if $\eta \in (0,1/2)$ is sufficiently small then one of the following must hold:
% \begin{enumerate}[(a)]
% \item there is $\eta > 0$ such that $v\log v \geq \eta^3 \log\log x$; or else
% \item we have $\mb{D}(f,\lambda;x)^2 \ll \eta \log\log x.$ 
% \end{enumerate}
% Furthermore, if $v = O(1)$ then
$$
\mb{D}(f,\lambda;x)^2 \ll (\log\log x)^{1-\tau} (v\log(v\log\log x))^{\tau},
$$
where $\tau$ is as in \eqref{eq:tauDef}. 
\end{prop}
\begin{rem}
The appearance of the parameter $v$ in these results is an outcome of an application of Lemma \ref{lem:KeKl} towards lower bounds for $M_g(y)$. This parameter is necessary in order to address the possibility that e.g. $f(p)= - 1$ for all $x^{1/v} < p \leq x$, and therefore that $g(p^k) = 1_{2|k}$ for all such $p$. In such an instance, $g(n)$ is supported on integers of the form $ab^2$, where $P^+(a) \leq x^{1/v}$ and $b > 1 \Rightarrow b > x^{1/v}$, a set which is extremely sparse if $v$ is very large. \\
In the forthcoming estimates, the size of $v$ is therefore important. Unfortunately, we do not have much to say when $v \log v \gg \log\log x$, and it would be interesting to understand this case. If $f \in \mc{F}_{\pm}$ then it is of course possible, for example, to treat $L_f(x)$ directly when $f(p) = +1$ \emph{extremely rarely} (e.g., a variant of the Liouville function with a very sparse number of exceptional primes with $f(p) = +1$). However, we have not been able to give a uniform treatment of functions for which e.g.
$$
\sum_{\ss{x^{1/v} < p \leq x \\ f(p) = +1}} \frac{1}{p} = 2 
$$
with $v \asymp \log\log x$, say. 
\end{rem}

The above propositions have a common origin (that is related to the analysis of Section \ref{sec:random}). To see this, recall that on taking $T = (\log x)^2$, say, Proposition \ref{prop:passToSmooth} (with $\theta = 0$) and Lemma \ref{lem:monoBdd} yield
\begin{equation} \label{eq:mainThmApp}
L_f(x) = \left(1+O\left(\frac{1}{\log x}\right) \right) \tilde{M}_g(w_0x) + O\left(\frac{\log\log x}{\log x} (H_1(x) + H_2'(x;T))\right).
\end{equation}
In proving Proposition \ref{prop:improveNeg}, the condition $L_f(x) < 0$ may be recast as
$$
0 > \tilde{M}_g(w_0x) - C \frac{\log\log x}{\log x} \max\{H_1(x), H_2'(x;T)\},
$$
so that our goal is to establish structural information on functions $f$ that satisfy either 
\begin{equation}\label{eq:cases}
\text{(a) } \tilde{M}_g(w_0x) \leq C\frac{\log\log x}{\log x} H_1(x) \text{ or (b) } \tilde{M}_g(w_0x) \leq C\frac{\log\log x}{\log x} H_2'(x;T).
\end{equation}
%To see that the same discussion applies in the %context of Proposition \ref{prop:convLf}, 
%we first establish the following simple lemma.
By 
%Lemma \ref{lem:Limit} and 
\eqref{eq:mainThmApp}, we see that if $|L_f(x)| \ll \tfrac{\exp(C(\log\log x)^{2/3})}{\log x}$ then there is a constant $C' > 0$ such that either
$$
\tilde{M}_g(w_0x) \leq C'
%\frac{\log\log x}{\log x}\left(1+H_1(x) + H_2(x;T)\right) 
\frac{\log\log x}{\log x} \max\{H_1(x), H_2'(x;T)\},
%(H_1(x) + H_2(x;T)),
$$
or else
$$
\tilde{M}_g(w_0x) \ll \frac{1}{\log x} \exp\left(C(\log\log x)^{2/3}\right).
$$
Thus, up to replacing $C$ by a larger constant if necessary, in the first case \eqref{eq:cases} holds; the second case will be dealt with separately.
\subsection{Small values of $\tilde{M}_g(w_0x)$}
In light of the above discussion, our first goal is to establish precise conditions under which \eqref{eq:cases} occurs. We recover the following.
\begin{lem} \label{lem:casesAB}
Fix $\e \in (0,1)$ and let $x$ be large. Suppose that $f\in \mc{F}$ satisfies \eqref{eq:fplus1} with some $\delta \in (0,1)$ and some $v \geq v_0(\e)$. Set $w := v/(1-\delta)$ and define
\begin{equation}\label{eq:Wdef}
W := w^w(\log\log x)^{\tfrac{1}{1-\delta}}.
%\exp\left(w\log w + \frac{1}{1-\delta}\log\log\log x\right).
\end{equation}
%let $W$ be as in \eqref{eq:Wdef}.
(a) If case (a) holds in \eqref{eq:cases} then there is a $|t_0| \leq 1/2$ such that 
$$
\sum_{\ss{p \leq x \\ f(p) \geq -\delta}} \frac{1-\cos(t_0 \log p)}{p} \ll \log W.
%\frac{1}{1-\delta}\left(v\log w + \log\log\log x\right).
$$
(b) Suppose that $f$ satisfies (b) above. Then there is $k = \exp(O(\sqrt{W}))$ and $t_k \in [k-1/2,k+1/2]$ such that
$$
\sum_{\ss{p \leq x \\ f(p) \geq -\delta}} \frac{1-\cos(t_k \log p)}{p} \ll \log W.
%w \log w + \frac{1}{1-\delta} \log\log\log x.
$$
In both of cases (a) and (b), we have
\begin{equation}\label{eq:fpdelta}
\sum_{\ss{p \leq x \\ f(p) \geq -\delta}} \frac{1}{p} \ll (\log\log x)^{2/3}(\log W)^{1/3}.
\end{equation}
\end{lem}
\begin{proof}
(a) Note that the LHS of \eqref{eq:fplus1} increases as we increase $v$, so we may assume that $v \geq \tfrac{40000}{\e^2}$. Applying Lemma \ref{lem:KeKl} 
%with $\delta = 1/2$, say, noting that as $f(p) = %\pm 1$,
%$$
%\mc{P}_{\delta} = \{p : f(p) \neq -1\} = \{p : %f(p) = +1\},
%$$
we obtain that
$$
\tilde{M}_g(w_0x) \geq w_0^{-1} \tilde{M}_g(x) \gg \e^2 w^{-(1+o(1))v/e} \exp\left(\sum_{p \leq x} \frac{f(p)}{p}\right).
$$
Combining this lower bound with case (a), we find that when $x$ is sufficiently large (and as $\e$ is fixed)
\begin{align*}
\exp\left(\sum_{p \leq x} \frac{f(p)}{p}\right) \ll \frac{w^{v}\log\log x}{\log x} \max_{|t| \leq 1/2} \exp\left(\sum_{p \leq x} \frac{(f(p)-1)\cos(t\log p)}{p}\right),
\end{align*}
or equivalently, after applying Mertens' theorem,
%again and replacing $C'$ by a slightly larger constant,
\begin{equation}\label{eq:Casea}
w^{-v}\frac{(\log x)^2}{\log\log x} \ll  \max_{|t| \leq 1/2} \exp\left(\sum_{p \leq x} \frac{(1-f(p))(1-\cos(t\log p))}{p}\right).
\end{equation}
In the sequel, define $Z := w^v (\log\log x)$. Suppose the RHS of \eqref{eq:Casea} is maximised at some $|t_0| \leq 1/2$. We show first that 
\begin{equation} \label{eq:t0LowBd}
|t_0| \gg 1/\sqrt{Z}, 
\end{equation}
a bound that we will use shortly. Indeed, observe that
\begin{align*}
\sum_{p \leq e^{1/|t_0|}} \frac{1-\cos(t_0 \log p)}{p} \ll t_0^2 \sum_{p \leq e^{1/|t_0|}} \frac{(\log p)^2}{p} \ll 1
\end{align*}
and as $|t_0| \leq 1/2$, by Lemma \ref{lem:TenEst} we have
$$
\sum_{e^{1/|t_0|} < p \leq x} \frac{1-\cos(t_0 \log p)}{p} = \log(|t_0|\log x) + O(1).
$$
It therefore follows from \eqref{eq:Casea} that
\begin{align*}
2\log\log x - \log Z + O(1) \leq \sum_{p \leq x} \frac{(1-f(p))(1-\cos(t_0\log p))}{p} &\leq 2 \sum_{p \leq x} \frac{1-\cos(t_0 \log p)}{p} \\
&\leq 2\log(|t_0| \log x) + O(1),
\end{align*}
from which we obtain $\log(1/|t_0|) \leq \frac{1}{2} \log Z + O(1)$, and \eqref{eq:t0LowBd} follows. \\
Now, by Lemma \ref{lem:uppBdcosSum} (or on refining the reasoning of the previous paragraph) we have that
\begin{align*}
\sum_{p \leq x} \frac{(1-f(p))(1-\cos(t_0\log p))}{p} &\leq 2\sum_{\ss{p \leq x \\ f(p) < -\delta}} \frac{1-\cos(t_0\log p)}{p} + (1+\delta) \sum_{\ss{p \leq x \\ f(p) \geq -\delta}} \frac{1-\cos(t_0\log p)}{p} \\
&\leq 2 \log\log x - (1-\delta)\sum_{\ss{p \leq x \\ f(p) \geq -\delta}} \frac{1-\cos(t_0\log p)}{p} + O(1).
\end{align*}
In view of \eqref{eq:Casea}, we obtain
$$
\sum_{\ss{p \leq x \\ f(p) = +1}} \frac{1-\cos(t_0\log p)}{p} \leq \frac{\log Z + O(1)}{1-\delta} \ll \log W,
$$
as claimed.\\
(b) Next, suppose $f$ satisfies case (b). As in the previous part, we have that
$$
\exp\left(\sum_{p \leq x} \frac{f(p)}{p} \right) \leq w^{v} \frac{\log\log x}{\log x} \left(\sum_{1 \leq k \leq T} \frac{(\log(2k))^6}{k^2} \max_{|t-k| \leq 1/2} \exp\left(2\sum_{y_k < p \leq x} \frac{f(p) \cos(t\log p)}{p}\right)\right)^{1/2},
$$
which implies on rearranging and using
$$
\exp\left(\sum_{p \leq y_k} \frac{f(p)}{p}\right) \ll (\log (2k))^2,
$$
that
$$
\frac{\log x}{Z} \ll \left(\sum_{1 \leq k \leq T} \frac{(\log(2k))^{10}}{k^2} \max_{|t-k| \leq 1/2} \exp\left(-2\sum_{y_k < p \leq x} \frac{f(p)(1-\cos(t\log p))}{p}\right)\right)^{1/2}.
$$
Taking the maximum and summing over $k$, we deduce that there is $1 \leq k_0 \leq T$ and $t_{k_0} \in [k_0-1/2,k_0+1/2]$ such that
$$
\frac{\log x}{Z} \ll \exp\left(-\sum_{y_{k_0} < p \leq x} \frac{f(p)(1-\cos(t_{k_0}\log p))}{p}\right).
$$
By Lemma \ref{lem:TenEst}, we have
$$
-\sum_{y_{k_0} < p \leq x} \frac{f(p)(1-\cos(t_{k_0}\log p))}{p} \leq \sum_{y_{k_0} < p \leq x} \frac{1-\cos(t_{k_0} \log p)}{p} = \log\log x - \log\log y_{k_0} + O(1),
$$
so in light of the previous estimate, we obtain that 
$$
\log(2k_0)^2 = \log y_{k_0} \ll Z \ll W,
$$
and thus $k_0 = \exp(O(\sqrt{W}))$ as claimed. \\
Next, arguing more precisely,
\begin{align*}
-\sum_{y_{k_0} < p \leq x} \frac{f(p)(1-\cos(t_{k_0}\log p))}{p} &\leq \sum_{\ss{ p \leq x \\ f(p) < -\delta}} \frac{1-\cos(t_{k_0} \log p)}{p} + \delta \sum_{\ss{ p \leq x \\ f(p) \geq -\delta}} \frac{1-\cos(t_{k_0}\log p)}{p} \\
&= \log\log x - (1-\delta)\sum_{\ss{p \leq x \\ f(p) \geq -\delta}} \frac{1-\cos(t_{k_0} \log p)}{p} + O(1),
\end{align*}
so we conclude as before that
$$
\sum_{\ss{p \leq x \\ f(p) \geq -\delta}} \frac{1-\cos(t_{k_0} \log p)}{p} \leq \frac{\log Z + O(1)}{1-\delta} \ll \log W.
$$
Finally, we prove \eqref{eq:disttoLiou}, setting $t = t_0$ in case (a) of \eqref{eq:cases} and $t = t_{k_0}$ in case (b) of \eqref{eq:cases}. Since 
$$
1-\cos(t\log p) = 2\sin^2(\pi(\tfrac{t}{2\pi}\log p)) \geq 8\left\|\frac{t}{2\pi}\log p\right\|^2,
$$
we deduce that for any $\theta \in (0,1)$,
\begin{align*}
\sum_{\ss{p \leq x \\ f(p) \geq -\delta \\ \|\tfrac{t}{2\pi}\log p\| \geq \theta}} \frac{1}{p} \leq \theta^{-2} \sum_{\ss{p \leq x \\ f(p) \geq-\delta}} \frac{\|\tfrac{t}{2\pi}\log p\|^2}{p} \leq \frac{1}{8\theta^2} \sum_{\ss{p \leq x \\ f(p) \geq-\delta}} \frac{1-\cos(t\log p)}{p} \ll \theta^{-2} \log W.
\end{align*}
Next, we control the contribution from $\|t(\log p)/2\pi\| \leq \theta$. First, when $k = 0$ we use \eqref{eq:t0LowBd} and Lemma \ref{lem:TenEst}, taking $\phi(u) := 1_{[0,\theta]}(\|\tfrac{u}{2\pi}\|)$, to obtain
$$
\sum_{\ss{p \leq x \\ f(p) \geq -\delta \\ \|\tfrac{t}{2\pi}\log p\| \leq \theta}} \frac{1}{p} \leq \log(1/|t|) + \sum_{\ss{e^{1/|t|} < p \leq x \\ \|\frac{t}{2\pi} \log p\| \leq \theta}} \frac{1}{p} + O(1) \leq 2 \theta \log\log x + \frac{1}{2} \log Z  +O(1). 
$$
In the case that $1 \leq k \leq T$, we get instead
\begin{align*}
\sum_{\ss{p \leq x \\ f(p) \geq -\delta \\ \|\tfrac{t}{2\pi}\log p\| \leq \theta}} \frac{1}{p} \leq \sum_{\ss{y_k < p \leq x \\ \|\frac{t}{2\pi} \log p\| \leq \theta}} \frac{1}{p} + \log \log y_k &\leq 2 \theta \log\left(\frac{\log x}{\log y_k}\right) + \log\log\log x  +O(1) \\
&\leq 2\theta \log\log x + \log\log\log x + O(1).
\end{align*}
Combining these results and recalling that $\log\log x \leq Z \leq W$, we obtain, regardless of $k$,
$$
\sum_{\ss{p \leq x \\ f(p) \geq -\delta}} \frac{1}{p} \ll \theta^{-2} \log W + \theta \log\log x.
$$
Setting $\theta := \left(\tfrac{\log W}{\log\log x}\right)^{1/3}$, we deduce the claim.
\end{proof}
We will use Lemma \ref{lem:casesAB} to give upper bounds for $\mb{D}(f,\lambda;x)^2$, as follows.
\begin{lem} \label{lem:LiouDist}
Assume the hypotheses of Lemma \ref{lem:casesAB}.
\begin{enumerate}[(i)]
\item If $f\in \mc{F}_{\pm}$ then
%Moreover, we have that
\begin{equation}\label{eq:disttoLiou}
\mb{D}(f,\lambda; x)^2 \ll (\log\log x)^{2/3}(v\log(v\log\log x))^{1/3}.
%w\log w + \frac{1}{1-\delta}\log\log\log x\right)^{1/3}.
\end{equation}
\item More generally, if $f \in \mc{F}$ then
\begin{equation} \label{eq:disttoLiouGen}
\mb{D}(f,\lambda;x)^2 \ll (\log\log x)^{3/4}(v \log(v\log\log x))^{1/4}.
\end{equation}
\end{enumerate}
\end{lem}
\begin{proof}
(i) Taking any $\delta \in (0,1)$, we see that $f(p) \geq -\delta$ if and only if $f(p) = +1$. Replacing $\delta$ by $1/2$, say,
\begin{align*}
\mb{D}(f,\lambda;x)^2 = 2\sum_{\ss{p \leq x \\ f(p) = +1}} \frac{1}{p} = 2\sum_{\ss{p \leq x \\ f(p) \geq -\delta}} \frac{1}{p} &\ll (\log\log x)^{2/3} (\log W)^{1/3} \\
&\ll (\log\log x)^{2/3}(v \log v + \log\log\log x)^{1/3},
\end{align*}
as claimed. \\
(ii) 
% Next, for each $\eta \in (0,1)$ let us define
% $$
% v(\delta) := \text{argmin}\left\{v \geq v_0(\e) : \sum_{\ss{x^{1/v} < p \leq x \\ f(p) \geq -\eta}} \frac{1}{p} \geq 1+\e\right\}.
% $$
% If $\eta_2 \geq \eta_1$ then $f(p) \geq -\eta_1 \geq -\eta_2$, and thus $v(\eta_1) \geq v(\eta_2)$. We use this momentarily. \\
Observe that for any $\eta > 0$ we have
$$
\mb{D}(f,\lambda;x)^2 \leq 2\sum_{\ss{p \leq x \\ f(p) \geq -\eta}} \frac{1}{p} + (1-\eta)\log\log x + O(1).
$$
Now, by hypothesis, for any $\eta \in [\delta,1)$ we have
% $$
% \sum_{\ss{x^{1/v} < p \leq x \\ f(p) \geq -\delta}} \frac{1}{p} \geq 1+\e,
% $$
% so that for any $\eta \in [\delta,1)$ we also have
$$
\sum_{\ss{x^{1/v} < p \leq x \\ f(p) \geq -\eta}} \frac{1}{p} \geq \sum_{\ss{x^{1/v} < p \leq x \\ f(p) \geq -\delta}} \frac{1}{p} \geq 1+\e.
$$
Let $(1-\delta)^{-1} \leq R\leq \log\log x$ be a parameter to be chosen, and let $\eta := 1-1/R$. Applying Lemma \ref{lem:casesAB} and recalling the definition of $W = W(\eta)$, we get that
$$
\sum_{\ss{p \leq x \\ f(p) \geq -\eta}} \frac{1}{p} \ll (\log\log x)^{2/3} (Rv \log (Rv\log\log x))^{1/3},
$$
and therefore as $R \leq \log\log x$,
$$
\mb{D}(f,\lambda;x)^2 \ll R^{1/3}(\log\log x)^{2/3} (v \log(v \log\log x))^{1/3} + \frac{\log\log x}{R}.
$$
If we select
$$
R := \left(\frac{\log\log x}{v\log(v\log\log x)}\right)^{1/4}
$$
then we obtain the claimed bound
$$
\mb{D}(f,\lambda;x)^2 \ll (\log\log x)^{3/4}(v\log(v\log\log x))^{1/4},
$$
as claimed.
\end{proof}
We are now ready to deduce both of our propositions for this section. 
\begin{proof}[Proof of Proposition \ref{prop:improveNeg}]
If $L_f(x) \geq 0$ then the claim holds trivially. Otherwise, assume that $L_f(x) < 0$. Taking $T = (\log x)^2$ and combining Proposition \ref{prop:passToSmooth} with Lemmas \ref{lem:monoBdd} and \ref{lem:Limit}, we have
\begin{equation}\label{eq:MgtoH1H2}
\tilde{M}_g(w_0x) \ll \frac{1}{\log x} \int_1^x (H_1(y) + H_2'(y;T)) \frac{dy}{y\log y} + \frac{\log\log x}{(\log x)^2} \ll \frac{\log\log x}{\log x} \max\{H_1(x), H_2'(x;T)\}.
\end{equation}
Next, let $|k| \leq T$. 
%and set $a := 1$ if $k =0$ and $a:= 0$ otherwise. 
Given that $1+f(p) \geq 0$ for all $p$, if we apply Lemma \ref{lem:TenEst} then, writing $y_0 := 1$,
\begin{align*}
&\max_{|t-k| \leq 1/2} \sum_{y_k < p \leq x} \frac{(f(p)-1_{k = 0}) \cos(t\log p)}{p} \\
&= \max_{|t-k| \leq 1/2} \sum_{y_k < p \leq x} \left(\frac{(f(p)+1)\cos(t\log p)}{p} - \frac{(1+1_{k = 0})\cos(t\log p)}{p}\right) + O(1) \\
&\leq \mb{D}(f,\lambda;x)^2 - (1+1_{k = 0})\min_{|t-k| \leq 1/2} \sum_{\max\{y_k, e^{1/|t|}\} < p \leq x} \frac{\cos(t\log p)}{p} \\
&= \mb{D}(f,\lambda;x)^2 + O(1).
\end{align*}
It follows that
\begin{align}\label{eq:H12Dlambda}
\max \{H_1(x), H_2'(x;T)\} \ll \exp(\mb{D}(f,\lambda;x)^2).
%\ll \exp\left((\log\log x)^{2/3} (\log W)^{1/3}\right).
\end{align}
% Let $\tau$ be defined as in the statement of Proposition \ref{prop:improveNeg}.
% By Lemma \ref{lem:casesAB}, we have
% $$
% \mb{D}(f,\lambda;x)^2 \ll (\log\log x)^{1-\tau}(v\log (v\log\log x))^{\tau}.
% $$
We therefore deduce using $\tilde{M}_g(w_0x) \geq 0$ that for some $c_1 > 0$,
$$
L_f(x) \geq - \frac{c_1 \log\log x}{\log x} \exp(\mb{D}(f,\lambda;x)^2),
$$
Finally, let $\tau$ be defined as in the statement of Proposition \ref{prop:improveNeg}.
By Lemma \ref{lem:casesAB}, we have that for some $c_2 > 0$,
$$
\mb{D}(f,\lambda;x)^2 \leq c_2(\log\log x)^{1-\tau}(v\log (v\log\log x))^{\tau}.
$$
The claimed bound now follows.
%depending on whether or not $f \in \mc{F}_{\pm}$.
\end{proof}
\begin{proof}[Proof of Proposition \ref{prop:convLf}]
Since by Mertens' theorem we trivially have
$$
\mb{D}(f,\lambda;x)^2 \leq 2\log\log x + O(1),
$$ 
the claim is trivial if $v \log(v\log\log x) > C' \log\log x$ for $C' > 0$ sufficiently large. Thus, in the sequel we shall assume that $v\log(v\log\log x) \leq C' \log\log x$ for some fixed $C' > 0$. \\
Applying Proposition \ref{prop:passToSmooth} as in the proof of Proposition \ref{prop:improveNeg}, we get
\begin{align*}
\tilde{M}_g(w_0x) &\ll |L_f(x)| + \frac{\log\log x}{\log x} (H_1(x) + H_2'(x;T)) \\
&\ll \frac{\exp\left(C(\log\log x)^{2/3}\right)}{\log x} + \frac{\log\log x}{\log x} \max\{H_1(x), H_2'(x;T)\}.
\end{align*}
Suppose first of all that
\begin{equation}\label{eq:firstposs}
\tilde{M}_g(w_0x) \ll \frac{\exp(C(\log\log x)^{2/3})}{\log x}.
\end{equation}
Applying Lemma \ref{lem:KeKl} and rearranging, we get
$$
\exp(\mb{D}(f,\lambda;x)^2) \asymp (\log x) \exp\left(\sum_{p \leq x} \frac{f(p)}{p}\right) \ll \exp\left( v \log w + C(\log\log x)^{2/3}\right),
$$
so on taking logarithms and using $1 \leq v \log w \leq C'\log\log x$, we get
$$
\mb{D}(f,\lambda;x)^2 \leq C(\log\log x)^{2/3} + v\log w \ll (\log\log x)^{2/3}(v\log w)^{1/3},
$$
which is stronger than the claimed bound. \\
If \eqref{eq:firstposs} fails then \eqref{eq:MgtoH1H2} must hold, and thus the proof follows in the same way as for Proposition \ref{prop:improveNeg}.
\end{proof}

%Both of Propositions \ref{prop:improveNeg} and \ref{prop:convLf} follow immediately from Lemma \ref{lem:casesAB}.
\subsection{Optimality of Lemma \ref{lem:casesAB}} \label{sec:Optim}
Finally, we show here that Lemma \ref{lem:casesAB} is close to optimal. The following result implies Corollary \ref{cor:const}.
\begin{lem}
There exists $f \in \mc{F}_{\pm}$
%$\pm 1$-valued completely multiplicative function $f$ 
such that \eqref{eq:fplus1} holds with $\e > 0$ fixed and $v = O(1)$, and for which the following bounds hold with any $C > 4$:
\begin{align*}
\mb{D}(f,\lambda;x)^2 &\asymp (\log\log x)^{2/3} \text{ and } |L_f(x)| \ll \frac{\exp(C(\log\log x)^{2/3})}{\log x}.
% \tilde{M}_g(w_0x) &\asymp \frac{\exp(4(\log\log x)^{2/3})}{\log x}, \\
% H_1(x) + H_2(x;T) &\ll \exp(4(\log\log x)^{2/3}).
\end{align*}
Furthermore, we have
$$
L_f(x) = \left(1+O\left(\frac{1}{\log x}\right)\right)\left(-\int_1^{x^{1/v}} \frac{|L_f(t)|}{t\log(x/t)} dt\right)
%\left(-\sum_{x^{1/(3e)} < p \leq x} \frac{|L_f(x/p)|}{p}\right) 
+ O\left(\frac{\exp(4(\log\log x)^{2/3})}{\log x}\right).
$$
\end{lem}
% \begin{rem}
% Though we do not directly use the final estimate in the statement, we found it compelling to obtain such a formula with a negative ``main term''. This is perhaps misleading as in our construction it is possible that the error term is larger than this term, but we were unable to verify this. Note that this is not immediate, since if 
% $$
% |L_f(t)| \asymp \frac{\exp(4(\log\log t)^{2/3})}{\log t}
% $$
% then 
% %(an indeed, if $|L_f(t)|$ were known to grow a bit faster than 
% \end{rem}
\begin{proof}
Let $x$ be large, fix $\e := \log(3/2) > 0$ and set $v := 3e$, $t_0 := 1/2$. Define also $\theta := (\log\log x)^{-1/3}$. Define a completely multiplicative function $f$ at primes by
\begin{equation}
f(p) := \begin{cases} 
+1 &\text{ if } p \leq x^{1/v} , \, \|\tfrac{t_0}{2\pi} \log p\| \leq \theta \\ 
-1 &\text{ if } p \leq x^{1/v} , \, \|\tfrac{t_0}{2\pi} \log p\| > \theta \\
+1 &\text{ if } x^{1/v} < p \leq x^{1-1/v} \\
-\text{sign}(L_f(x/p)) &\text{ if } x^{1-1/v} < p \leq x \\
-1 &\text{ if } p > x
\end{cases}
\end{equation}
(the values at $p > x$ are unimportant for our purposes). Note that the choices of $f(p)$ for $x^{1-1/v} < p \leq x$ are well-defined, since whenever $p$ lies in this range, $L_f(x/p)$ depends only on the values of $f(p')$ with $p' \leq x^{1/v} < x^{1-1/v}$. \\ 
Applying Lemma \ref{lem:TenEst} with $h(u) := 1_{[0,\theta]}(\|\tfrac{u}{2\pi}\|)$,
$$
\mb{D}(f,\lambda;x)^2 = \sum_{p \leq x} \frac{1+f(p)}{p} = 2\sum_{\ss{p \leq x \\ \|\tfrac{1}{4\pi} \log p\| \leq \theta}} \frac{1}{p} + O(1) = 4 \theta \log\log x + O(1) = 4 (\log\log x)^{2/3} + O(1).
$$
% Moreover, since $1-\cos x  = 2\pi^2\|x/2\pi\|^2 + O(\|x/2\pi\|^4)$, we have
% \begin{align*}
% \sum_{\ss{p\leq x \\ f(p) = +1}} \frac{1-\cos(\tfrac{1}{2}\log p)}{p} &= \sum_{\ss{p \leq x \\ \|\tfrac{1}{4\pi}\log p\| \leq \delta}} \frac{1-\cos(\tfrac{1}{2}\log p)}{p} =  4\pi^2 \delta^2 \sum_{\ss{p \leq x \\ \|\tfrac{1}{4\pi} \log p\| \leq \delta}} \frac{1}{p} + O(\delta^4 \log\log x) \\
% &= 2\pi^2 \delta^3(1+O(\delta)) \log\log x \asymp 1,
% \end{align*}
% as required.
and therefore also
$$
\sum_{p \leq x} \frac{f(p)}{p} = -\log\log x + \mb{D}(f,\lambda;x)^2 + O(1) = -\log\log x + 4 (\log\log x)^{2/3} + O(1).
$$
Aiming to apply Proposition \ref{prop:passToSmooth}, we recall the notation $f_{1-1/v}(n) := f(n)1_{P^+(n) \leq x^{1-1/v}}$ and $g_{1-1/v} := 1\ast f_{1-1/v}$. We observe that
$$
\sum_{\ss{x^{1/v} < p \leq x \\ f(p) = +1}} \frac{1}{p} \geq \sum_{x^{1/v} < p \leq x^{1-1/v}} \frac{1}{p} \geq \log(3e/2) \geq 1 + \e.
$$
Applying Lemma \ref{lem:KeKl} for the lower bound, and Lemma \ref{lem:UppBdNonNeg}(a) for the upper bound,
$$
\tilde{M}_{g}(w_0x) \asymp \tilde{M}_{g_{1-1/v}}(w_0x) \asymp \exp\left(\sum_{p \leq x} \frac{f(p)}{p}\right)  \asymp \frac{\exp(4(\log\log x)^{2/3})}{\log x}.
$$
Next, we give upper bounds for $H_1(x)$ and $H_2'(x;T)$, taking $T = (\log x)^2$. As in \eqref{eq:H12Dlambda} we have
$$
\max\{H_1(x), H_2'(x;T)\} \ll \exp(\mb{D}(f,\lambda;x)^2) \ll \exp\left(4(\log\log x)^{2/3}\right).
$$
% Applying Lemma \ref{lem:TenEst} and using the fact that $\cos(t\log p) \geq 0$ for $p \leq e^{1/|t|}$ and $|t| \leq 1/2$,
% \begin{align*}
% H_1(x) &\ll \max_{|t| \leq 1/2} \exp\left(-2\sum_{\ss{p \leq x \\ f(p) = -1}} \frac{\cos(t\log p)}{p}\right) \leq \max_{|t| \leq 1/2} \exp\left(-2\sum_{p \leq x} \frac{\cos(t\log p)}{p} + 2\sum_{\ss{p \leq x \\ f(p) = +1}} \frac{\cos(t \log p)}{p}\right) \\
% &\ll \max_{|t| \leq 1/2} \exp\left(-2\sum_{e^{1/|t|} < p \leq x} \frac{\cos(t\log p)}{p} + 2\sum_{\ss{p \leq x \\ f(p) = +1}} \frac{1}{p}\right) \ll \exp(\mb{D}(f,\lambda;x)^2) \\
% &\ll \exp\left(4(\log\log x)^{2/3}\right).
% \end{align*}
% Similarly, 
% %setting $T = (\log x)^2$, 
% %we can bound 
% \begin{align*}
% H_2'(x;T) &\ll \max_{1 \leq k \leq T} \max_{|t-k| \leq 1/2} \exp\left(-\sum_{y_k < p \leq x} \frac{\cos(t\log p)}{p} + 2\sum_{\ss{y_k < p \leq x \\ f(p) = +1}} \frac{\cos(t\log p)}{p}\right) \\
% &\ll \exp\left(2\sum_{\ss{y_k < p \leq x \\ f(p) = +1}} \frac{1}{p}\right) \ll \exp\left(4(\log\log x)^{2/3}\right).
% \end{align*}
It follows from Proposition \ref{prop:passToSmooth} (with $\theta = 0$) that
$$
|L_f(x)| \ll \tilde{M}_g(w_0x) + \frac{\log\log x}{\log x} (H_1(x) + H_2'(x;T)) \ll \frac{\log\log x}{\log x} \exp(4(\log\log x)^{2/3}) \ll \frac{\exp(C(\log\log x)^{2/3})}{\log x},
$$
for all $C > 4$. This proves the first claim.\\
We next prove the more precise asymptotic claim. Note that
$$
\sum_{x^{1-1/v} < p \leq x} \frac{f(p)}{p} L_f(x/p) = -\sum_{x^{1-1/v} < p \leq x} \frac{|L_f(x/p)|}{p},
$$
which is genuinely negative. 
We now apply Proposition \ref{prop:passToSmooth} (with $\theta = 1-1/v$ in the notation there). In light of the above bounds for $H_1(x)$, $H_2(x;T)$ and $\tilde{M}_{g_{1-1/v}}(w_0x)$, we obtain
\begin{align*}
L_f(x) = \left(1 + O\left(\frac{1}{\log x}\right)\right) \left(-\sum_{x^{1-1/v} < p \leq x} \frac{|L_f(x/p)|}{p}\right) + O\left(\frac{\exp\left(4(\log\log x)^{2/3}\right)}{\log x}\right).
\end{align*}
We next analyse the sum over $p$ here.  We have
\begin{align*}
\sum_{x^{1-1/v} < p \leq x} \frac{|L_f(x/p)|}{p} &= \sum_{d \leq x^{1/v}} |L_f(d)| \sum_{x/(d+1) < p \leq x/d} \frac{1}{p} \\
&= \sum_{d \leq x^{1/v}} |L_f(d)|\cdot \frac{d}{x} \left(1 + O\left(\frac{1}{d}\right)\right) \left(\pi(x/d) - \pi(x/(d+1))\right).
\end{align*}
Note that if $d \leq x^{1/v}$ then the length of the interval $(x/(d+1),x/d]$ is 
$$
\frac{x}{d(d+1)} \geq x^{-1/v} \frac{x}{d+1} \geq \left(\frac{x}{d+1}\right)^{1-1/(v-1)},
$$
so as $v = 3e \geq 7$, we get e.g. by Ingham's prime number theorem in short intervals \cite{Ing} that
$$
\pi(x/d) - \pi(x/(d+1)) = \left(1+O\left(\frac{1}{\log x}\right)\right)\frac{x}{d(d+1) \log(x/(d+1))},
$$
for all $d \leq x^{1/v}$. Thus, we find
\begin{align*}
\sum_{x^{1-1/v} < p \leq x} \frac{|L_f(x/p)|}{p} &= \left(1+O\left(\frac{1}{\log x}\right)\right) \sum_{d \leq x^{1/v}} \frac{|L_f(d)|}{(d+1) \log(x/(d+1))} \left(1+O\left(\frac{1}{d}\right)\right) \\
&= \left(1+O\left(\frac{1}{\log x}\right)\right) \sum_{d \leq x^{1/v}} \frac{|L_f(d)|}{(d+1) \log(x/(d+1))} + O\left(\frac{1}{\log x}\right),
\end{align*}
using the bound $|L_f(d)| \leq \log d + O(1)$ to estimate the error term in $d$. \\
Comparing to an integral, we have
\begin{align*}
\sum_{d \leq x^{1/v}} \frac{|L_f(d)|}{(d+1)\log(x/(d+1))} = \int_1^{x^{1/v}} \frac{|L_f(t)|}{t \log(x/t)}\left(1+O\left(\frac{1}{t}\right)\right) dt = \int_1^{x^{1/v}} \frac{|L_f(t)|}{t \log(x/t)} dt + O\left(\frac{1}{\log x}\right).
\end{align*}
To conclude, we obtain the expression
\begin{align*}
L_f(x) = \left(1+O\left(\frac{1}{\log x}\right)\right) \left(-\int_1^{x^{1/v}} \frac{|L_f(t)|}{t\log(x/t)} dt\right) + O\left(\frac{\exp(4(\log\log x)^{2/3}}{\log x}\right),
\end{align*}
as claimed.
\end{proof}

% \begin{lem}
% Let $f: \mb{N} \ra \{-1,+1\}$ be completely multiplicative. Assume that
% $$
% |L_f(x)| \ll \frac{\log\log x}{\log x}.
% $$
% Then for all $v \leq [...]$ we have
% $$
% \sum_{\ss{x^{1/v} \leq p \leq x \\ f(p) = +1}} \frac{1}{p} \leq 1 + o(1).
% $$
% \end{lem}
% \begin{proof}
% Set $V := [...].$ Assume for the sake of contradiction that there is $\e > 0$ and a $40000/\e^2 \leq v \leq V$ such that
% $$
% \sum_{\ss{x^{1/v} \leq p \leq x \\ f(p) = +1}} \frac{1}{p} \geq 1+\e.
% $$
% Clearly, $\{p \in \mb{P} : f(p) = +1\} = \mc{P}_{\delta}$ for $\delta = \tfrac{1}{\log x}$. Applying the previous lemma, we obtain
% $$
% \tilde{M}_g(w_0x) \asymp \tilde{M}_g(x) \gg \e^2 v^{-v(1+o(1))/e} \exp\left(\sum_{p \leq x} \frac{f(p)}{p}\right) \asymp \e^2 v^{-v(1+o(1))/e} (\log x) \exp\left(-\sum_{p \leq x} \frac{1-f(p)}{p}\right)
% $$
% Now, by Theorem \ref{thm:HalDelta} with $T = \log x$, Lemma \ref{lem:Limit} and our assumption, we have
% $$
% \tilde{M}_g(w_0x) \ll |L_f(x)| + \frac{\log\log x}{\log x}(1+e^{-M(x;T)}) \ll \frac{\log\log x}{\log x} e^{-M(x;T)} = \frac{\log\log x}{\log x} \exp\left(-\sum_{p \leq x} \frac{(1-f(p))\cos(t_0\log p)}{p}\right),
% $$
% for some $|t_0| \leq \log x$. Combining these bounds together, we obtain
% $$
% v^{(1+o(1))v/e} \exp\left(\sum_{p \leq x} \frac{(1-f(p))(1-\cos(t_0\log p))}{p}\right) \gg \e^2 \frac{(\log x)^2}{\log\log x}. 
% $$
% %for some $|t_0| \leq \log x$. 
%\end{proof}

\bibliographystyle{plain}
\bibliography{NegTrunc.bib}

\end{document}